\documentclass[onefignum,onetabnum]{siamonline220329}

\usepackage{lipsum}
\usepackage{colortbl}
\usepackage{amsfonts}
\usepackage{graphicx}
\usepackage{epstopdf}
\usepackage{algorithmic}
\usepackage{cite}
\usepackage{soul}
\usepackage{amsmath,amssymb,amsfonts}
\usepackage{pgf,tikz}
\usetikzlibrary{shapes,shapes.arrows,decorations.pathreplacing,
    calc}
\usepackage{array}
\usepackage[thinlines]{easytable}
\usepackage{textcomp}
\usepackage{multirow}
\PassOptionsToPackage{table,xcdraw}{xcolor}
\usepackage[algo2e]{algorithm2e} 
\usepackage{enumitem}
\usepackage{comment}
\usepackage{leftidx}
\usepackage{float}
\usepackage[position=top]{subfig}
\usepackage{setspace}
\ifpdf
  \DeclareGraphicsExtensions{.eps,.pdf,.png,.jpg}
\else
  \DeclareGraphicsExtensions{.eps}
\fi

\usepackage{enumitem}
\setlist[enumerate]{leftmargin=.5in}
\setlist[itemize]{leftmargin=.5in}

\newsiamremark{remark}{Remark}
\newsiamremark{hypothesis}{Hypothesis}
\crefname{hypothesis}{Hypothesis}{Hypotheses}
\newsiamthm{claim}{Claim}

\headers{Inexact Multilevel FISTA for Image Restoration}{G. Lauga, E. Riccietti, N. Pustelnik and P. Gon\c calves}

\title{
IML FISTA: A Multilevel Framework for Inexact and Inertial Forward-Backward. Application to Image Restoration. 
\thanks{Submitted to the editors on DATE.
\funding{The authors would like to thank the GdR ISIS for funding the MOMIGS project and the ANR-19-CE48-0009 Multisc'In project.} We also gratefully acknowledge the support of the Centre Blaise Pascal's IT test platform at ENS de Lyon (Lyon, France) for the computing facilities. The platform operates the SIDUS \cite{quemener2013} solution developed by Emmanuel Quemener.}}

\author{Guillaume Lauga\thanks{Univ Lyon, Inria, EnsL, UCBL, CNRS, LIP, UMR 5668, F-69342, Lyon Cedex 07, France 
  (\email{guillaume.lauga@ens-lyon.fr}, 
  \email{elisa.riccietti@ens-lyon.fr},
  \email{paulo.goncalves@ens-lyon.fr}
  ).
  }
\and Elisa Riccietti\footnotemark[2]
\and Nelly Pustelnik\thanks{Ens de Lyon,
		CNRS, Laboratoire de Physique, F-69342, Lyon, France 
  (\email{nelly.pustelnik@ens-lyon.fr}).}
\and Paulo Gon\c calves\footnotemark[2]
}

\usepackage{amsopn}

\newcommand{\yhk}{y_{h,k}}
\newcommand{\yhkun}{y_{h,k+1}}

\newcommand{\xhk}{x_{h,k}}
\newcommand{\xhkun}{x_{h,k+1}}
\newcommand{\xHl}{s_{H,k,\ell}}
\newcommand{\xHlun}{s_{H,k,\ell+1}}
\newcommand{\xHO}{s_{H,k,0}}
\newcommand{\xHm}{s_{H,k,m}}

\newcommand{\taubarh}{\bar{\tau}_{h,k}}

\newcommand{\tauhk}{\tau_{h}}
\newcommand{\tauHl}{\tau_{H}}

\newcommand{\IhH}{I_h^H}
\newcommand{\IHh}{I_H^h}
\newcommand{\Lo}{L}
\newcommand{\R}{R}

\newcommand{\ehk}{e_{h,k}}

\newcommand{\vHk}{v_{H,k}}

\newcommand{\A}{\mathrm{A}}

\newcommand{\D}{\mathrm{D}}

\newcommand{\RR}{\mathbb{R}}
\newcommand{\Res}{\mathbf{R}}

\newcommand{\Id}{\mathrm{Id}}

\DeclareMathOperator*{\argmin}{arg\,min}
\DeclareMathOperator*{\Argmin}{Argmin}
\newcommand{\prox}{\normalfont \textrm{prox}}
\newcommand{\ftn}{\scriptsize}
\newcolumntype{L}[1]{>{\raggedright\let\newline\\\arraybackslash\hspace{0pt}}m{#1}}
\newcolumntype{C}[1]{>{\centering\let\newline\\\arraybackslash\hspace{0pt}}m{#1}}
\newcolumntype{R}[1]{>{\raggedleft\let\newline\\\arraybackslash\hspace{0pt}}m{#1}}

\ifpdf
\hypersetup{
  pdftitle={Inexact Multilevel FISTA for Image Restoration},
  pdfauthor={G. Lauga, E. Riccietti, N. Pustelnik and P. Gon\c calves}
}
\fi

\begin{document}

\maketitle

\begin{abstract}
This paper presents a multilevel framework for inertial and inexact proximal algorithms, that encompasses multilevel versions of classical algorithms such as forward-backward and FISTA.  The methods are supported by strong theoretical guarantees: we prove both the rate of convergence and the convergence of the iterates to a minimum in the convex case, an important result for ill-posed problems. 
We propose a particular instance of IML (Inexact MultiLevel) FISTA, based on the use of the Moreau envelope to build efficient and useful coarse corrections, fully adapted to solve problems in image restoration. Such a construction is derived for a broad class of composite optimization problems with proximable functions. 
We evaluate our approach on several image reconstruction problems and we show that it considerably accelerates the convergence of the corresponding one-level (i.e. standard) version of the methods, for large-scale images.

\end{abstract}


\section{Introduction}
In the context of image restoration, we aim to recover a good quality image $\widehat x$ from a corrupted version $z = \A\bar{x}+\epsilon$ of an original image $\bar x$, where $\A$ models a linear degradation operator and $\epsilon$ stands for additive noise. This problem is known to be ill-posed, and is generally tackled by solving a regularized least squares
problem. This formulation involves a data-fidelity term $\Lo$ and a regularization term $\R$ that allows us to choose the properties one wishes to impose on the solution:
\begin{equation}
    \widehat{x} \in \Argmin_{x \in \RR^N} F(x):=\Lo(x) + \R(x),
    \label{eq:optim_article}
\end{equation}
where $\Lo:\RR^N \rightarrow (-\infty,+\infty]$ and $\R:\RR^N \rightarrow (-\infty,+\infty]$ belong to the class of convex, lower semi-continuous (l.s.c), and proper functions on $\RR^N$. Moreover, $\Lo$ is assumed to be  differentiable with $\beta$-Lipschitz gradient, while $\R$ is usually non-smooth. $F$ is supposed to be coercive. 

Many iterative algorithms have been proposed in the literature to estimate $\widehat{x}$ (cf. for instance \cite{Combettes_P_2011,parikh2014,chambolle2016,condat2021,combettes2021} and references therein). 
Most of them are based on the use of proximal methods, as $\R$ is non differentiable, and they all share the same weakness: the required computational time for the reconstruction turns prohibitive for large size problems. This is particularly critical when the proximity operator of $R$ cannot be computed explicitly \cite{villa2013,schmidt2011convergence,he2012accelerated}. It is the case when $\R$ is the sum of two functions \cite{chen2019}, when it encodes a total variation \cite{chambolle2009}, or a non-local total variation \cite{chierchia2014}. Indeed, for these two state-of-the-art regularizations, the proximity operator of $\R$ can be estimated by an iterative procedure in the dual domain (cf. \cite{beck2009}), which considerably increases the cost of the optimization. Many methods circumvent this dual optimization by directly introducing dual steps paired with primal steps to reach a minimizer \cite{boyd2010,chambolle2011,condat2013}, but their cost for large-scale problems remains high and they may still need to compute inexact proximity operators \cite{rasch2020inexact}. The already challenging task of designing algorithms that can handle large-scale problems turns even harder when inexact proximity operators are to be dealt with.

Various attempts have been made to accelerate the resolution of standard convex optimization problems, i.e., to reduce the number of necessary iterations to reach convergence \cite{boyd2010,chambolle2011,condat2013,beck2009,chambolle2015-1,fercoq2015,salim2020,chambolle2018,ehrhardt2018,chouzenoux2010,chouzenoux2022,chierchia2015epigraphical}. As convergence guarantees (e.g. to a minimzer) of the seminal forward-backward (FB) algorithm \cite{combettes2005} are paramount in the context of image restoration, these attempts are commonly constructed around the sequence generated by this algorithm (see for instance \cite{beck2009,aujol2015,chambolle2015-1,chouzenoux2014variable,chouzenoux2016block,combettes2021}). The $k$-th iteration of FB reads:
\begin{equation}\label{eq:fb} 
 (\forall k=0,1,\ldots)\qquad    x_{k+1}  =\prox_{\tau_k \R}(x_k - \tau_k \nabla \Lo(x_k)).
\end{equation}
where $0<\tau_k<2/\beta$. Among the most efficient methods to accelerate sequences obtained by \eqref{eq:fb}, the \textit{fast iterative soft thresholding algorithm} (FISTA) \cite{beck2009-1,chambolle2015-1}, is based on an inertial/Nesterov principle where an extrapolation step is built to improve at each iteration the forward-backward step. The $k-$th iteration of FISTA reads for every $k=0,1,\ldots$ as:
\begin{align}
    x_{k+1} & =\prox_{\tau_k \R}(y_k - \tau_k \nabla \Lo(y_k))\label{eq:fista_steps1} \\
    y_{k+1} & = x_{k+1} + \alpha_{k}(x_{k+1}-x_{k})
    \label{eq:fista_steps2}
\end{align}
where $0<\tau_k<1/\beta$ and $\alpha_k = \frac{t_{k}-1}{t_{k+1}}$. 
The sequence $\{t_k\}_{k \in \mathbb{N}}$ can be chosen in different ways, yielding different relaxations of the forward-backward sequence, but it must verify the general condition $t_{k}-t_{k+1}^2+t_{k+1}$ $\geq$ $0$ to guarantee convergence of the objective function to the optimal value. A common practice is to choose $t_0=1$ and $t_{k} = \left(\frac{k+a-1}{a}\right)^d \label{eq:tk_relaxed}$ for all $k \in \mathbb{N}^*$, with $d\in(0,1]$ and $a > \max\{1,(2d)^{\frac{1}{d}}\}$ \cite[Definition 3.1]{aujol2015}. This choice ensures convergence of the objective function with rate $o(1/k^{(2d)})$ and, under mild conditions \cite{chambolle2015-1,aujol2015}, weak convergence of the iterates to a minimizer. Here the parameter $d$ defines a continuous way to go from a standard FB to FISTA by steadily adding inertia. We will restrict ourselves to this specific choice in the following.

To go further with acceleration techniques, we aim to use the structure of these optimization problems to reduce both the number of iterations needed to converge and the computation time through some dimensionality reduction techniques. This is a recurring idea in a lot of algorithms proposed in the literature considering either stochastic block selection \cite{fercoq2015,salim2020,chambolle2018,ehrhardt2018} or subspace methods \cite{chouzenoux2010,chouzenoux2022}. 
Specifically here, we seek to combine inertial techniques with \textit{multilevel} approaches that exploit different resolutions of the same problem.  In such methods the objective function is approximated by a sequence of functions defined on reduced dimensional spaces (coarse scales) and  descent steps are calculated at coarse levels with smaller cost before being transfered back to the fine level. Our goal is to embed  such coarse correction into the descent step computed at fine level in \eqref{eq:fista_steps1} before computing the approximation of the proximity operator to benefit from both types of acceleration: inertial and multilevel.

Multilevel approaches have been mainly studied for the resolution of partial differential equations (PDEs), in which $\Lo$ and $\R$ are both supposed to be differentiable \cite{nash2000,calandra2020,gratton_TR}. Indeed, most of the multilevel algorithms are based on the seminal work of Nash \cite{nash2000} and are applied to smooth objective functions minimized by first order methods. They have been employed in many applications, such as photoacoustic tomography \cite{javaherian2017}, discrete tomography \cite{plier2021} and phase retrieval \cite{fung2020}. They have been also extended to higher order optimization in \cite{gratton_TR,ho2019,high_dim_ML}. 

Only recently this idea has been extended in \cite{parpas2016,parpas2017} to define multilevel FB algorithms applicable to problem \eqref{eq:optim_article} in the case where $\R$ is non differentiable but its proximity operator is known in closed form expression. In the experiments of \cite{parpas2017}, the framework is restricted to $\R = \Vert W\cdot \Vert_1$ with $W$ an orthogonal wavelet transform in the context of image restoration, and it is restricted to $\R = \Vert \cdot \Vert_1$ in the case of face recognition \cite{parpas2016}. These works were the first attempts to introduce multilevel methods in non-smooth optimization and they introduced key concepts such as the smoothing of $\R$ to obtain first order coherence between levels. Similar ideas have been proposed in \cite{ang2023} with adaptive restriction operators. This method requires strong convexity assumption on $\Lo$ to benefit from additional convergence properties. 

In our previous works, based on similar concepts, we proposed a multilevel forward-backward algorithm \cite{lauga2022} and a multilevel FISTA \cite{lauga2022_1}, both with stronger convergence guarantees than the one proposed in \cite{parpas2016,parpas2017,ang2023} (e.g., the convergence to a minimizer of the objective function). Our results do not require strong convexity assumption. 

Here, we extend our algorithmic procedure and its associated convergence guarantees to the more general case where the proximity operator of $\R$ is not necessarily known in explicit form. We replace the exact proximity operator in the forward-backward step of Equation \eqref{eq:fista_steps1}, by an approximated version:%
\begin{equation}
\label{eq:T_inexact}
(\forall x \in \RR^N) \qquad    \mbox{T}_{i}^{\epsilon}(x) \approx_{i,\epsilon} \prox_{\tau \R}\left(x - \tau \nabla \Lo(x)\right)
\end{equation} 

\noindent for some step-size $\tau>0$. In this expression, the index $i=\{0,1,2\}$ will refer to one of the three types of approximation that we will consider hereafter and $\epsilon$ corresponds to the induced approximation error  \cite{villa2013,aujol2015}. 
Accordingly, the inexact and inertial FB iterate reads: 
\begin{align}
    x_{k+1} & =\mbox{T}_{i}^{\epsilon_k}(y_k), \label{eq:inexact_fista_1}\\
    y_{k+1} & = x_{k+1} + \alpha_k(x_{k+1}-x_{k}). \label{eq:inexact_fista_2}
\end{align}

By injecting coarse corrections into the iterative scheme (\ref{eq:inexact_fista_1})-(\ref{eq:inexact_fista_2}), we propose a family of multilevel inertial forward-backward methods that we call IML FISTA for \textit{Inexact MultiLevel FISTA} .
It provides a multilevel extension of inertial strategies such as FISTA~\cite{beck2009-1,aujol2015}, that is fully adaptable to solve all problems of the form \eqref{eq:optim_article}, whether the proximity operator is known in close form, or approximated at each iteration. Naturally, when $d=0$, our framework coincides with a multilevel version of FB.

Our approach relies on the Moreau envelope, which in many cases can be easily derived to define smooth coarse approximations of $\R$.
Furthermore, we show that under mild assumptions, the convergence guarantees of inertial forward-backward algorithms \cite{aujol2015} hold also for IML FISTA. In particular, this is true for the convergence of the iterates, an important result for ill-posed problems. 

In addition, we propose a detailed version of the algorithm to solve Problem \eqref{eq:optim_article}, specifically designed for image restoration. Notably, we discuss the construction of coarse models and of information transfer operators that have good properties for image deblurring and image inpainting problems.

It is worth noticing that studying the properties of multilevel methods is a relevant perspective to tackle large scale problems in imaging: the multilevel framework is a quite general scheme that can be used whenever a hierarchical structure can be constructed on the underlying problem, as it is the case in this context. More importantly, such schemes can potentially be applied to any optimization method with suitable modifications, and the multilevel versions usually show faster convergence as compared to their one-level counterpart. 

As well as being interesting in its own right, the study of multilevel versions of FB and FISTA is therefore a first necessary step towards the acceleration of more complex schemes.

\noindent\paragraph{Contributions and organization of the article}
\begin{itemize}
    \item In Section \ref{sec:ml_framework}, we develop the first multilevel framework for inertial and inexact forward-backward to solve Problem \eqref{eq:optim_article}. Our proposition includes other multilevel methods previously proposed in the literature. We carry out the associated convergence analysis of the iterates and of the objective function. 
    \item In Section \ref{sec:MMiFB_image_restoration}, the proposed algorithm is specifically adapted to image restoration problems of the form \eqref{eq:optim_article}, when the proximity operator of $\R$  is not necessarily known in closed form. In addition, we focus on the design of wavelet-based transfer operators between resolution scales, for image reconstruction problems. 
    \item Extensive numerical experiments are performed in Section \ref{sec:results}, to compare the performances of IML FISTA versus FISTA on image reconstruction problems.
\end{itemize}
\section{Multilevel, Inexact and Inertial algorithm} 
\label{sec:ml_framework}

This first section focuses on Problem \eqref{eq:optim_article} to present the proposed IML FISTA in the most general context.
As in classical multilevel schemes for smooth optimization, our framework exploits a hierarchy of objective functions, representative of $F$ at different levels (scales or resolutions), and alternates minimization among these objective functions. 
The basic idea is to compute cheaper refinements at coarse resolution, which after prolongation to the fine levels, are used to update the current iterate.

\subsection{IML FISTA Algorithm}
\label{sec:MMiFB}
Without loss of generality and for the sake of clarity, we consider the two-level case: we index by $h$ (resp. $H$) all quantities defined at the fine (resp. coarse) level. We thus define $F_h:=F: \RR^{N_h}\rightarrow (-\infty,+\infty]$ the objective function at the fine level where $N_h=N$, such that $F_h = \Lo_h+ \R_h$ with $\Lo_h:=\Lo$ and $\R_h:=\R$ 
We associate this objective function at fine level with its coarse level approximation which we denote $F_H: \RR^{N_H}\rightarrow (-\infty,+\infty]$, with $N_H<N_h$, and in which $ \Lo_H,\R_H$ are lower dimensional approximations of $\Lo$ and $\R$.

One standard step of our algorithm can be summarized by the  following three instructions:
\begin{align}
    \bar{y}_{h,k} & = \text{ML}(\yhk), \label{eq:MLinexact_fista_0}\\
    \xhkun & = \mbox{T}_{i}^{\epsilon_{h,k}}(\bar{y}_{h,k}), \label{eq:MLinexact_fista_1}\\
    \yhkun & = \xhkun + \alpha_{h,k}(\xhkun-x_{h,k}) \label{eq:MLinexact_fista_2}
\end{align}
which are developped in detail in Algorithm \ref{alg:MMiFB_general}, and where  ML encompasses Steps \ref{alg_step:debutML} to \ref{alg_step:finML}.
Given the current iterate $\yhk$ at fine level, we can decide to update it either by a standard fine step, combining Steps \ref{alg_step:standard_update} and \ref{alg_step:FB}-\ref{alg_step:extrapolation} of the algorithm, or by performing iterations at the coarse level (cf. steps \ref{alg_step:projection}-\ref{alg_step:prolongation}) followed by a standard fine step (cf. \ref{alg_step:FB}-\ref{alg_step:extrapolation}). 
A particular attention needs to be paid to steps \ref{alg_step:projection}-\ref{alg_step:prolongation}, which produce a coarse correction that is used to define an intermediate fine iterate $\bar{y}_{h,k}$. The coarse correction is used to update the auxiliary variable $\yhk$ and not $\xhk$ directly (see Equations \eqref{eq:MLinexact_fista_1} and \eqref{eq:MLinexact_fista_2}). %
Thus, to obtain this coarse correction, the current iterate $\yhk$ is projected to the coarse level thanks to a projection operator $\IhH$, and it is used as the initialisation for the minimization of the coarse approximation $F_H$, which generates a sequence $(\xHl)_{\ell \in \mathbb{N}}$, where $k$ represents the current iteration at the fine level and $\ell$ indexes the iterations at the coarse level. %
This sequence is defined by $\xHlun = \Phi_{H,\ell}(\xHl)$, with $\Phi_{H,\ell}$ any operator such that, after $m>0$ coarser iterations,  $F_H(\xHm)\leq F_H(\xHO)$. For a discussion about an adequate choice for $m$, the reader could refer to \cite{lauga2022_1}. While this operator has to implicitly adapt to the current step $k$, its general construction does not depend on $k$.
After $m$ iterations at the coarse level we obtain a coarse direction $\xHm-\xHO$, prolongated at the fine level to update $\yhk$. 

A multilevel scheme requires transferring information from one level to an other. To do so,  we define two transfer information operators: a linear operator $\IhH: \RR^{N_h} \to \RR^{N_H}$ referred to as the \textit{restriction operator} that sends information from the fine level to the coarse level, and reciprocally $\IHh: \RR^{N_H} \to \RR^{N_h}$, the \textit{prolongation operator} that sends information from the coarse level back to the fine level. 

\begin{algorithm}
\caption{IML FISTA}\label{alg:MMiFB_general}
\begin{algorithmic}[1]
\STATE Set $x_{h,0},y_{h,0} \in \RR^N$, $t_{h,0}=1$
\WHILE {Stopping criterion is not met} 
    \IF {Descent condition and $r<p$} \label{alg_step:debutML}
    \STATE $r = r+1$,
    \STATE $\xHO  = \IhH \yhk \text{ \textcolor{gray}{Projection}}$ \label{alg_step:projection}\\
    \STATE $\xHm = \Phi_{H,m-1} \circ .. \circ \Phi_{H,0}(\xHO) \text{ \textcolor{gray}{Coarse minimization}}$\\
    \STATE Set $\taubarh>0$, \\
    \STATE $\bar{y}_{h,k} = \yhk + \mbox{$\taubarh$}\IHh\left(\xHm-\xHO\right) \text{ \textcolor{gray}{Coarse step update whose size is set by $\taubarh$}} \label{alg_step:prolongation}$
    \ELSE 
     \STATE $\bar{y}_{h,k} = y_{h,k}$ \label{alg_step:standard_update}
    \ENDIF \label{alg_step:finML}
    \STATE $x_{h,k+1} =\mbox{T}_{i}^{\epsilon_{h,k}}(\bar{y}_{h,k}) \text{ \textcolor{gray}{forward-backward step}}$ \label{alg_step:FB}\\
    \STATE $t_{h,k+1} = \left(\frac{k+a}{a}\right)^d$, $\alpha_{h,k} = \frac{t_{h,k}-1}{t_{h,k+1}}$ \label{alg_step:tk} \\
    \STATE $\yhkun = \xhkun + \alpha_{h,k}(\xhkun-x_{h,k}). \text{ \textcolor{gray}{Inertial step}}$ \label{alg_step:extrapolation}
\ENDWHILE
\end{algorithmic}

\end{algorithm}
\vspace{-0.3cm}

The central point of multilevel approaches is to ensure that the correction term $\xHm-\xHO$, after prolongation from the coarse to the fine level, leads to a decrease of $F_h$. For this, particular care must be taken in the selection of the following elements:
\begin{enumerate}[label=(\roman*)]
\item the coarse model $F_H$,
\item the minimization scheme $\Phi_{H, \cdot}$, 
\item the information transfer operators $\IhH$ and $\IHh$.
\end{enumerate}
We detail these choices in the following subsections. 

\subsubsection{Coarse model $F_H$} 
In our algorithm the construction of coarse functions relies on smoothing the non differentiable $\R_h$ \cite{beck2012} to maintain fidelity with the fine model, and at the same time to impose desirable properties to the coarse model. 

As demonstrated in \cite{lauga2022,lauga2022_1}, smoothing is a natural choice to extend ideas coming from the classical smooth case \cite{gratton2008} to multilevel proximal gradient methods. We take the ideas originally proposed in \cite{parpas2016,parpas2017}, and develop them further in the present contribution.  

\begin{definition}{\emph{(Smoothed convex function \cite[Definition 2.1]{beck2012})}} \label{def:smooth_convex_func} Let $R$ be a convex, l.s.c., and proper function on $\RR^{N}$. For every $\gamma>0$, a continuously differentiable $R_{\gamma}$ %
is a smoothed convex approximation of $R$ if there exist finite valued scalars $\eta_1,\eta_2$ satisfying $\eta_1+\eta_2 >0$ such that the following holds:
\begin{equation}
    (\forall y \in \RR^{N}) \qquad R(y) - \eta_1 \gamma \leq R_{\gamma}(y) \leq R(y) + \eta_2 \gamma.
\end{equation}
\end{definition}

Such smoothed convex functions exist if the smoothing is done according to the principles developed in \cite{beck2012} where the sum  $\eta_1+\eta_2$ depends on $R$ and on the type of smoothing.
\begin{definition}{\emph{(Coarse model $F_{H}$ for non-smooth functions.)}}\label{def:coarse_model}
The coarse model $F_{H}$ is defined for the point $y_h\in\RR^{N_h}$ as:
\begin{equation}
\label{eq:FH}
  \qquad F_{H} = \Lo_H + \R_{H,\gamma_H} + \langle v_H,\cdot \rangle,
\end{equation}
where
\begin{align}
\label{eq:v_Hk_exact}
v_H = \IhH \left(\nabla \Lo_h(y_h) + \nabla \R_{h,\gamma_h}(y_h)\right) -(\nabla \Lo_H(\IhH y_h) + \nabla\R_{H,\gamma_H}(\IhH y_h)).
\end{align} 
$\R_{h,\gamma_h}$ and  $\R_{H,\gamma_H}$ are  smoothed versions of $\R_h$ and $\R_H$ respectively, and they verify Definition \ref{def:smooth_convex_func} with smoothing parameters $\gamma_h>0$ and $\gamma_H>0$.
\end{definition}
Adding the linear term $\langle v_H,\cdot \rangle$ to $\Lo_H + \R_{H,\gamma_H}$ allows to impose the so-called \emph{first order coherence} recalled in Definition \ref{def:first_order_coherence} below.
\begin{remark}
    Note that if $\R_h$ and $\R_H$ are smooth by design, one can simply replace $\R_{H,\gamma_H}$ and $\R_{h,\gamma_h}$ by $\R_H$ and $\R_h$, respectively. The construction stays otherwise the same. %
\end{remark}
\begin{definition}{\emph{(First order coherence \cite{nash2000,parpas2016,parpas2017})}.}\label{def:first_order_coherence}
The first order coherence between the smoothed version of the objective function $F_h$ at the fine level and the coarse level objective function $F_H$ is verified in a neighbourhood of $y_h$ if the following equality holds:
\begin{equation}\label{eq:first_order_coherence}
\nabla F_{H}(\IhH y_h) = \IhH \nabla \left(\Lo_h + \R_{h,\gamma_h}\right)(y_h).
\end{equation} 
\end{definition}
\begin{lemma}
\label{lemma:first_order_coherence}
If $F_H$ is given by Definition \ref{def:coarse_model},  it necessarily verifies the first order coherence (Definition \ref{def:first_order_coherence}).
\begin{proof}
Considering the gradient of the coarse model $F_{H}$ and combining it with the definition of $v_H$ in Equation \eqref{eq:v_Hk_exact}, yields 
\begin{equation}
\begin{split}
    \nabla F_{H}(\IhH y_h) & = \nabla \Lo_H(\IhH y_h) + \nabla\R_{H,\gamma_H}(\IhH y_h) + v_H, \\
    & = \IhH \left(\nabla \Lo_h(y_h) + \nabla \R_{h,\gamma_h}(y_h)\right).
\end{split}
\end{equation}
\end{proof}
\end{lemma}
This condition ensures that, in the neighbourhood of the current iterates $y_h = \yhk$ and $\IhH \yhk = \xHO$, smoothed versions of the fine and of the coarse level objective functions are coherent up to order one \cite{parpas2017}.

Figure \ref{fig:first_order_coherence} illustrates the effect of the first order coherence on the alignment of the  gradients of smooth objective functions at fine and coarse levels.

\begin{figure}
    \centering
    \includegraphics[width = 0.8 \textwidth]{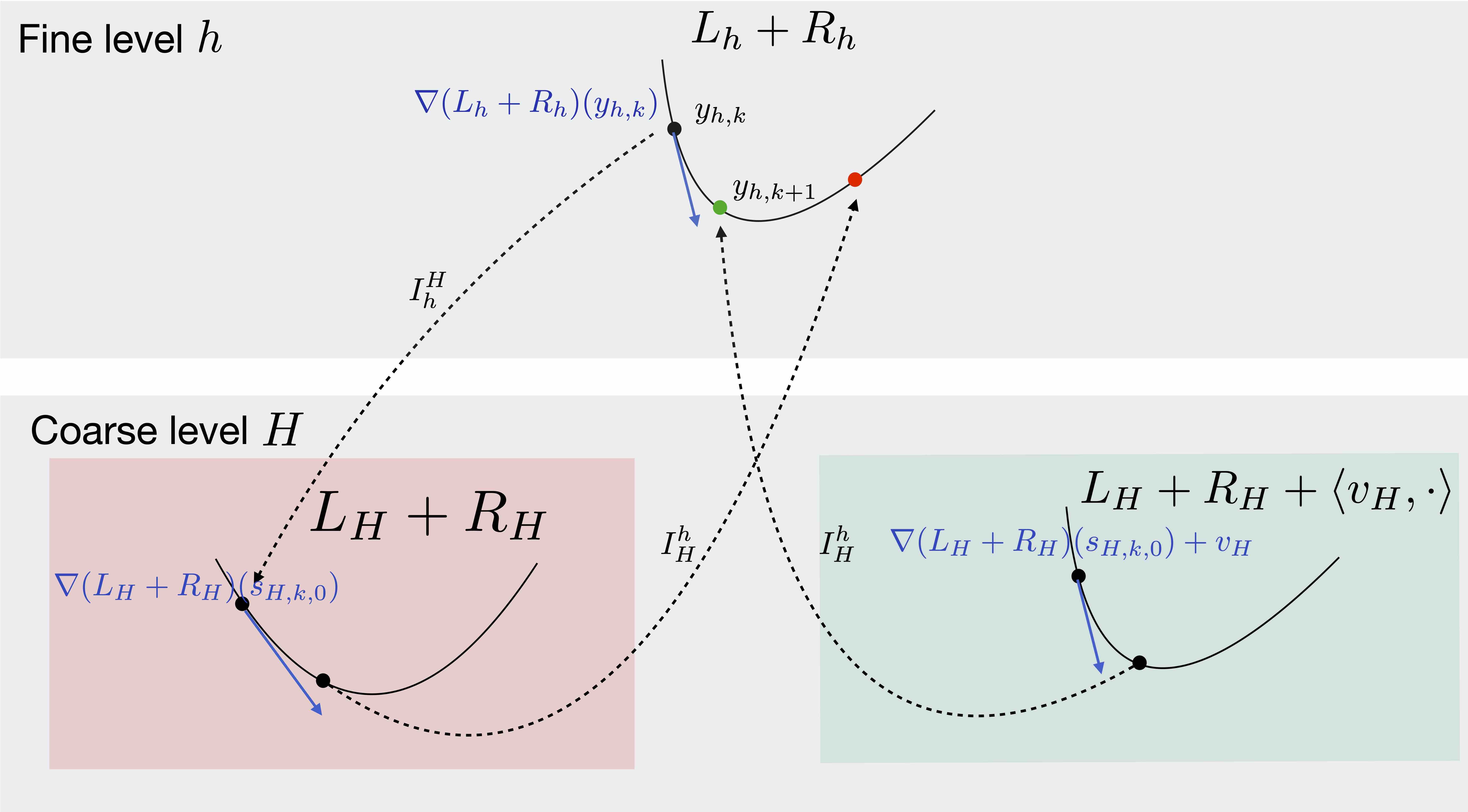}
    \caption{Illustration of the first order coherence between two smooth functions $\Lo_h+\R_h$ and $\Lo_H+\R_H$. Left lower part: Without first order coherence, points decreasing $\Lo_H+\R_H$ do not necessarily decrease $\Lo_h+\R_h$. Right lower part: First order coherence rotates the graph of $\Lo_H+\R_H$ around $\xHO$ so that decreasing $\Lo_H+\R_H$ also entails decreasing $\Lo_h+\R_h$.\vspace{-0.6cm}}
    \label{fig:first_order_coherence}
\end{figure}

\subsubsection{Choice of coarse iterations} The operators $\Phi_{H,\bullet}$ aim to build a sequence producing a sufficient decrease of $F_H$ after $m$ iterations.
\begin{assumption}{(Coarse model decrease)} \label{def:coarse_operator}
Let $(\Phi_{H,\ell})_{\ell \in \mathbb{N}}$ be a sequence of operators such that there exists an integer $m>0$ that guarantees that if $s_{H,m} = \Phi_{H,m-1} \circ \ldots \circ \Phi_{H,0}(s_{H,0})$ then $F_H(s_{H,m})\leq F_H(s_{H,0})$. Moreover, $s_{H,m}-s_{H,0}$ is bounded.
\end{assumption} 
Some typical choices for $\Phi_{H,\ell}$ are the gradient descent step, inertial gradient descent step, forward-backward step or inertial forward-backward step (see \cite{lauga2022_1} for a comparison of these operators in a multilevel context - the choice depends on the intensity of degradation for image reconstruction problems). These operators guarantee that $s_{H,m}-s_{H,0}$ is a bounded (through convergence of the sequence \cite{aujol2015}) descent direction for $F_H$. 

\subsubsection{Construction of information transfer operators} \label{sec:information_transfer} Going from one level to the other requires several information transfers. For this purpose we use the following classical definition.
\begin{definition} \label{def:inf_transf_op} The two operators  $\IhH: \RR^{N_h} \to \RR^{N_H}$ and $\IHh: \RR^{N_H} \to \RR^{N_h}$ are \textit{coherent information transfer (CIT)} operators, if there exists $\nu>0$ such that:
\begin{equation}
    \IHh = \nu (\IhH)^T.
\end{equation}
\end{definition}
There are many ways to construct CIT operators. The most standard one for multilevel methods is the dyadic decimated weighted operator \cite{briggs2000}. In the particular case of squared grids of size $\sqrt{N_h}\times\sqrt{N_h}$ and  $\sqrt{N_H}\times\sqrt{N_H}$ at fine and coarse level respectively, and for $N_H = N_h/4$ corresponding to a decimation factor of 2 along rows and columns, the restriction operator reads: 
\begin{equation}
    \IhH =  \frac{1}{16}
    \underbrace{\left(\begin{array}{ccccccc}
        2 & 1 & 0 & \ldots & & & 0\\
        0 & 1 & 2 & 1 & 0 & \ldots & 0 \\
        \vdots & \ddots & \ddots & & & & 0\\
        0  & \ldots & & 0 & 1 & 2 & 1
    \end{array}\right)}_{\sqrt{N_h}/2 \times \sqrt{N_h}} \otimes \underbrace{\left(\begin{array}{ccccccc}
        2 & 1 & 0 & \ldots & & & 0\\
        0 & 1 & 2 & 1 & 0 & \ldots & 0 \\
        \vdots & \ddots & \ddots & & & & 0\\
        0  & \ldots & & 0 & 1 & 2 & 1
    \end{array} \right)}_{\sqrt{N_h}/2 \times \sqrt{N_h}}\in \RR^{N_H \times N_h}. \vspace{-1em}
    \label{eq:inf_trans_standard}
\end{equation}
The pair ($\IhH$, $\IHh$) provides a simple and intuitive way 
to transfer information back and forth between fine and coarse scales, by means of linear B-spline interpolation. Other operators of the form of \eqref{eq:inf_trans_standard} corresponding to higher order interpolation have been proposed in \cite{donatelli2010} and are  commonly used in multigrid methods for solving PDEs \cite{donatelli2012}.
The literature on transfer operators being much more developed in the context of PDEs, it gives a rich starting point for multilevel optimization algorithms. In particular, the authors of \cite{greenfeld2019} introduced a learning framework to optimize multigrid PDEs solvers that pay great attention to the properties of the information transfer operators. 
\subsubsection{Fine model minimization with multilevel steps}
With the previous definitions of $F_H$, $\Phi_{H,\bullet}$ and $\IhH$, the following lemmas prove that minimization at the coarse level also induces a descent direction at the fine level. 
\begin{lemma}{\emph{(Descent direction for the fine level smoothed function)}.}\label{lemma:decrease_smooth}
Let us assume that $\IhH$ and $\IHh$ are CIT operators and that $F_H$ satisfies Definition \ref{def:coarse_model}. 
and  $\Phi_{H,\bullet}$ verifies  Assumption \ref{def:coarse_operator}. 
Then, $\IHh(s_{H,m}-s_{H,0})$ is a descent direction for $\Lo_h + \R_{h,\gamma_h}$.
\begin{proof}
Set $y_h \in \RR^{N_h}$ and let us define $p_H:=s_{H,m}-s_{H,0}$. Recall that $s_{H,0} = \IhH y_h$. From the definition of descent direction we have that: 
\begin{equation*}
     \langle p_H,\nabla F_{H}(s_{H,0}) \rangle \leq 0.
\end{equation*}
By the  first order coherence and imposing $\IhH = \nu^{-1} \left(\IHh\right)^T$ we obtain
\begin{align*}
    \langle p_H,\nabla F_{H}(s_{H,0}) \rangle = \langle p_H,\IhH \nabla (\Lo_h + \R_{h,\gamma_h})(y_h) \rangle =
    \nu^{-1} \langle \IHh(p_H),\nabla (\Lo_h + \R_{h,\gamma_h})(y_h) \rangle \leq 0.
\end{align*}
\end{proof}
\end{lemma}
We can now go a step further and derive a bound on the decrease of the (non-smooth) objective function at the fine level $F_h:=L_h+R_h$. Following \cite{parpas2016,parpas2017}, we search a proper step size  $\bar{\tau}_h$ that avoids ``too" big corrections from the coarse level by guaranteeing that:
\begin{equation}
\label{eq:tau_bar}
(L_h+R_{h,\gamma_h})(y_h + \bar{\tau}_h\IHh(s_{H,m}-s_{H,0})) \leq (L_h+R_{h,\gamma_h})(y_h).
\end{equation}
\begin{lemma}{\emph{(Fine level decrease)}.}\label{lemma:fine_level_decrease} If the assumptions of Lemma  \ref{lemma:decrease_smooth} hold, the iterations of Algorithm \ref{alg:MMiFB_general} ensure:
\begin{equation}
\label{eq:fine_level_decrease}
F_h(y_h + \bar{\tau}\IHh(s_{H,m}-s_{H,0})) \leq F_h(y_h) + (\eta_1 + \eta_2) \gamma_h.
\end{equation}
\end{lemma}
\begin{proof}
This directly comes from the definition of a smoothed convex function (Definition \ref{def:smooth_convex_func}). As there exists a value of $\bar{\tau}_h$ satisfying Equation \eqref{eq:tau_bar}, we have:
\begin{equation}
\begin{split}
    F_h(y_h + \bar{\tau}_h\IHh(s_{H,m}-s_{H,0})) & \leq (L_h+R_{h,\gamma_h})(y_h+\bar{\tau}_h\IHh(s_{H,m}-s_{H,0})) + \eta_1 \gamma_h \\
    & \leq (L_h+R_{h,\gamma_h})(y_h) + \eta_1 \gamma_h \\
    & \leq F_h(y_h) + (\eta_1 + \eta_2) \gamma_h.
\end{split}
\end{equation}
\end{proof} %

This result shows that a coarse level minimization step leads to a decrease of $F_h$, up to  a constant $(\eta_1+\eta_2) \gamma_h$ that can be made arbitrarily small by driving $\gamma_h$ to zero.

This type of result is commonly found in the literature of multilevel algorithms \cite{parpas2016,parpas2017,lauga2022,lauga2022_1} but it is not sufficient to guarantee the convergence of the generated sequence. In the next section we derive stronger convergence guarantees.

\subsection{Convergence of the iterates}\label{sec:convergence}

In order to obtain the convergence of the iterates to a minimizer of $F=F_h$ and the optimal rate of convergence of the objective function values, we need to take into account two types of inexactness in the computation of an iterate: one on the proximity operator of $\R_h$ and one on the gradient of $\Lo_h$. The error on the gradient will allow us to compute coarse corrections with our multilevel framework, while the error on the proximity operator will allow us to consider approximation of proximity operators whose closed form is unknown.

The goal of this section is to show that an iteration of our algorithm (Steps \ref{alg_step:FB}-\ref{alg_step:extrapolation} in Algorithm \ref{alg:MMiFB_general}) can be reformulated as:
\begin{equation}
\begin{split}
   \xhkun & \approx_{i,\epsilon_{h,k}} \prox_{\tauhk \R_h}\left(\yhk - \tauhk \nabla \Lo_h \left(\yhk \right) + \ehk \right), \\
   \yhkun & = \xhkun + \alpha_{h,k}(\xhkun - \xhk),
\end{split}
\end{equation}
where we introduce $\ehk$ to model uncertainties on the gradient step  due to the multilevel corrections and the pair $(i,\epsilon_{h,k})$ introduced in \eqref{eq:T_inexact}, to designate the type and the accuracy of the proximity operator approximation.
Such rewriting allows us to fit in the framework described by the authors of \cite{aujol2015} to define an inexact and inertial forward-backward algorithm. 

\paragraph{Inexactness due to coarse corrections}
As presented in the algorithm, a coarse correction is inserted before a typical fine level step. We can see this coarse correction as some kind of error on the gradient of $\Lo_h$. In a typical multilevel step, at the fine level (cf. Steps \ref{alg_step:FB} and \ref{alg_step:prolongation} of Algorithm~\ref{alg:MMiFB_general}), the update would simply take the form: %
\begin{align}
\xhkun & \approx_{i,\epsilon_{h,k}} \prox_{\tauhk \R_h}\left(\bar{y}_{h,k} - \tauhk \nabla \Lo_h \left(\bar{y}_{h,k}\right) \right),
\label{eq:typ_mult_step}\\
 \bar{y}_{h,k} & =\yhk + \taubarh\IHh (\xHm-\xHO).
 \end{align} 
It is straightforward that the coarse corrections are finite as we sum a finite number of bounded terms, thanks to computing updates at the coarse level with a Lipschitz gradient. This reasoning is detailed in the following proof.
%
\begin{lemma}{\emph{(Coarse corrections are finite)}} \label{lemma:correction} 
Let $\beta_{h}$ and $\beta_{H}$ be the Lipschitz constants of the gradients of $\Lo_h$ and $\Lo_H$, respectively. Assume that we compute at most $p$ coarse corrections. Let $\tauhk, \tauHl \in (0, +\infty)$ be the step sizes taken at fine and coarse levels, respectively.
Assume that $\tauHl < \beta_H^{-1}$ and that $\tauhk < \beta_{h}^{-1}$ and denote $\bar{\tau}_h = \sup_{k} \taubarh$.
Then the sequence $(\ehk)_{k\in \mathbb{N}}$ in $\RR^{N_h}$ generated by Algorithm \ref{alg:MMiFB_general} is defined as:
\begin{equation}
\ehk = \tauhk \left( \nabla \Lo_h (\yhk) - \nabla \Lo_h (\bar{y}_{h,k}) + (\tauhk)^{-1}\bar{\tau}_{h,k} \IHh (\xHm-\xHO) \right), 
\end{equation}
if a coarse correction has been computed, and  $\ehk = 0$ otherwise. This sequence is such that  $\sum_{k\in \mathbb{N}} k\Vert \ehk \Vert < +\infty$.
\begin{proof}
We are not concerned with the proximity operator (backward step) in Equation \eqref{eq:typ_mult_step} so we focus on the forward step. Considering $\nabla \Lo_h \left(\bar{y}_{h,k}\right) = \nabla \Lo_h \left(\bar{y}_{h,k}\right) - \nabla \Lo_h \left(\yhk \right) + \nabla \Lo_h \left(\yhk\right)$ and $\bar{y}_{h,k} = \bar{y}_{h,k} + \yhk - \yhk$, the forward step can be rewritten as: 
\begin{equation*}
\bar{y}_{h,k} - \tauhk\nabla \Lo_h (\bar{y}_{h,k})  = 
\yhk - \tauhk \nabla \Lo_h (\yhk)  + \tauhk \left( \nabla \Lo_h (\yhk) - \nabla \Lo_h (\bar{y}_{h,k}) + \frac{1}{\tauhk} (\bar{y}_{h,k}-\yhk) \right).
\end{equation*}
And so, each time a multilevel step is performed, it induces at iteration $k$, an error that reads:
\begin{equation*}
\ehk = \tauhk \left( \nabla \Lo_h (\yhk) - \nabla \Lo_h (\bar{y}_{h,k}) + (\tauhk)^{-1}\bar{\tau}_{h,k} \IHh (\xHm-\xHO) \right). 
\end{equation*}
Now, assuming that we use inertial inexact proximal gradient steps at the coarse level, the corresponding minimization verifies Assumption \ref{def:coarse_operator} on the decrease of $F_H$. It also produces bounded sequences if constructed according to the rules of \cite[Definition 3.1, Theorem 4.1]{aujol2015} as the sequences $(\xHl)_{k \in \mathbb{N},\ell \in \mathbb{N}^*}$ converge. 
The sequence $(\ehk)_{k\in \mathbb{N}}$ has at most $p$ non zero bounded terms, as shown below:
\begin{align}
    \tauhk^{-1}\Vert \ehk \Vert & = \Vert \nabla \Lo_h (\yhk) - \nabla \Lo_h (\bar{y}_{h,k}) + (\tauhk)^{-1}\bar{\tau}_h \IHh (\xHm-\xHO) \Vert \\
    & \leq \beta_h \bar{\tau}_h \Vert \IHh (\xHm-\xHO) \Vert + (\tauhk)^{-1} \bar{\tau}_h \Vert\IHh (\xHm-\xHO) \Vert \\
    & \leq \bar{\tau}_h\left(\beta_h +\frac{1}{\tauhk}\right)\Vert\IHh (\xHm-\xHO) \Vert.
\end{align}
The second inequality is deduced from the fact that $\Lo_h$ has a $\beta_h$-Lipschitz gradient and that $\bar{y}_{h,k} -\yhk = \bar{\tau}_{h,k} \IHh (\xHm-\xHO)$. 
Finally as $(\Vert \xHO - s_{H,k,m}\Vert)_{k \in \mathbb{N}}$ is bounded, we have:
\begin{align}
    \tauhk^{-1}\Vert \ehk \Vert & \leq \bar{\tau}_h\left(\beta_h +\frac{1}{\tauhk}\right)\sup_{k \in \mathbb{N}}\Vert\IHh (\xHm-\xHO) \Vert < +\infty. 
\end{align}
\end{proof}
\vspace{-0.3cm}
\end{lemma}

\paragraph{Inexactness due to approximation of the proximity operator} \label{sec:approx_prox} To account for  inexactness in the proximity operator computation, one needs to enlarge the notion of subdifferential through the following definition \cite{aujol2015}:
\begin{definition}{\emph{($\epsilon$-subdifferential)}}
The $\epsilon$-subdifferential of $\R$ at $z \in $ dom $\R$ is defined as:
\begin{equation}
    \partial_{\epsilon}\R(z) = \{ y \in \RR^N \;|\; \R(x) \geq \R(z) + \langle x-z,y \rangle - \epsilon, \forall x \in \RR^N \}.
\end{equation}
\end{definition}

\noindent Based on this definition, three types of approximations of proximity operators are proposed.
\begin{definition}{\emph{(Type $0$ approximation \cite{combettes2005})}.} %
We say that $z \in \RR^N$ is a type $0$ approximation of $\prox_{\gamma \R}(y)$ with precision $\epsilon$, and we write $z \approx_{0,\epsilon}\prox_{\gamma \R}(y)$, if and only if:
\begin{equation}
    \Vert z - \prox_{\gamma \R}(y) \Vert \leq \sqrt{2 \gamma\epsilon}.
\end{equation}
\end{definition}

\begin{definition}{\emph{(Type $1$ approximation \cite{villa2013})}.} 
We say that $z \in \RR^N$ is a type $1$ approximation of $\prox_{\gamma \R}(y)$ with  precision $\epsilon$, and we write $z \approx_{1,\epsilon} \prox_{\gamma \R}(y)$, if and only if:
\begin{equation}
     0 \in \partial_{\epsilon} \left(\R(z) + \frac{1}{2\gamma}\Vert z - y \Vert^2\right).
\end{equation}
\end{definition}
\begin{definition}{(\emph{Type $2$ approximation \cite{villa2013})}.}
We say that $z \in \RR^N$ is a type $2$ approximation of $\prox_{\gamma \R}(y)$ with  precision $\epsilon$, and we write $z \approx_{2,\epsilon}\prox_{\gamma \R}(y)$, if and only if:
\begin{equation}
     \gamma^{-1}(y-z) \in \partial_{\epsilon} \R(z).
\end{equation}
\end{definition}
\noindent Approximation of type 2 implies approximation of type 1 \cite{villa2013,aujol2015} and under some conditions discussed in \cite{villa2013}, approximation of type 0 implies approximation of type 2.

When these approximations are used in forward-backward-based algorithms, convergence guarantees are known from the literature: approximations of type $1$ and $2$ are covered by \cite{aujol2015} for inertial versions of the forward-backward algorithm, while the type 0 approximation is treated in \cite{combettes2005} only for the forward-backward algorithm.  Typical cases of image restoration, where dual optimization is used, are based on approximations of type 2 (see Section \ref{sec:MMiFB_image_restoration}). 

The type of chosen approximation defines how the sequence $(\epsilon_{h,k})_{k\in\mathbb{N}}$ will be summable against $k^{2d}$ and thus, it does not depend on the multilevel framework.
\paragraph{Convergence of Algorithm \ref{alg:MMiFB_general}} We now discuss the convergence of our algorithm for the three types of approximation of the proximity operator. 

We first consider a standard inexact forward-backward with a finite number of multilevel coarse corrections.
\begin{theorem}[Approximation of Type $0$] 
\label{th:convergence_0}
Let us suppose in Algorithm \ref{alg:MMiFB_general} that $\forall k \in \mathbb{N}^*,~\alpha_{h,k} = 0$ at step  \ref{alg_step:extrapolation}, that the assumptions of Lemma \ref{lemma:correction} hold, and that the sequence $(\epsilon_{h,k})_{k\in \mathbb{N}}$ is such that $\sum_{k\in \mathbb{N}}\sqrt{\Vert \epsilon_{h,k} \Vert}< + \infty$. Set $x_{h,0} \in \RR^{N_h}$ %
and choosing approximation of Type 0,  the sequence $(\xhk)_{k\in\mathbb{N}}$ converges to a minimizer of $F_h$.
\begin{proof}
The proof stems from Theorem 3.4 in \cite{combettes2005} applied to the defined sequence.
\end{proof}
\end{theorem}

\begin{theorem}[Approximations of Type $1$ and Type $2$]
\label{th:convergence_12}
Let us suppose in Algorithm \ref{alg:MMiFB_general}, that  $\forall k\in \mathbb{N}^*,~t_{h,k+1}=\left(\frac{k+a}{a}\right)^d$, with $(a,d)$ satisfying the conditions in \cite[Definition 3.1]{aujol2015}, and that the assumptions of Lemma \ref{lemma:correction} hold. Moreover, if we assume that:
\begin{itemize}
    \item[] $\sum_{k=1}^{+\infty}k^d\sqrt{\epsilon_{h,k}} < + \infty$ in the case of Type 1 approximation,
    \item[] $\sum_{k=1}^{+\infty}k^{2d}\epsilon_{h,k} < + \infty$ in the case of Type 2 approximation,
\end{itemize}
then, we have that:
\begin{itemize}
    \item[-] The sequence $(k^{2d}\left(F_h(\xhk)-F_h(x^*)\right))_{k \in \mathbb{N}}$ belongs to $\ell_{\infty}(\mathbb{N})$.
    \item[-] The sequence $(x_{h,k})_{k \in \mathbb{N}}$ %
    converges to a minimizer of $F_h$.
\end{itemize}

\begin{proof}
\cite[Theorem 3.5, 4.1, and Corollary 3.8]{aujol2015} with Lemma  \ref{lemma:correction} yield the desired result.
\end{proof}
\end{theorem}
Theorem \ref{th:convergence_0} and \ref{th:convergence_12} generalize convergence results previously obtained in \cite[Theorem 1]{lauga2022_1}.
\noindent When $\epsilon_{h,k} = 0$ for all $k$, we recover the convergence result obtained in \cite[Theorem 1]{lauga2022} for exact proximity operators.

\subsection{Extension to the multilevel case}
Extending these convergence results to more than two levels is straightforward. If the algorithm is used on $J$ levels, we just have to apply the analysis derived above to each pair of consecutive levels. Then, recursively, showing that the coarsest level produces a bounded coarse correction will ensure that the upper finer level will converge to one of its minimizers, producing in turn a bounded coarse correction for the next upper finer level, and so on.
\paragraph{Defining the coarse cycles}
We use the following notation for the multilevel  schemes. If the dimension of the problem at fine level is $N_h = (2^{J})^2$, following the classical wavelet nomenclature, we index with $J$ the finest level. So, for an image of size $1024 \times 1024$, $J=10$. The coarse levels are then associated to $J-1$, $J-2$, $J-3$, etc. We use V-cycles \cite{briggs2000}, as depicted in  Figure \ref{fig:v-cycle}. 
\begin{figure}
    \centering
    \includegraphics[width = 0.5\textwidth]{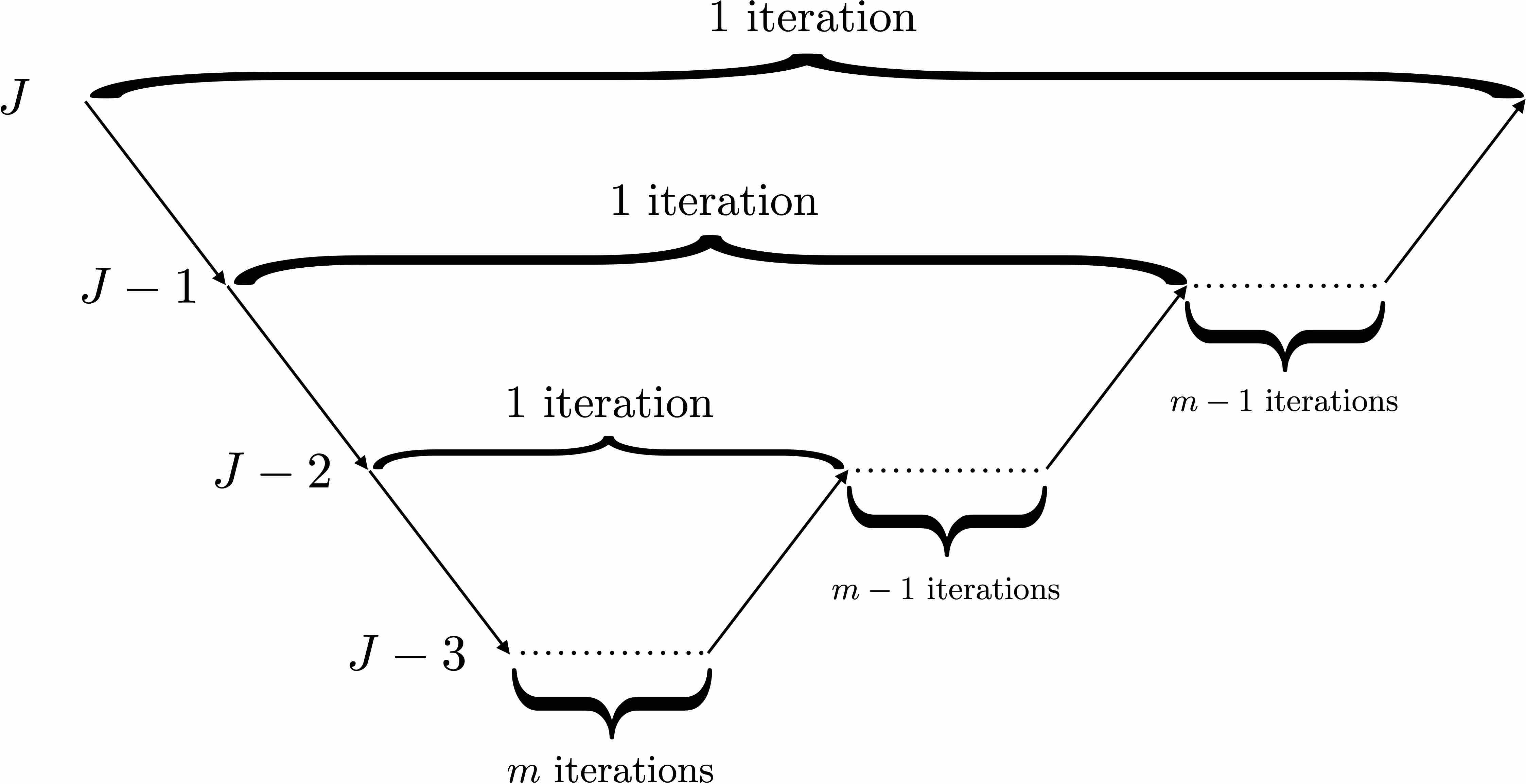}
    \caption{Scheme of a typical V-cycle for a multilevel algorithm with $4$ levels and $3$ coarse levels. When we use $p$ times the coarse model, we repeat $p$ times this V-cycle scheme. \vspace{-0.5em}}
    \label{fig:v-cycle}
\end{figure}

\section{IML FISTA for image reconstruction}
\label{sec:MMiFB_image_restoration}

In this section we adapt our Inexact MultiLevel FISTA to image reconstruction problems in the framework of Problem \eqref{eq:optim_article}. We present our problem in a multilevel context, then we propose CIT operators designed for image reconstruction problems and we derive the construction of a good coarse model through a specific choice of smoothing. Finally, we detail the computation of the proximity operator of $g_h \circ \D_h$.
\subsection{Definition of the problem at fine level}
Let us specify Problem \eqref{eq:optim_article} to the specific context of image restoration in multilevel notations:
\begin{equation}\label{eq:ml_pb}
    \widehat{x} \in \Argmin_{x_h \in \RR^{N_h}} F_h(x):=f_h(\A_h x_h) + g_h(\D_h\vspace{0.1em}x_h)\vspace{-0.5em}
\end{equation} 
with $A_h\in\RR^{M_h\times N_h}$ and $\D_h \in \RR^{(N_h \times \widetilde{K})\times N_h}$ ($\widetilde{K},M_h>0$). The parameter $\tilde{K}$ expresses the fact that operator $D_h$ can map $x_h$ to a higher dimensional space, e.g.  $\widetilde{K}=2$ for Total Variation penalization. In this expression, $x_h = \left(x_h^i\right)_{1\leq i \leq N_h}$ is the vectorized version of an image $X_h$ of $N_{h,r}$ rows and $N_{h,c}$ columns, and where each pixel corresponds to a vector of $C\geq 1$ components (e.g. $C=3$ for the RGB bands of a color image). Hence, we have $N_h = N_{h,r} \times N_{h,c} \times C$.
In the following, as the operators we deal with, apply separately to each channel, for the sake of clarity and without loss of generality, we present their construction for grayscale images corresponding to $C=1$.
\subsubsection{Examples of data fidelity term $f_h\circ\A_h$}
\paragraph{Deblurring problem} When the degradation of the image corresponds to a blurring effect, the operator $\A_h$ is a convolution matrix built from a two dimensional Point Spread Function (PSF).
As it is the case for Gaussian blurs, the PSF function takes often the form of a separable kernel (horizontally and vertically) and $\A_h$ can be decomposed into a Kronecker product:
\begin{equation}
    \A_h = \A_{h,r} \otimes \A_{h,c} 
\end{equation} 
with $\A_{h,r} \in \mathbb{R}^{N_{h,c}\times N_{h,c}}$ and $\A_{h,c} \in \mathbb{R}^{N_{h,r}\times N_{h,r}}$. From the numerical viewpoint, this Kronecker decomposition is particularly efficient for processing large images, and can be easily implemented with the HNO package \cite{hansen2006}.
Finally, as it is common in image restoration, the data-fidelity term is a least square regression:
\begin{equation}
    (\forall x_h \in \RR^{N_h}) \qquad f_h(\A_h x_h) = \frac{1}{2}\Vert \A_h x_h - z_h \Vert_2^2 = \frac{1}{2}\sum_{i=1}^{N_h} ((\A_h x_h)^i-(z_h)^i)^2.
    \label{eq:standard_least_squares}
\end{equation}

\paragraph{Inpainting problem} When the degraded image coincides with the original image but with potentially altered or missing pixels, the reconstruction task is called inpainting and $\A_h$ is a measurement operator that keeps a subset $I\subseteq \{1,\ldots,N_h\}$ of pixels of the image and removes the others. Here, we assume that the subset $I$ is chosen randomly. Formally $\A_h$ takes the form of a diagonal matrix with a Bernoulli random variable (zeros and ones) on its entries, and it plays the role of a mask applied to the  image $x_h$:
\begin{equation}
    \left(\A_h x_h \right)^i = \left\{\begin{array}{cc}
    x_h^i     & \mbox{if } i\in I \\
    0     & \mbox{otherwise}
    \end{array}\right.
    \label{eq:inpainting_hadamard}
\end{equation}
In this case too, the data-fidelity term is a least square regression as in Equation \eqref{eq:standard_least_squares}.

\subsubsection{Examples of regularization term $g_h\circ\D_h$}
\paragraph{Wavelet transform norm} The operator $\D_h$ associated with a wavelet transform regularization is the discrete wavelet transform operator which computes a given number of consecutive decimated low pass and high pass filtering of the image $x_h$. The classical regularization associated is the application of the $l_1$-norm on the discrete wavelet transform coefficients. Such regularization was for instance used in a multilevel framework in \cite{parpas2017,lauga2022,lauga2022_1}. 
\paragraph{Total Variation} The operator $\D_h$ associated with the Total Variation (TV) computes the first order differences between the component $i$ of $x_h$ and its
horizontal/vertical nearest neighbours ($x_h^{i_c},x_h^{i_r}$) (lower/right in the image case). It is defined such that for all $x_h \in \RR^{N_h}$, and for each pixel $i\in \{1,\ldots,N_h\}$, 
\begin{equation}
    \left(\D_h x_h\right)^{i} = \left[\begin{array}{c}
        x_h^{i}-x_h^{i_r} \, , \,
        x_h^{i}-x_h^{i_c} 
    \end{array}\right], 
\end{equation}
paying particular attention to the management of border effects.
Here $\D_h x_h$ belongs to $\RR^{N_h \times 2}$ ($\widetilde{K} = 2$). With this definition, the classical isotropic Total Variation semi-norm \cite{beck2009} reads:
\begin{equation}
    g_h(\D_h x_h) = \lambda_h \sum_{i=1}^{N_h} \Vert \left(\D_h x_h\right)^{i} \Vert_2 = \lambda_h\sum_{i=1}^{N_h} \sqrt{|x_h^{i}-x_h^{i_1}|^2 + |x_h^{i}-x_h^{i_2}|^2} = \lambda_h\Vert \D_h x_h \Vert_{2,1}
    \label{eq:total_variation_l1}
\end{equation}
with $\lambda_h>0$.

\paragraph{Non-Local Total Variation}
The operator $\D_h$ associated with the Non-Local Total Variation (NLTV) extends TV to a non local neighbourhood of the current pixel $i$. In words, it is the operator that computes the weighted differences between the current pixel $i$ of an image $x_h$ and a subset $\mathcal{N}_i$ of pixels localized near $i$.

For every $x_h \in \RR^{N_h}$, and at each pixel $i\in \{1,\ldots,N_h\}$, for some given weights $\omega^{i,j}>0$, 
\begin{equation}
    \left(\D_h x_h\right)^{i} = \left[\begin{array}{c}
        \omega^{i,j}\left(x_h^{i}-x_h^{j}\right) 
    \end{array}\right]_{j\in\mathcal{N}_i}.
\end{equation}
Here $\D_h x_h$ belongs to $\RR^{N_h \times \widetilde{K}}$ and  $\widetilde{K}$ is the cardinality of the subset $\mathcal{N}_i$. For every $i\in \{1,\ldots,N_h\}$ and $j\in\mathcal{N}_i$, the weights $\omega^{i,j}>0$ depend on the similarity (e.g., $\ell_2$ norm) between patches that are centered around components $i$ and $j$ of the image \cite{chierchia2014}. 

As for the isotropic TV semi-norm, a $\ell_p$ ($p\geq1$) based NLTV semi-norm takes the form:
\begin{equation}
    g_h(\D_h x_h) = \lambda_h\sum_{i=1}^{N_h} \Vert \left(\D_h x_h\right)^{i} \Vert_p \qquad \mbox{with} \qquad  \lambda_h>0.
    \label{eq:NL_total_variation_l1}
\end{equation}

\subsection{Information transfer for image reconstruction problems}
\label{sec:information_transfer_image}
In the context of image recons\-truction problems, we consider CIT operators that rely on wavelet bases (referred to as wavelet CIT in the following).
The idea of constructing such information transfer operators traces back to works dedicated to image deblurring problems either based on biorthogonal wavelets \cite{cheng2003} or Haar and Symlets wavelets \cite{malena2009,espanol2010,deng2012}.
Our objective is to obtain a computationally efficient coarse approximation of a vector lying in a higher resolution space, from the approximation coefficients of its discrete wavelet transform (DWT). 
We impose in this context that $N_h = (N_{h,r}\times N_{h,c})  =  (2\,N_{H,r} \times 2\,N_{H,c})= 4 \times N_H$. 
For a generic quadrature mirror filter $\textbf{q} =(q_1,\ldots,q_m)$:
\begin{align}
\IhH := \left(\Res_{\mathbf{q},r} \otimes \Res_{\mathbf{q},c}\right),
\label{eq:info_transfer_wavelet}
\end{align} 
where $\Res_{\mathbf{q},c}$ is the decimated $N_{H,r}$-by-$N_{h,r}$ matrix (every other line is kept) of the $N_{h,r}$-by-$N_{h,r}$ Toeplitz matrix generated by $\mathbf{q}$ as :
\begin{equation*}
    \left(\begin{array}{ccccccc}
    q_1 & q_2 & \ldots & q_m & 0 & \ldots & 0 \\
    0 & 0 & q_1 & q_2 & \ldots & \ldots & 0 \\
    \vdots & \ddots & \ddots & \ddots & \ddots & \ddots &  \\
    0 & \ldots & 0 & 0 & 0 & q_1 & q_2   
    \end{array}\right).
\end{equation*}
Similarly $\Res_{\mathbf{q},r}$ is the decimated $N_{H,c}$-by-$N_{h,c}$ matrix (every other line is kept) of the $N_{h,c}$-by-$N_{h,c}$ Toeplitz matrix generated by $\mathbf{q}$. For both matrices the vector $\mathbf{q}$ is completed with the right number of $0$'s to reach the size $N_{h,r}$ or $N_{h,c}$. $\IHh$ is then taken in order to satisfy Definition~\ref{def:inf_transf_op}.

\subsection{Fast coarse models for image restoration problems}
A challenging numerical problem is to keep the efficiency of matrix-vector product computation at coarse level if it exists at fine level. 
For instance, when considering convolutions, if the convolution matrix is expressed with a Kronecker product, such structure can be preserved with the right definition of operators at coarse levels. 
\paragraph{$A_H$ in the deblurring problem} Thanks to the Kronecker factorization of both $\A_h$ and $\IhH$, the coarsened operator $\A_H$  can be written as:
\begin{equation*}
    \A_H = \left(\Res_{\mathbf{q},c} A_{h,r} \Res_{\mathbf{q},c}^T \right)\otimes \left(\Res_{\mathbf{q},r} A_{h,c} \Res_{\mathbf{q},r}^T \right)
\end{equation*}
preserving the same computational efficiency.
Thus in image restoration problems where a separable blur is used, it is straightforward to design coarse operators (which can be computed beforehand) that are fast for matrix-vector products while keeping fidelity to the fine level. 

\paragraph{$A_H$ in the inpainting problem} Due to the specific diagonal form of $\A_h$, the coarsened inpainting operator $\A_H$ simply stems from decimating the rows and the columns of $\A_h$ by a factor 2. $\A_H \in \RR^{N_H \times N_H}$ remains a diagonal indicator matrix of a pixel subset $J \subseteq \{1,\ldots,N_H\}$ acting as a mask on the coarse image: %
\begin{equation*}
    (\A_H x_H)^j = \left\{ \begin{array}{cc}
     x^j_H  & \text{if } j \in J \\
     0    & \text{otherwise}
    \end{array}\right.
\end{equation*}

\paragraph{Examples of operators $\D_H$} For the regularization operators, the construction is simpler. For both TV and NLTV, we use the same hyper-parameters (maximum number of patches, size of patches,  computation of similarity between patches, etc.) for $\D_H$ as for $\D_h$. Adapting these parameters to current resolution could be worth investigating. However, due to the limited size of the chosen patches, we believe it would lead to marginal improvements. %
$\D_H$ is thus playing the same role as $\D_h$ but for images of size $N_H$. 
Here $\D_H x_H$ belongs to $\RR^{N_H \times \widetilde{K}}$.

\subsection{Coarse model construction}\label{sec:corse_model_image}
For the coarse model it is natural in this context to choose $\Lo_H$ and $\R_H$ as 
\begin{align*}
    \Lo_H= f_H\circ\A_H, \quad \R_H = g_H\circ \D_H,
\end{align*}
where $\A_H$, $\D_H$ are defined as described above and $f_H$, $g_H$ are the restrictions of $f_h$ and $g_h$ to a subspace of reduced dimension. 
We then have:
\begin{equation*}
\begin{split}
    (\forall x_H \in \RR^{N_H}), & ~~f_H(\A_H x_H) = \frac{1}{2}\sum_{i=1}^{N_H} ((\A_H x_H)^i-(z_H)^i)^2,
    \label{eq:standard_least_squaresH} \\
    & ~~g_H(\D_H x_H) = \lambda_H\sum_{i=1}^{N_H}\Vert(\D_H x_H)^i \Vert_p.
\end{split}
\end{equation*}
Ideally, in order to speed up the computations, one would like to choose an approximation $\R_H$ whose proximity operator is known under closed form, even when $\R_h$ does not possess this desirable property. However, we have seen in our experiments that choosing an $\R_H$ not faithful to $\R_h$ deteriorates the performance of the multilevel algorithm (for instance, when $\R_h$ is the TV based norm, choosing a Haar wavelet based norm for $\R_H$ is sub-optimal, even though there is a link between Haar wavelet and total variation thresholdings. \cite{steidl2004,kamilov2012}).

This motivates the construction presented in Section \ref{sec:ml_framework} that we adapt here to our problem: we replace $\R_h$ and $\R_H$ by their corresponding smooth Moreau envelopes, which possess several interesting properties.
\begin{definition}{\emph{(Moreau envelope)}.} Let $\gamma>0$ and $\R\colon\RR^N \to(-\infty,+\infty]$ a convex, lower semi-continuous, and proper function. The Moreau envelope of $\R$, denoted by $\leftidx{^{\gamma}}{\R}$, is the convex, continuous, real-valued function defined by
\begin{equation}
    \leftidx{^{\gamma}}{\R} = \inf_{y \in \RR^N} \R(y) + \frac{1}{2 \gamma} \Vert \cdot -y \Vert^2.
\end{equation}
$\leftidx{^{\gamma}}{\R}$ can be expressed explicitly with $\prox_{\gamma \R}$ \cite[Remark 12.24]{bauschke2017} as follows:
\begin{equation*}
    \leftidx{^{\gamma}}{\R}{}(x) = \R(\prox_{\gamma \R}(x)) + \frac{1}{2\gamma} \Vert x - \prox_{\gamma \R}(x) \Vert^2.
\end{equation*}
Moreover, $\leftidx{^{\gamma}}{\R}$ is Fréchet differentiable on $\RR^N$, and its gradient is $\gamma^{-1}$-Lipschitz and such that %
\cite[Prop. 12.30]{bauschke2017}
\begin{equation}
    \nabla(\leftidx{^{\gamma}}{\R}) = \gamma^{-1}(\Id- \prox_{\gamma \R}).
    \label{gradient_moreau}
\end{equation}
\end{definition}
However, the last equation is not directly applicable because we assumed that the proximity operator of $g \circ \D$ had no explicit form.
Therefore, instead of directly using the Moreau envelope of $R$, we first compute the Moreau envelope of $g$ and compose it with $\D$. This smoothing satisfies Definition \ref{def:smooth_convex_func} :
\begin{lemma}
$\leftidx{^{\gamma}}{g} \circ \D$ is a smoothed convex function approximating $g \circ \D$ in the sense of Definition 
\ref{def:smooth_convex_func}.
\begin{proof}
    Remark that $\leftidx{^{\gamma}}{g}$ is a smooth convex function in the sense of Definition \ref{def:smooth_convex_func} \cite{beck2012}. By \cite[Lemma 2.2]{beck2012}, the fact that $\leftidx{^{\gamma}}{g} \circ \D$ is a smooth function applied to a linear transformation concludes the proof. 
\end{proof}
\end{lemma}
This smooth approximation has the following interesting property:
\begin{lemma}{\emph{\cite[Lemma 3.2]{luu2017}}}
For any $x \in \RR^N$, $\D: \RR^N \rightarrow \RR^K$ and $g: \RR^K \rightarrow \RR$ a convex, l.s.c., and proper function, we have that: 
\begin{equation}
    \nabla \left( \leftidx{^{\gamma}}{g} \circ \D \right)(x) = \gamma^{-1} \D^* \left(\D\vspace{0.1em}x - \prox_{\gamma g}(\D\vspace{0.1em}x) \right).\vspace{-0.5em}
\end{equation}
\end{lemma}
This means that an explicit form of $\prox_{\gamma_h g_h}$ is sufficient to express the gradient of $\leftidx{^{\gamma_h}}{g}_h(\D_h \cdot)$. Accordingly, we define the following coarse model, 
where the first order coherence is enforced between the two objective functions,  smoothed similarly at fine and coarse levels:
\begin{definition}
A coarse model for the image restoration problem \eqref{eq:optim_article} is defined at iteration $k$ of a multilevel algorithm as:
\begin{equation}
    F_{H}(s_H) = (f_H\circ \A_H) (s_H) + \left(\leftidx{^{\gamma_H}}{g}_H \circ \D_H \right)(s_H) + \langle \vHk, s_H \rangle,
\end{equation}
where $\vHk$ will be set to:
\begin{equation*}
\label{eq:v_Hk_inexact1}
\vHk =  \IhH \left[\left(\nabla (f_h\circ\A_h) + \nabla (\leftidx{^{\gamma_h}}{g}{}_h\circ \D_h\right)(\yhk)\right]  -(\nabla (f_H\circ\A_H) + \nabla(\leftidx{^{\gamma_H}}{g}{}_H\circ\D_H))(\xHO).
\end{equation*}
\end{definition}

\subsection{Computation of the proximity operator of $g_h \circ \D_h$}
If $\D_h$ is the projection on a tight frame (e.g., a union of wavelets), meaning that $\D_h \D_h^* = \mu \Id$ for a constant $\mu>0$ and $\D_h^*$ the adjoint of $\D_h$, the proximity operator of $g_h \circ \D$ is expressed explicitly through the proximity operator of $g_h$, which is known in a large number of cases. 

Otherwise, a common way of estimating the proximity operator is through the \textit{dual problem}. 
Denoting $\R_h =g_h \circ \D_h$, we have that (see for instance \cite{le2022}):\vspace{-0.2cm}
\begin{equation}
\label{eq:prox_optim}
    (\forall x \in \RR^{N_h}) \qquad \prox_{\gamma \R_h}(x):= \prox_{\gamma g_h \circ \D_h}(x) = x - \D_h^* \widehat{u}\vspace{-0.3cm}
\end{equation}
with:
\vspace{-0.3cm}
\begin{equation}
    \widehat{u} \in \argmin_{u\in\RR^{K} } \frac{1}{2} \Vert \D_h^* u -x\Vert^2 + \gamma g_h^*(u),
\label{eq:dual_prox}
\end{equation}
\vspace{-0.3cm}

\noindent where $g_h^*$ is the convex conjugate of $g_h$. This problem is known as the dual problem. 
An approximation of $\widehat{u}$ may be obtained by applying any convenient optimization method to \eqref{eq:dual_prox}. For instance, FISTA  yields the following sequence (choosing $u_0=v_0$):
\begin{align}
    u_{k+1} &= \left(\Id-\gamma \prox_{g_h/\gamma}\left(\cdot/\gamma\right)\right)\left((\mbox{Id}-\D_h \D_h^*)v_{k} + \gamma_h \D_h x\right) \label{eq:dual_prox_steps} \\
    v_{k+1} &= (1+\alpha_k)u_{k+1} -\alpha_k u_k.
    \label{eq:dual_prox_steps_iner}
\end{align}
where the first step is deduced from the Moreau decomposition \cite{bauschke2017}.
Dual optimization is a simple way to estimate the proximity operator while offering guarantees on the computed approximation, as stated in the following lemma.  
\begin{proposition}{\emph{(Dual optimization yields approximation of type 2)}}
Assume that $(u_k)_{k \in \mathbb{N}}$ is a minimizing sequence for the dual function in \eqref{eq:dual_prox}. This yields:
\begin{itemize}
    \item A convergent sequence $(x-\D_h^* u_k)_{k \in \mathbb{N}}$ to the proximity operator \eqref{eq:prox_optim}.
    \item This sequence provides a type 2 approximation of the proximity operator.
\end{itemize}
\begin{proof}
The first point comes from \cite[Theorem 5.1]{villa2013}. Then the approximation of type 2 comes from \cite[Proposition 2.2, and 2.3]{villa2013}.
\end{proof}
\end{proposition}

\section{Experimental results}
\label{sec:results}
The objective of this section is to illustrate the benefits of the proposed IML FISTA in various 
image reconstruction tasks,  particularly when they involve 
large-scale images. We show that FISTA and IML FISTA both converge to the same solution but IML FISTA always converges faster, ensuring a good reconstruction in few iterations and thus providing a method of considerable interest for large-scale imaging applications. Code and examples are available here\footnote{\href{https://laugaguillaume.github.io}{https://laugaguillaume.github.io}}. 

\subsection{Experimental setting for color images reconstruction}

\paragraph{Degradation types} We consider two types of image reconstruction problems: a restoration problem where the  linear operator $\A$ is a Gaussian blur, and an inpainting problem where $\A$ models the action of random pixel deletion. 
In all cases, we consider an additive white Gaussian noise with standard deviation $\sigma$.

\paragraph{Minimization problem}
At fine level, we consider the state-of-the-art optimization problem in this context, the minimization of the sum of a quadratic data-fidelity term and a sparsity prior based on a total variation $\ell_{1,2}$-norm (isotropic total variation):
\begin{equation}\label{eq:fine_pb}
    (\forall x \in \RR^{N_h}),\qquad F_h(x) = \frac{1}{2}\Vert \A_h x - z_h \Vert_2^2 +  \lambda_h\Vert \D_h x \Vert_{1,2},
\end{equation}
with $\lambda_h>0$. In all the experiments, the regularisation parameter $\lambda_h$ was chosen by a grid search, in order to maximize the SNR of $\widehat{x}$ computed by FISTA at convergence. Finally, we choose as initialization $x_{h,0}$, the Wiener filtering of $z$.

\paragraph{Experiment datasets} 
We consider two color images of different sizes to evaluate the impact of the problem's dimension: ``ImageNet Car" the picture of a yellow car of size $512 \times 512 \times 3$, taken from the ImageNet dataset, and a picture taken by the James Webb Space Telescope with its Near-Infrared Camera and its Mid-Infrared Instrument of the structure called ``Pillars of Creation" of size $2048 \times 2048 \times 3$ (Figure \ref{fig:images_originales}). Pixels values are normalized so that the maximum value across all channels is $1$.%

\begin{figure}
    \centering
    \includegraphics[width = 4.67cm]{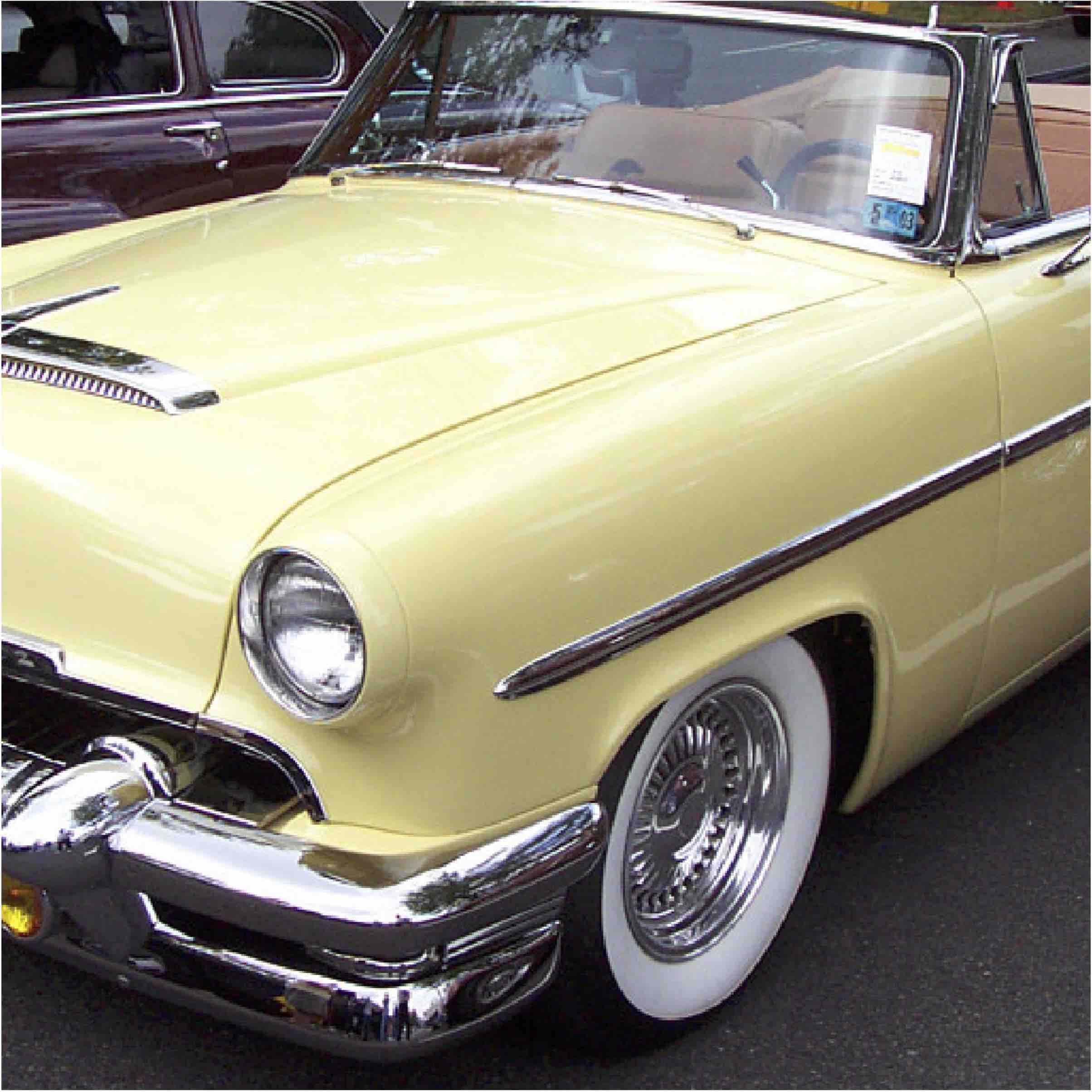}
    \hspace{2em}
    \includegraphics[width = 5cm]{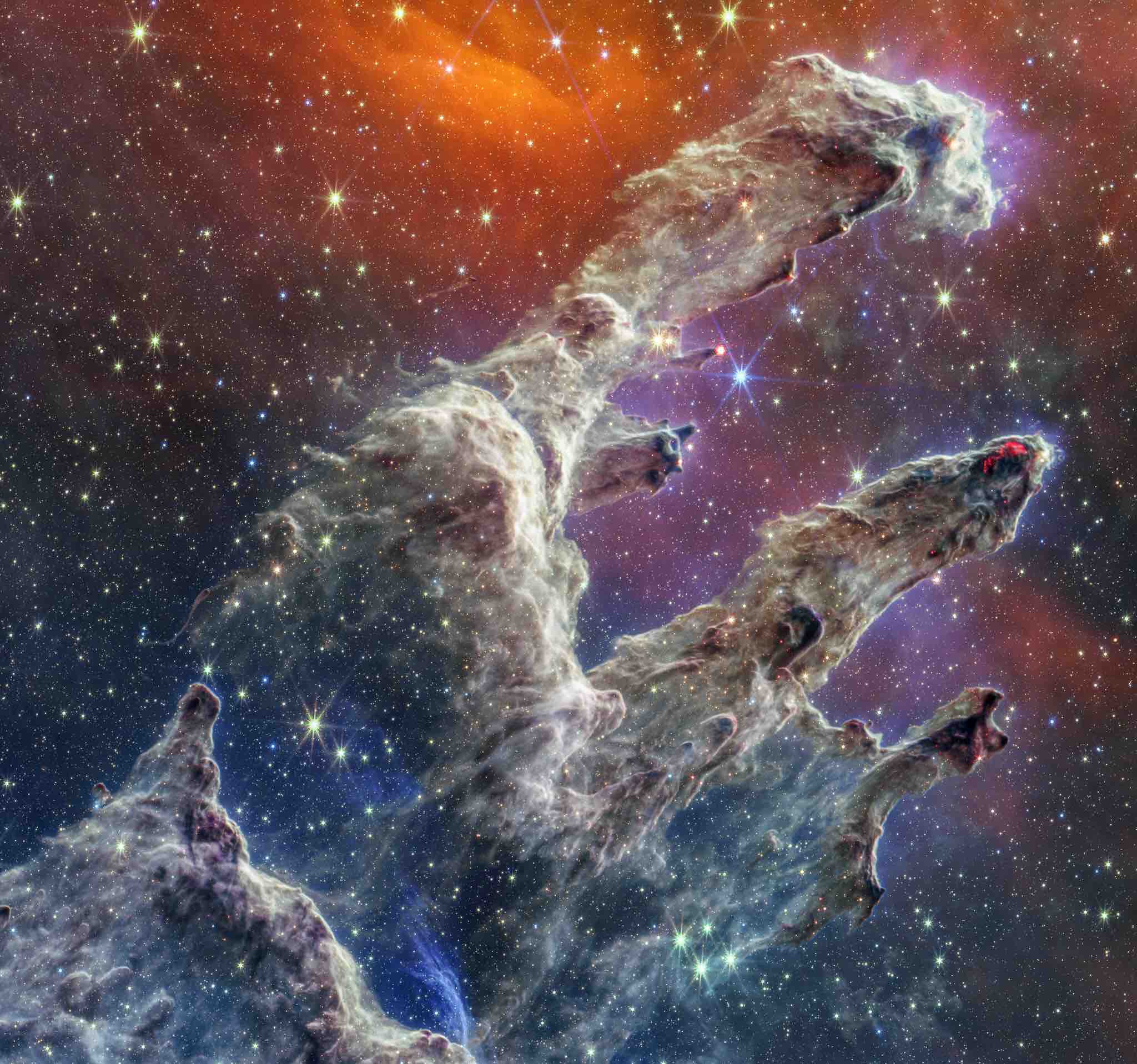}
    \caption[Caption for LOF]{ImageNet Car "ILSVRC2012$\_$test$\_$00000164"\footnotemark[1]. Pillars of Creation\footnotemark[2]. Credits: SCIENCE: NASA, ESA, CSA, STScI (Image processing): Joseph DePasquale (STScI), Alyssa Pagan (STScI), Anton M. Koekemoer (STScI). \vspace{-2em}}
    \label{fig:images_originales}
\end{figure}

\paragraph{Multilevel structure}
For all our experiments we use a 5-levels hierarchy. For ``Pillars of Creation" the first level corresponds to an image of size $2048 \times 2048 \times 3$, and the fifth level to an image of size $ 128 \times 128 \times 3 $. Similarly for ``ImageNet Car" the first level corresponds to an image of size $512 \times 512 \times 3$ and the fifth level to an image of size $32 \times 32 \times 3$. 

The coarse model associated to \eqref{eq:fine_pb} is written as:
\begin{equation}
    (\forall x \in \RR^{N_H}),\qquad F_H(x) = \frac{1}{2}\Vert \A_H x - z_H \Vert_2^2 + \lambda_H\left(\leftidx{^{\gamma_H}}{g}_H \left(\D_H x\right) \right) + \langle v_H, x \rangle,
\end{equation}
with $\lambda_H>0$, $z_H = \IhH z_h$ and $g_H$ the $\ell_{1,2}$-norm applied on the $N_H$ components of $\D_H x$, as for the fine level. 
    As the dimension of the problem is reduced by a factor 4 every time we lower the resolution, we set the regularization parameter $\lambda_H$ at coarse level to to a quarter of the value of the regularisation parameter at the next higher level. %
    In practice, this ratio gives the best performance in terms of decrease of the fine level objective function. 
The CIT operators were built for every pair of levels with ``Symlets 10" wavelets corresponding to a filter size of $20$ coefficients.

Based on our previous study in \cite{lauga2022_1}, we always impose $p=2$ %
coarse corrections (V-cycles) with $m=5$ iterations per level, and always performed at the beginning of the optimization process, as this configuration showed to perform well for different levels of degradation. 
This appears to be a common setting in the multilevel literature \cite{parpas2016,parpas2017,lauga2022,lauga2022_1,plier2021,fung2020,javaherian2017,ho2019}. Increasing the number of coarse corrections may be occasionally beneficial, while sometimes it decreases the potential gain. Being difficult to know this without solving several times the same optimization problem, 
we deem more important to use and display a configuration that is satisfactory regardless of the specific problem parameters. 
\footnotetext[1]{\href{https://www.kaggle.com/competitions/imagenet-object-localization-challenge/data}{https://www.kaggle.com/competitions/imagenet-object-localization-challenge/data}}
\footnotetext[2]{\href{https://webbtelescope.org/contents/media/images/01GK2KKTR81SGYF24YBGYG7TAP.html}{https://webbtelescope.org/contents/media/images/01GK2KKTR81SGYF24YBGYG7TAP.html}}
\paragraph{Accuracy of the computation of the proximity operator}
Convergence guarantees of the algorithm are directly linked to the decrease of the error introduced by estimating the proximity operator at each iteration. The necessary (see Theorems \ref{th:convergence_0} and \ref{th:convergence_12})  %
speed decrease depends on the choice of $d$ (Step \ref{alg_step:extrapolation} in Algorithm \ref{alg:MMiFB_general}) and on the type of approximation we are using. Indeed, based on the convergence result derived earlier (Theorem \ref{th:convergence_12}), going from $d=1$ to $d=0$ relaxes the speed decrease. In all cases, a lower error is correlated with a higher computational cost, which is why some strategies rather use a fixed budget of sub-iterations to compute the proximity operator \cite{beck2009}. This fixed budget comes at the cost of a limited precision on the estimated solution and may lead to divergence after a large number of iterations. 

This problem was notably addressed in \cite{machart2012}, where the authors introduced the Speedy Inexact Proximal-Gradient Strategy (SIP). %
The number of sub-iterations used to estimate the proximity operator is dynamically increased. More precisely, if at step $k$, $F_h(\xhk) > F_h(x_{h,k-1})$, we decrease the tolerance ($tol$) on the estimation of the proximity operator at the next steps $k+1,k+2,\ldots$ as $tol$ controls the relative distance between two consecutive sub-iterates of the proximity operator estimation.

\begin{algorithm}
\caption{Accuracy of the proximity operator estimation}\label{alg:prox_estimation}
\begin{algorithmic}[1]
\STATE Set $x_{h,0} \in \RR^N$,
\FOR {$k=0,1,\ldots,$}
    \IF {$F_h(\xhk)>F_h(x_{h,k-1})$} 
    \STATE $tol = tol/10$
    \ENDIF
\ENDFOR
\end{algorithmic}

\end{algorithm}
This minimization is carried out by FISTA coupled with a warm start strategy as in \cite{beck2009}. We set the initialization value of $tol$ based on the reconstruction quality of images in a Total Variation based denoising problem (that is equivalent to one computation of the associated proximity operator). $tol=10^{-8}$ at the start of the optimization unless stated otherwise.

\subsection{IML FISTA results on image restoration: deblurring}

To get a full picture of the performance of IML FISTA, we propose four scenarios, corresponding to four different combinations of the size of the Gaussian blur PSF and of the value of the standard deviation $\sigma$(noise) of the Gaussian noise. These four scenarios are described in Table~\ref{tab:scenarios_blur}.

\begin{table}[h]
\centering
\begin{tabular}{c|c|c|}
\cline{2-3}
\textbf{Blur $\backslash$ Noise} & $\sigma$(noise) $= 0.01$ & $\sigma$(noise) $= 0.05$ \\ \hline
\multicolumn{1}{|l|}{dim(PSF) $=20$, $\sigma$(PSF) $= 3.6$} & low blur, low noise & low blur, high noise \\ \hline
\multicolumn{1}{|l|}{dim(PSF) $=40$, $\sigma$(PSF) $= 7.3$} &   high blur, low noise & high blur, high noise \\ \hline
\end{tabular}
\caption{Four scenarios of Gaussian blur degradation with additive Gaussian noise.}
\label{tab:scenarios_blur}
\end{table}
\paragraph{FB/FISTA vs IML FB/FISTA} This first set of experiments allows us to compare several formulations of IML-FISTA, including its particular instances FB and FISTA. Algorithm~\ref{alg:MMiFB_general} can take the form of 
\begin{itemize}
    \item FB when $d=0$ and $\mbox{ML}(\yhk)=\yhk$,
    \item IML FB when $d=0$ and $\mbox{ML}(\yhk)=\bar{y}_{h,k}$,
     \item FISTA when $d=1$ and $\mbox{ML}(\yhk)=\yhk$,
    \item IML FISTA when $d=1$ and $\mbox{ML}(\yhk)=\bar{y}_{h,k}$.  
\end{itemize}
In Figure~\ref{fig:FBvsIMLFISTA}, we focus on the top left corner degradation configuration (Table \ref{tab:scenarios_blur}) and display the evolution of the objective function w.r.t. the CPU time for the four versions of Algorithm~\ref{alg:MMiFB_general}. We observe that IML FB (resp. FISTA) converges faster than FB (resp. FISTA) and additionally,  it confirms that FISTA and IML FISTA outperform forward-backward approaches without inertial steps.  In the following experiments, we focus on FISTA and IML FISTA comparisons.

\begin{figure}
    \centering
    \includegraphics[trim={2.5em 2.2em 0 2.6em},clip,width=0.5\textwidth]{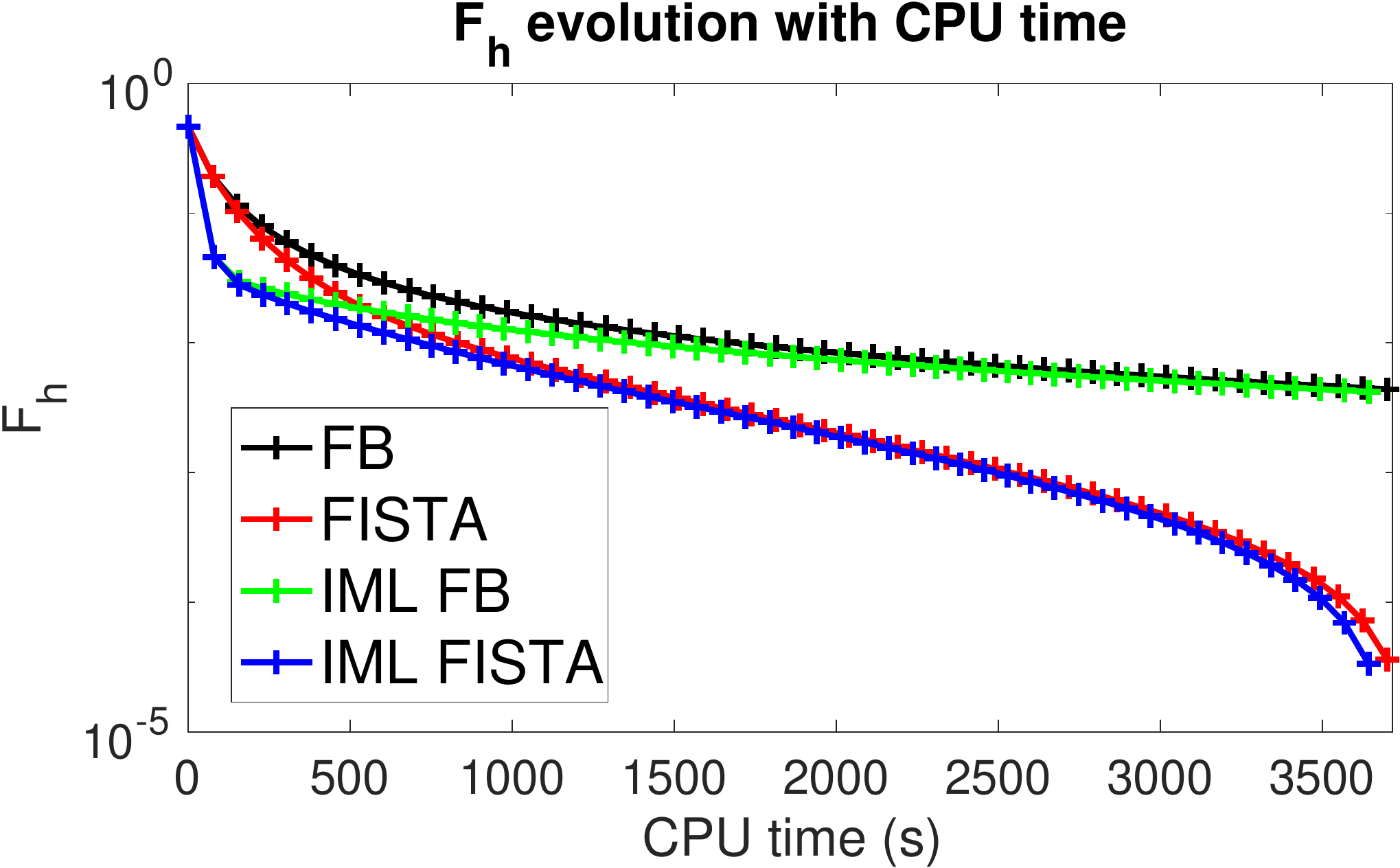}
    \caption{Comparison of FB and FISTA against their multilevel counterpart constructed with our framework, IML FB and IML FISTA for the restoration $\ell_{1,2}$-TV problem for the Pillars of Creation image (see top left corner Table \ref{tab:scenarios_blur}). To put the emphasis on the performance's difference between these algorithms, the objective function evolution is displayed in a log scale between the initial value and the minimum value obtained by these four algorithms in 50 iterations. \vspace{-1em}}
    \label{fig:FBvsIMLFISTA}
\end{figure}
\begin{figure}
    \centering
    \includegraphics[trim={2.5em 2.2em 0 2.7em},clip,width=0.4\textwidth]{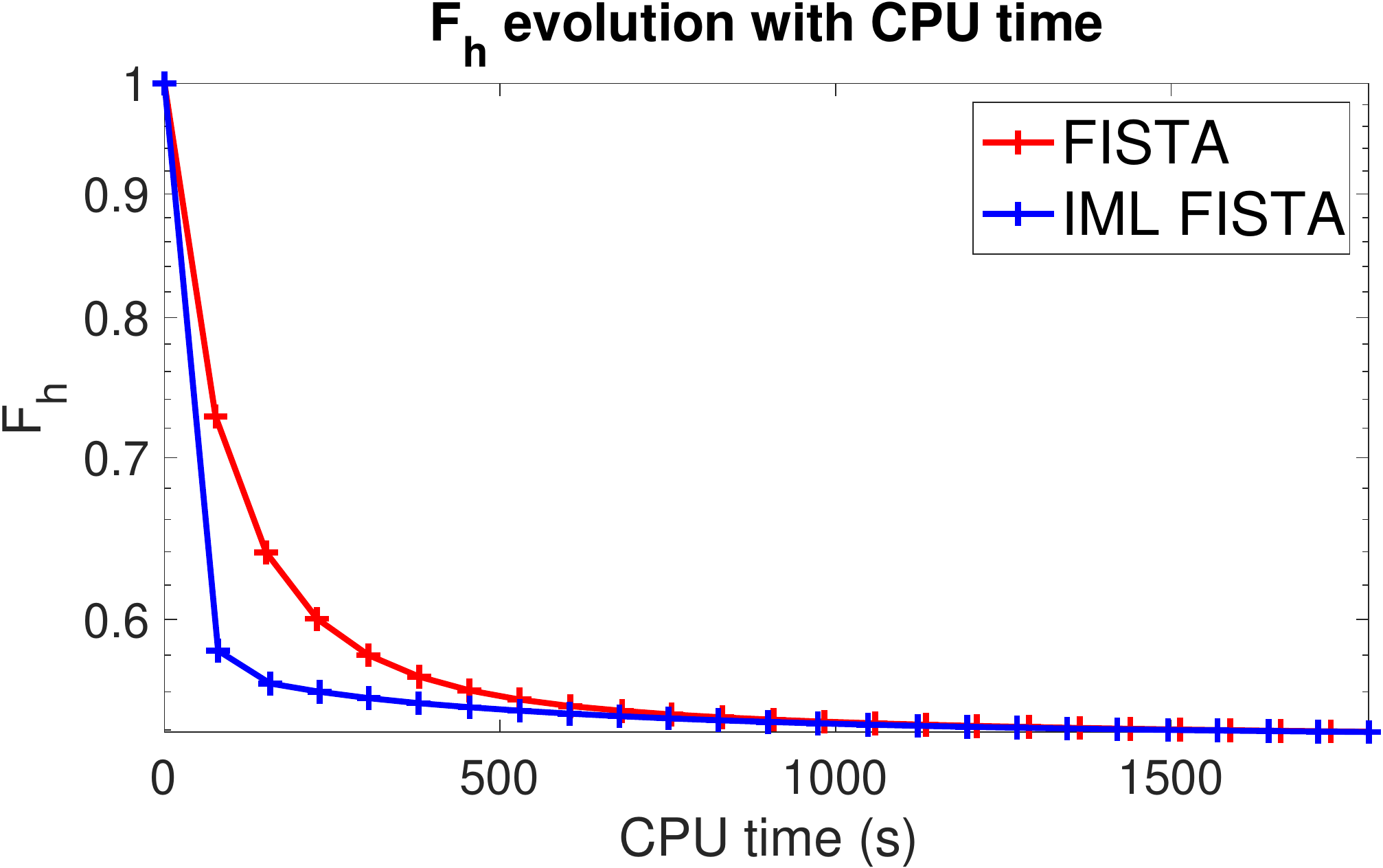}
    \hspace{2em}
    \includegraphics[trim={2.5em 2em 0 2.7em},clip,width=0.4\textwidth]{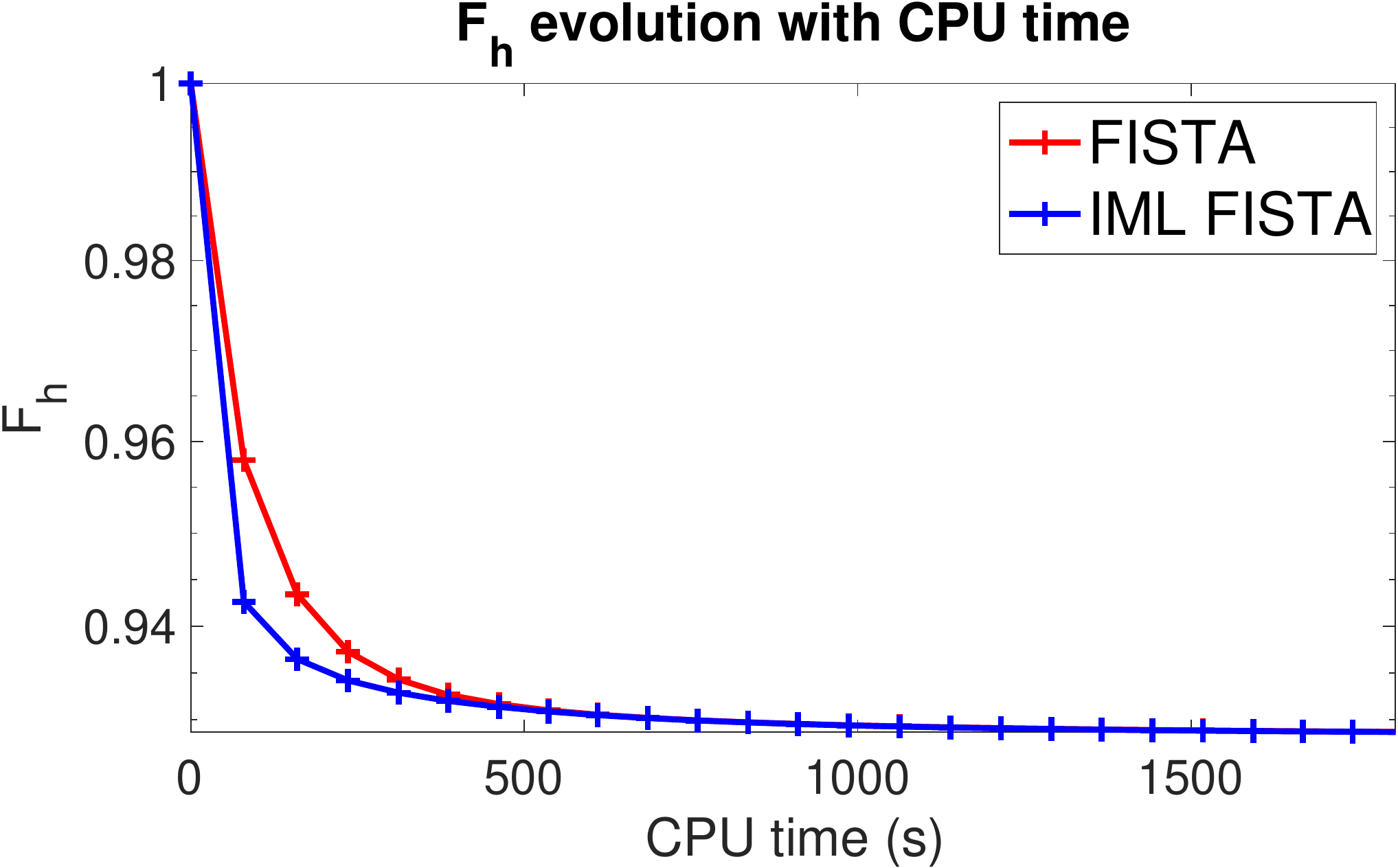} \\ \vspace{0.1em}
    \includegraphics[trim={2.5em 2.2em 0 2.7em},clip,width=0.4\textwidth]{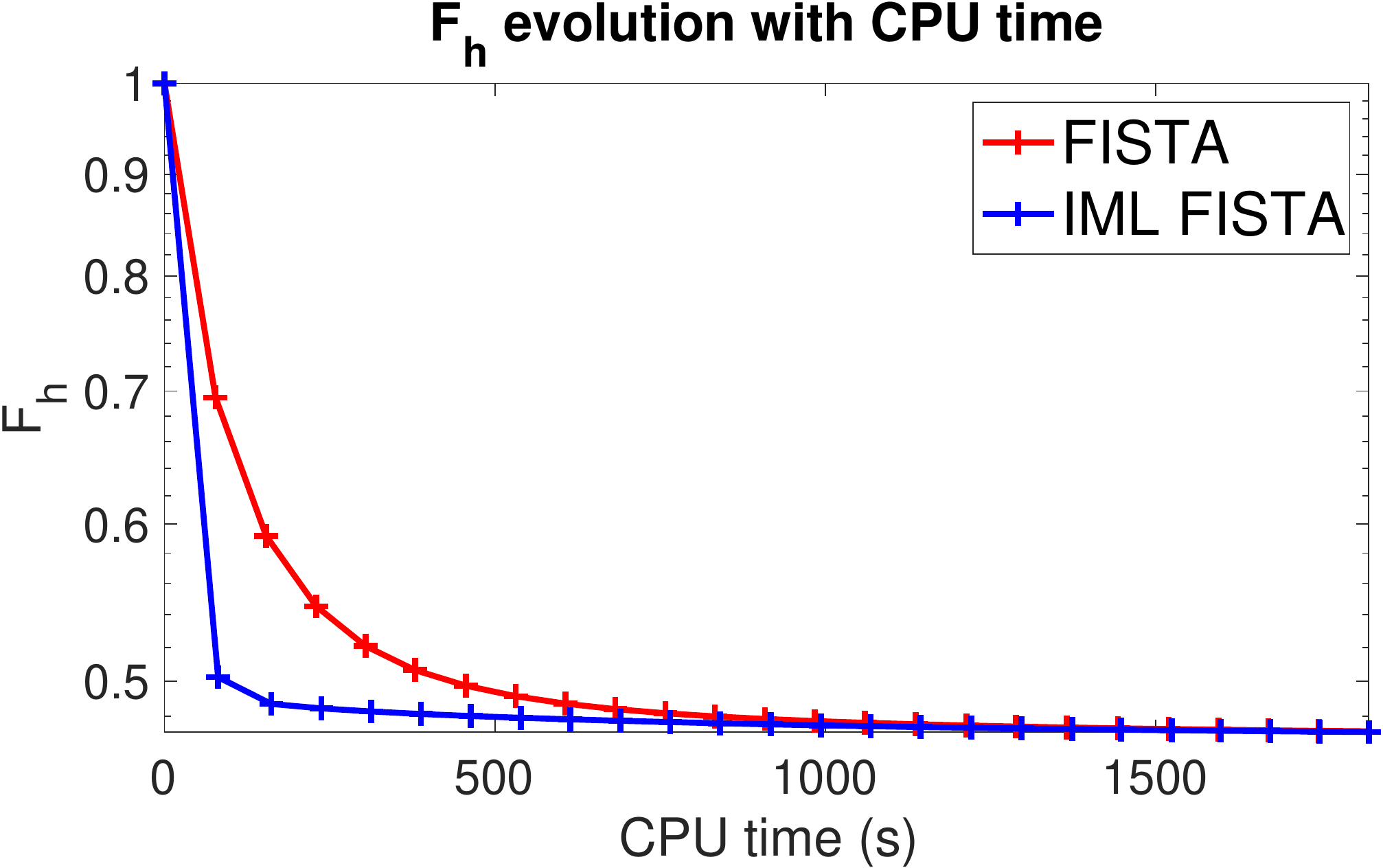}
    \hspace{2em}
    \includegraphics[trim={2.5em 2em 0 2.7em},clip,width=0.4\textwidth]{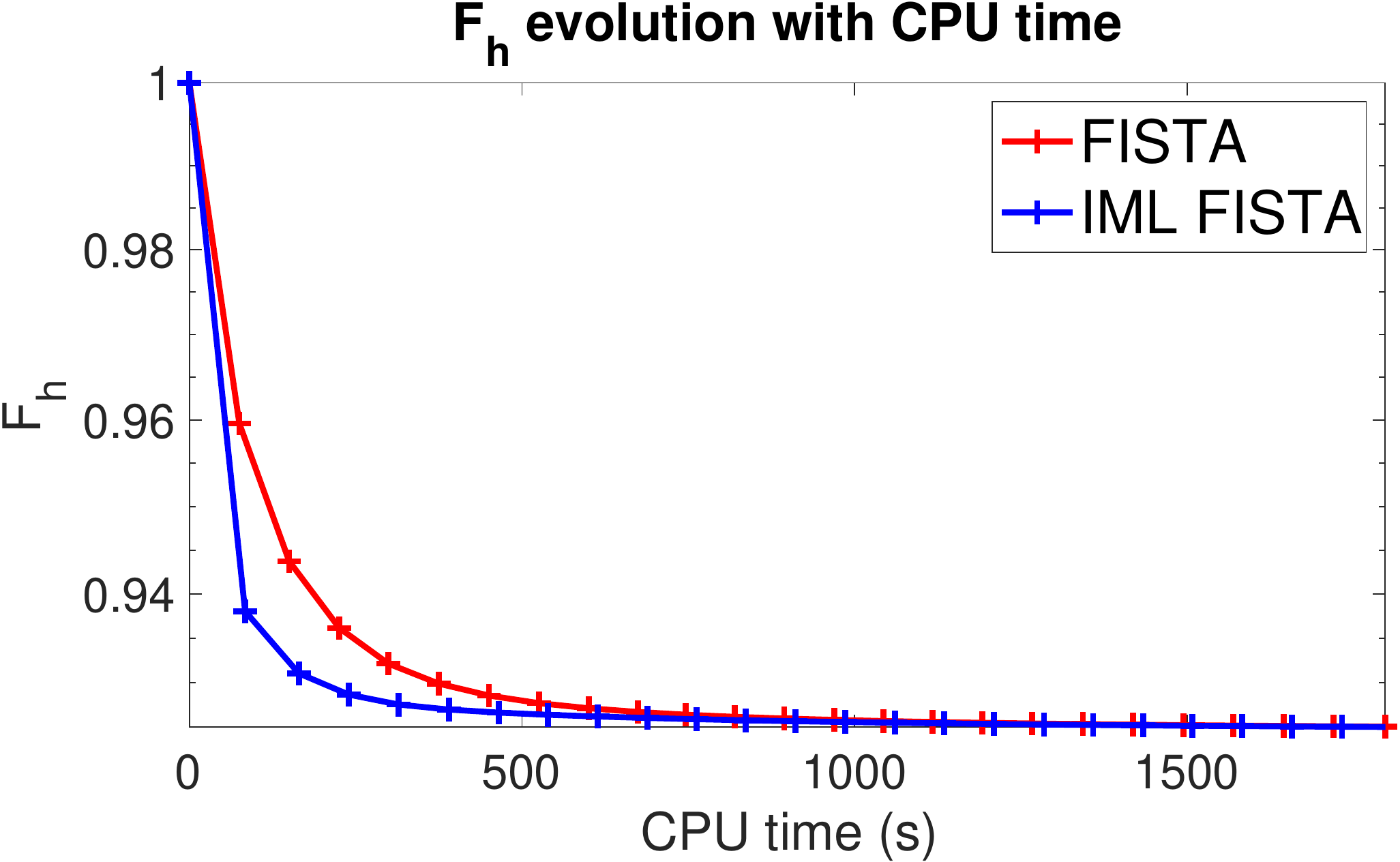}
    \caption{Deblurring $\ell_{1,2}$-TV for the Pillars of Creation image. Objective function (normalized w.r.t. the initial value) vs CPU time (sec). First column: $\sigma$(noise) $= 0.01$; second column: $\sigma$(noise) $= 0.05$. First row: dim(PSF) $=20$, $\sigma$(PSF) $= 3.6$; second row: dim(PSF) $=40$, $\sigma$(PSF) $= 7.3$. For each plot, the crosses represent iterations of the algorithm.} 
    \label{fig:TV_deblurring_FHIR}
\end{figure}
\begin{figure}
    \begin{center}
        \setlength{\tabcolsep}{3pt}
        \begin{tabular}{cccccc}
        \ftn$x$ & \ftn$x_{h,2}^{\text{FISTA}}(16.6)$ & \ftn$x_{h,50}^{\text{FISTA}}(16.8)$ & \ftn$x$ & \ftn$x_{h,2}^{\text{FISTA}}(16.4)$ & \ftn$x_{h,50}^{\text{FISTA}}(16.3)$ \\
         \includegraphics[trim={0 0 0 0},clip,width=0.15\textwidth]{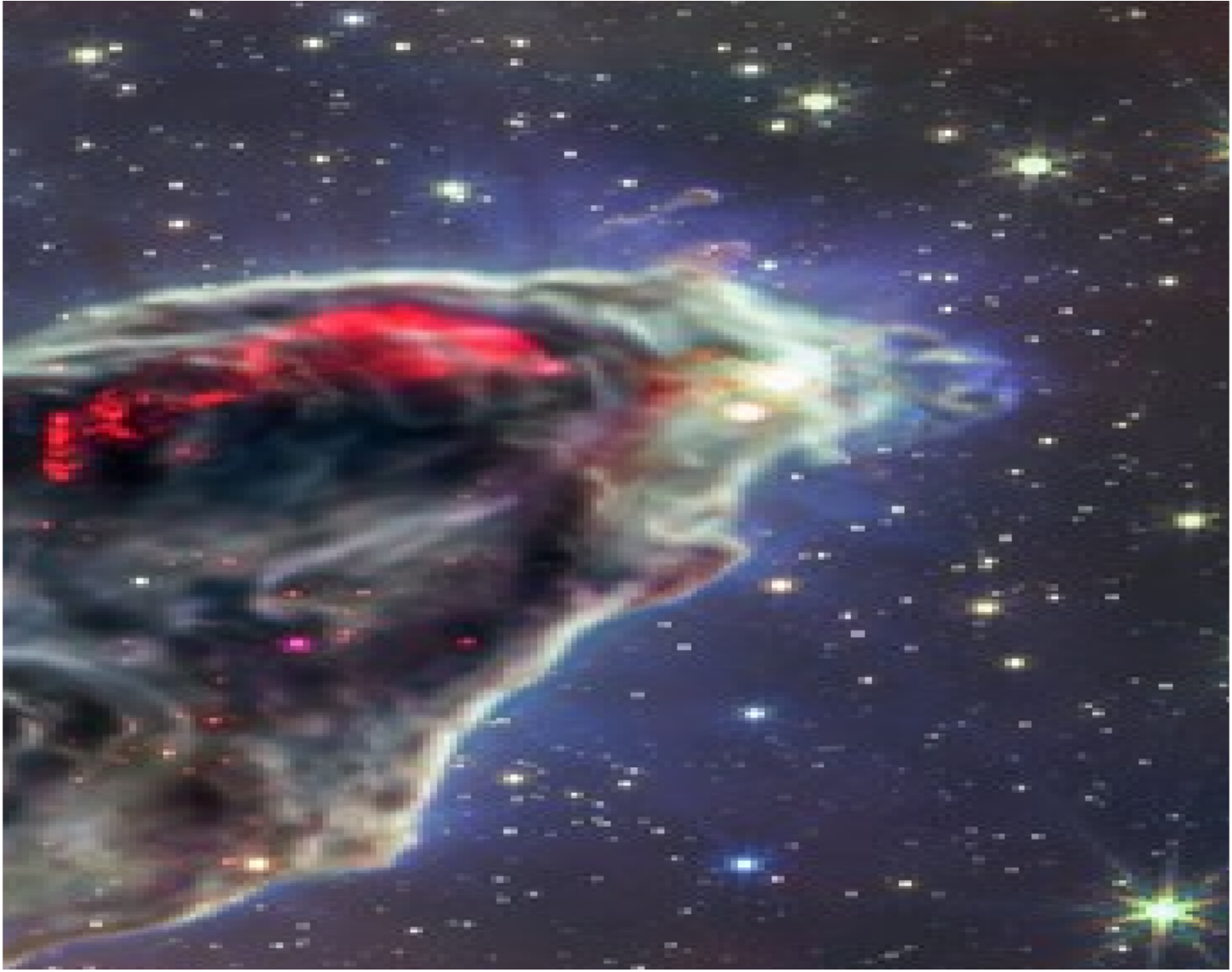} & \includegraphics[trim={0 0 0 0},clip,width=0.15\textwidth]{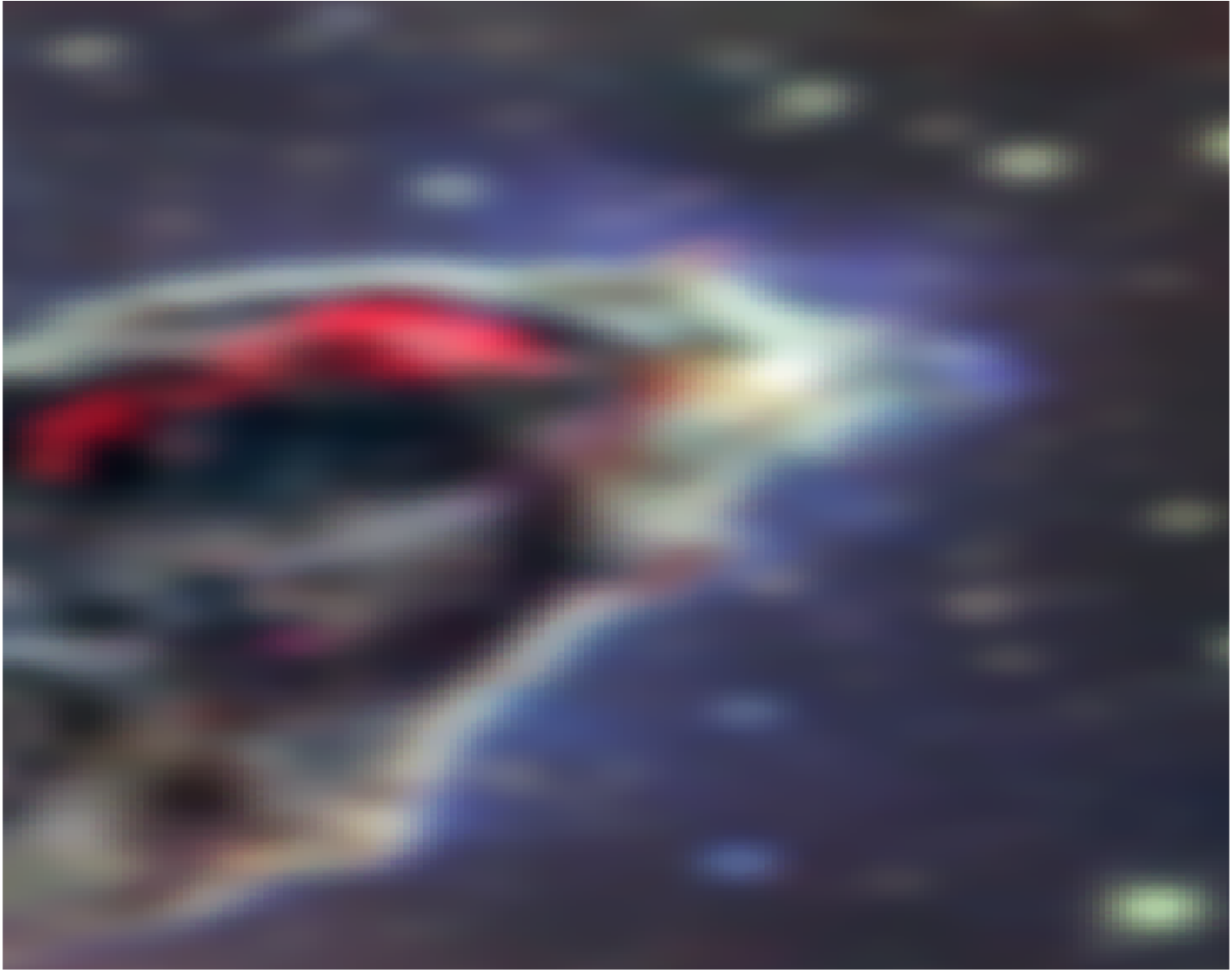} & \includegraphics[trim={0 0 0 0},clip,width=0.15\textwidth]{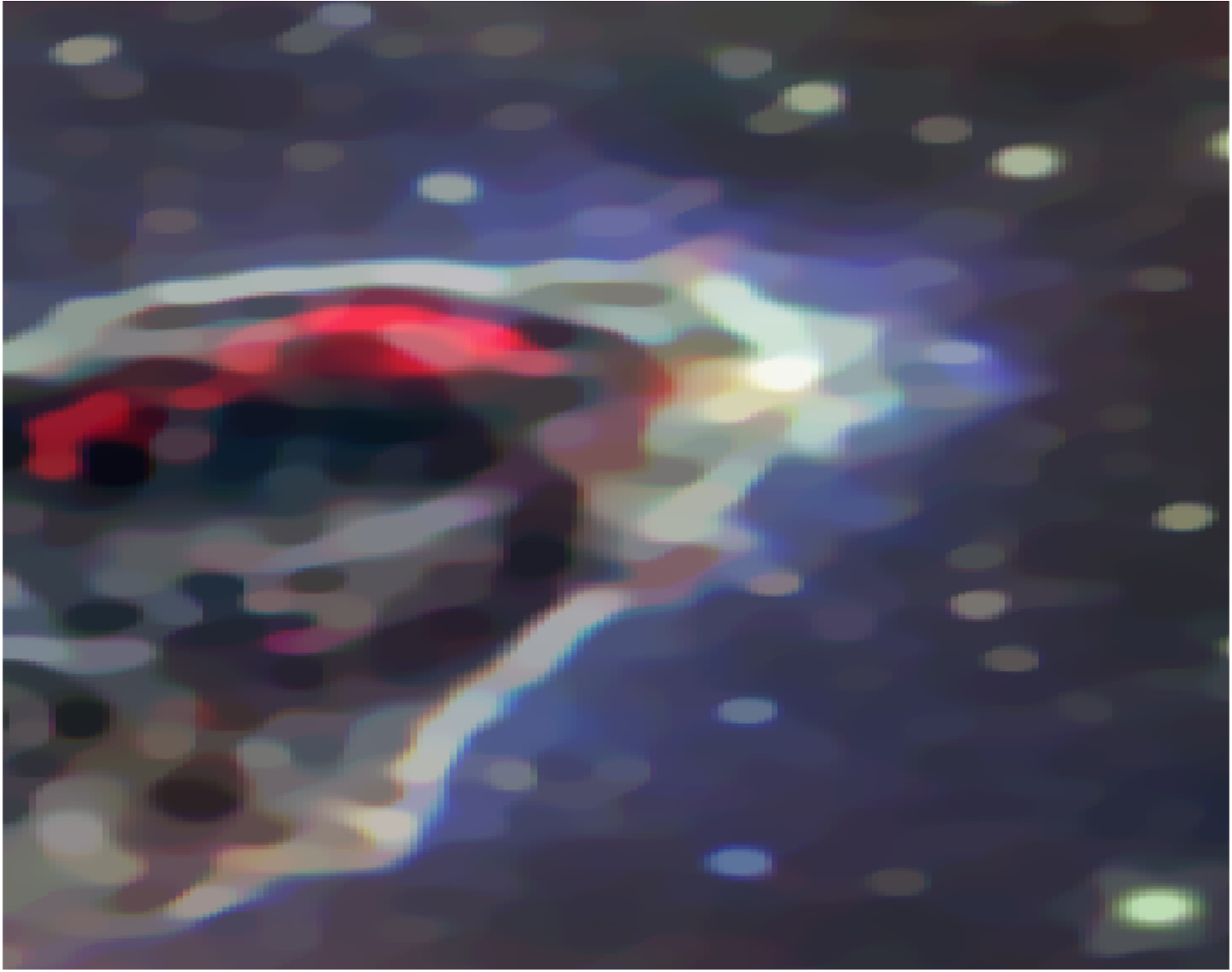} & \includegraphics[trim={0 0 0 0},clip,width=0.15\textwidth]{X-crop.pdf} 
         & \includegraphics[trim={0 0 0 0},clip,width=0.15\textwidth]{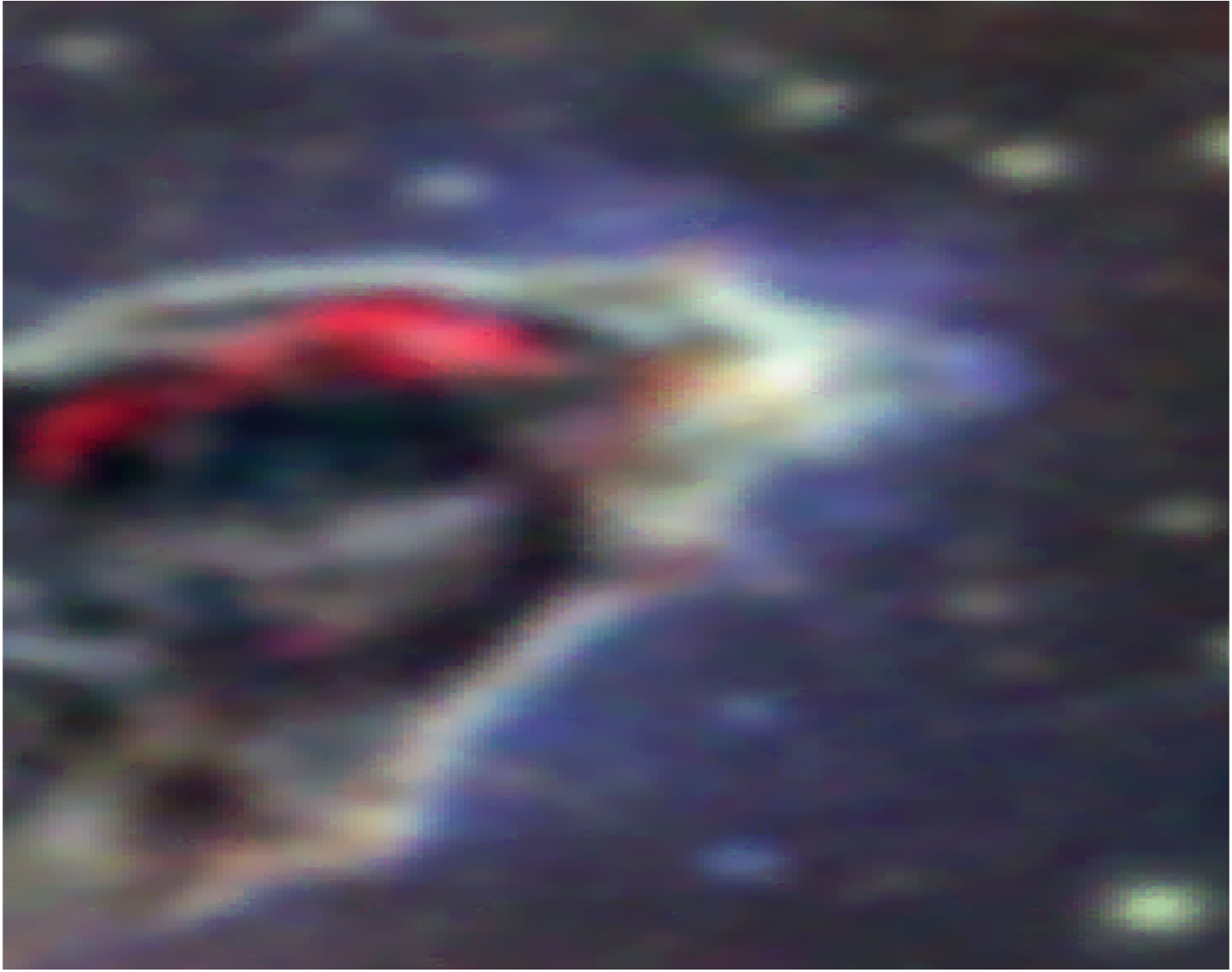} & \includegraphics[trim={0 0 0 0},clip,width=0.15\textwidth]{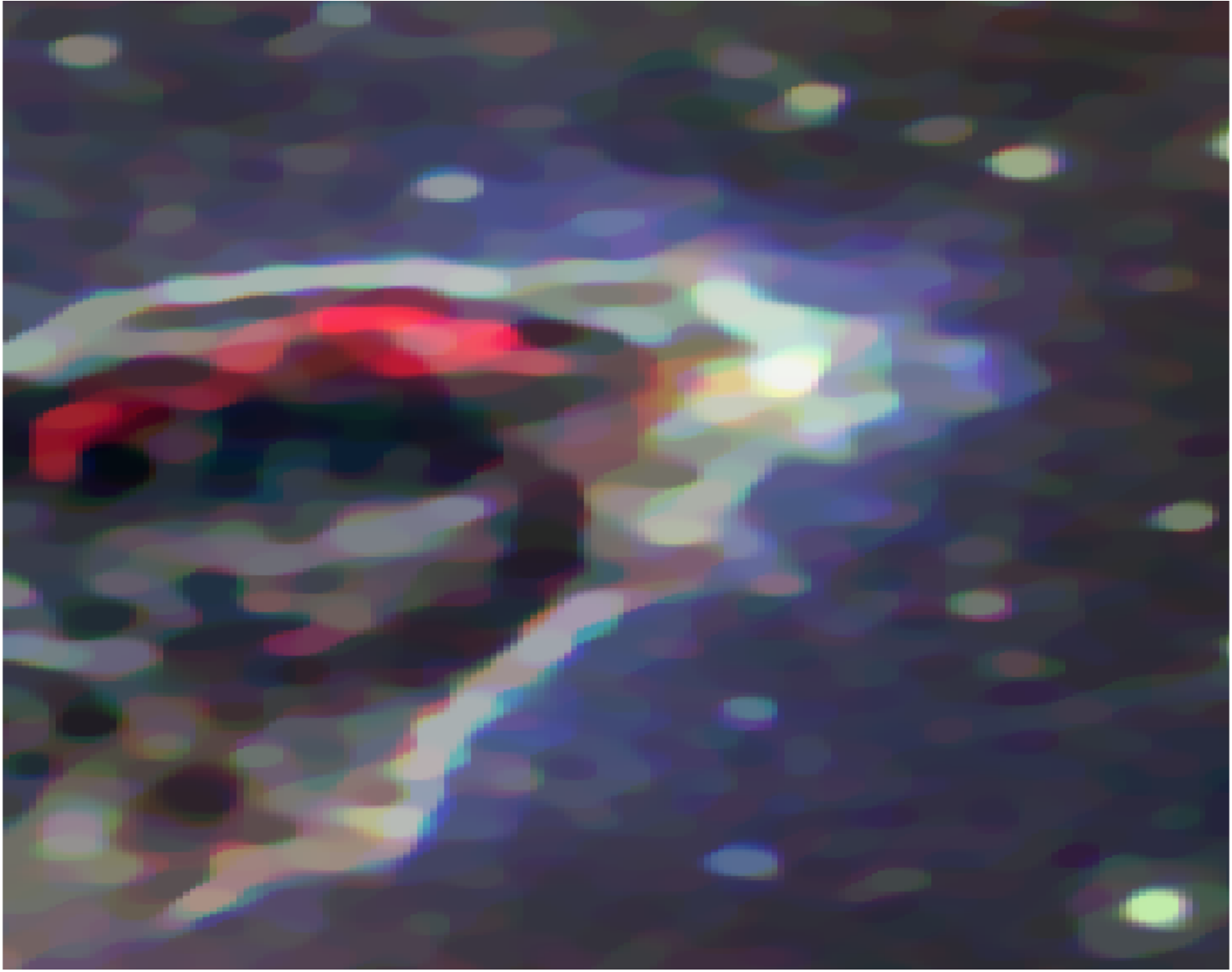}
         \\ 
         \ftn$z(16.3)$ & \ftn$x_{h,2}^{\text{IML FISTA}}(16.8)$ & \ftn$x_{h,50}^{\text{IML FISTA}}(16.8)$ & \ftn$z(16)$ & \ftn$x_{h,2}^{\text{IML FISTA}}(16.5)$ & \ftn$x_{h,50}^{\text{IML FISTA}}(16.3)$ 
         \\
         \includegraphics[trim={0 0 0 0},clip,width=0.15\textwidth]{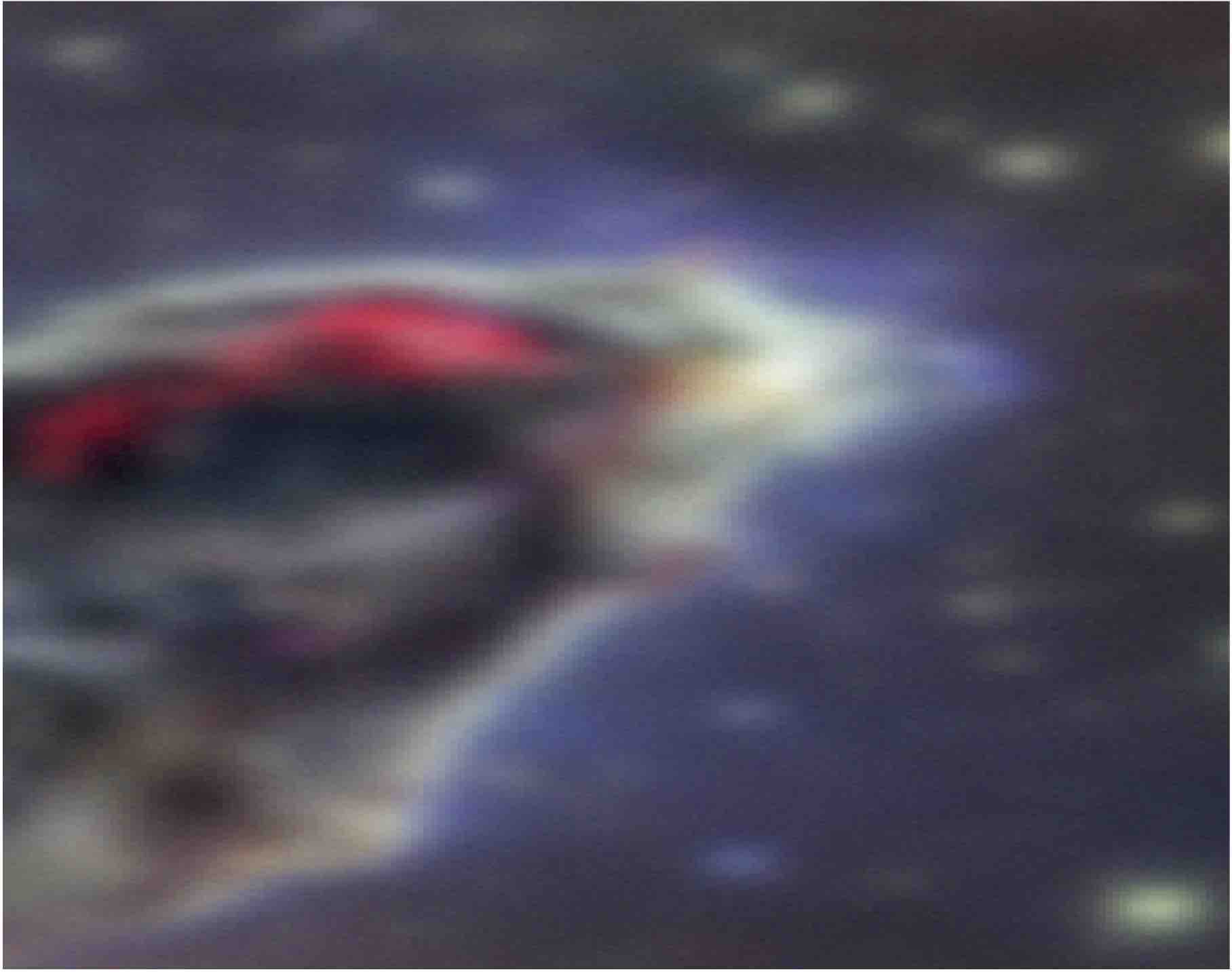} & \includegraphics[trim={0 0 0 0},clip,width=0.15\textwidth]{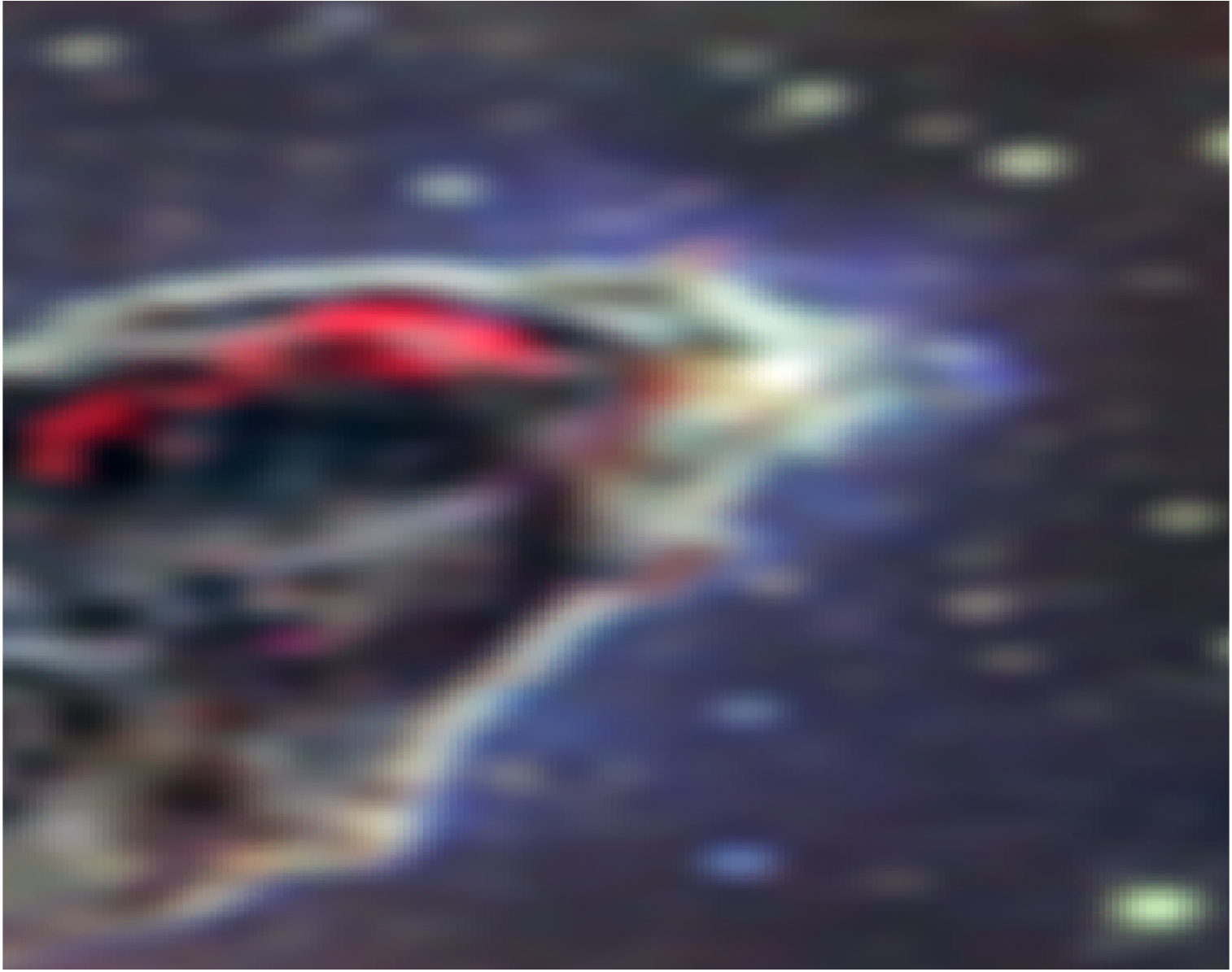} & \includegraphics[trim={0 0 0 0},clip,width=0.15\textwidth]{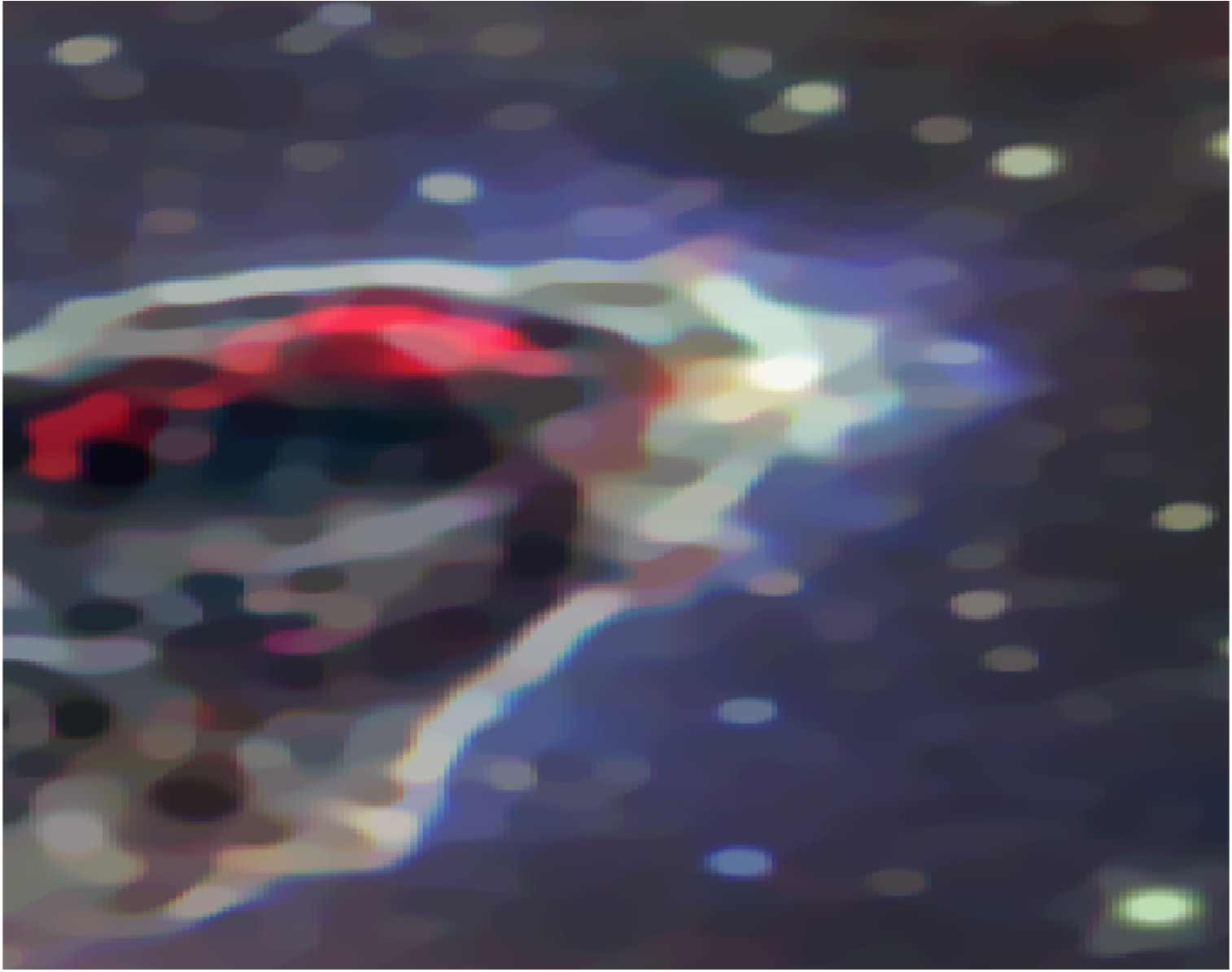} & \includegraphics[trim={0 0 0 0},clip,width=0.15\textwidth]{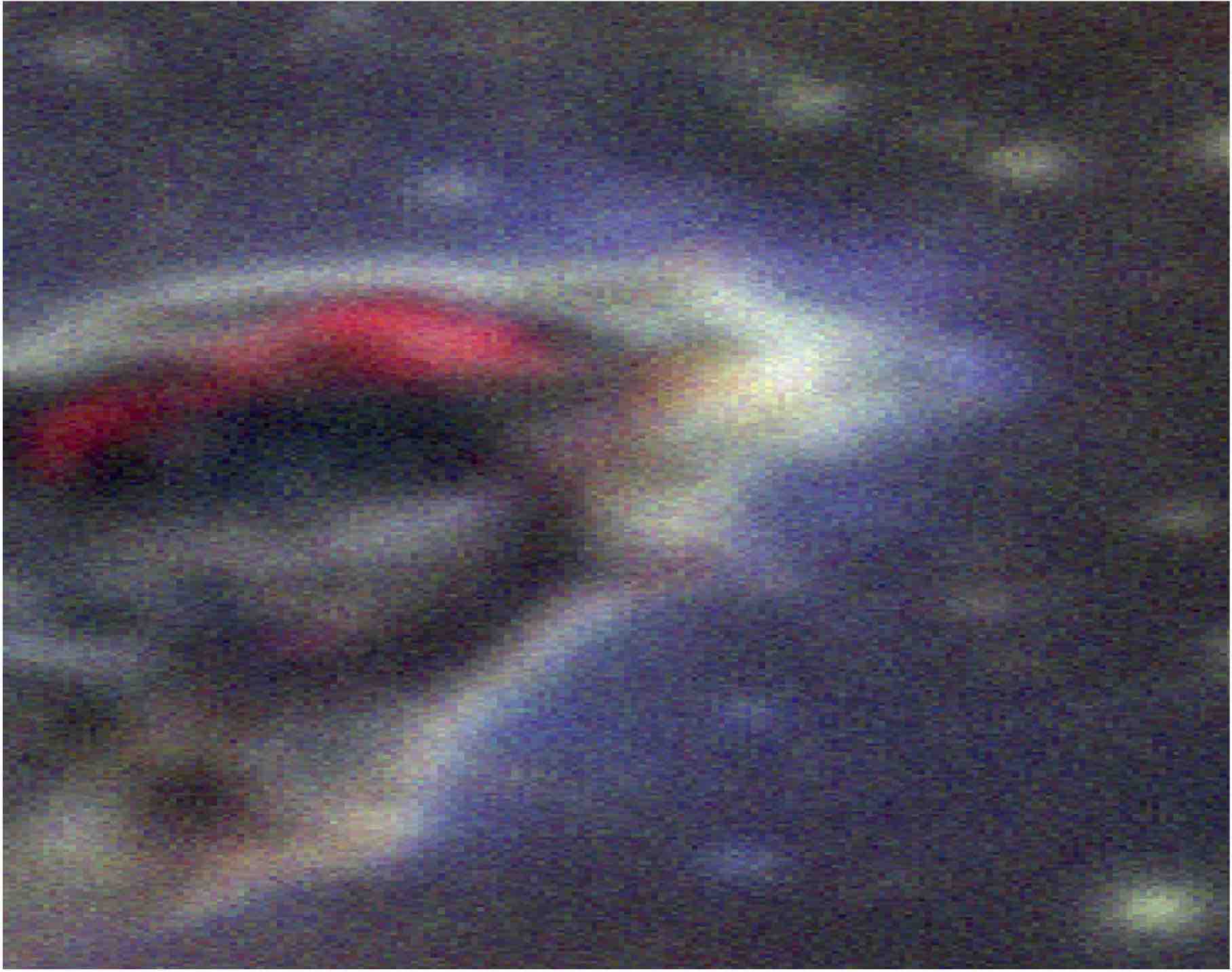} &  \includegraphics[trim={0 0 0 0},clip,width=0.15\textwidth]{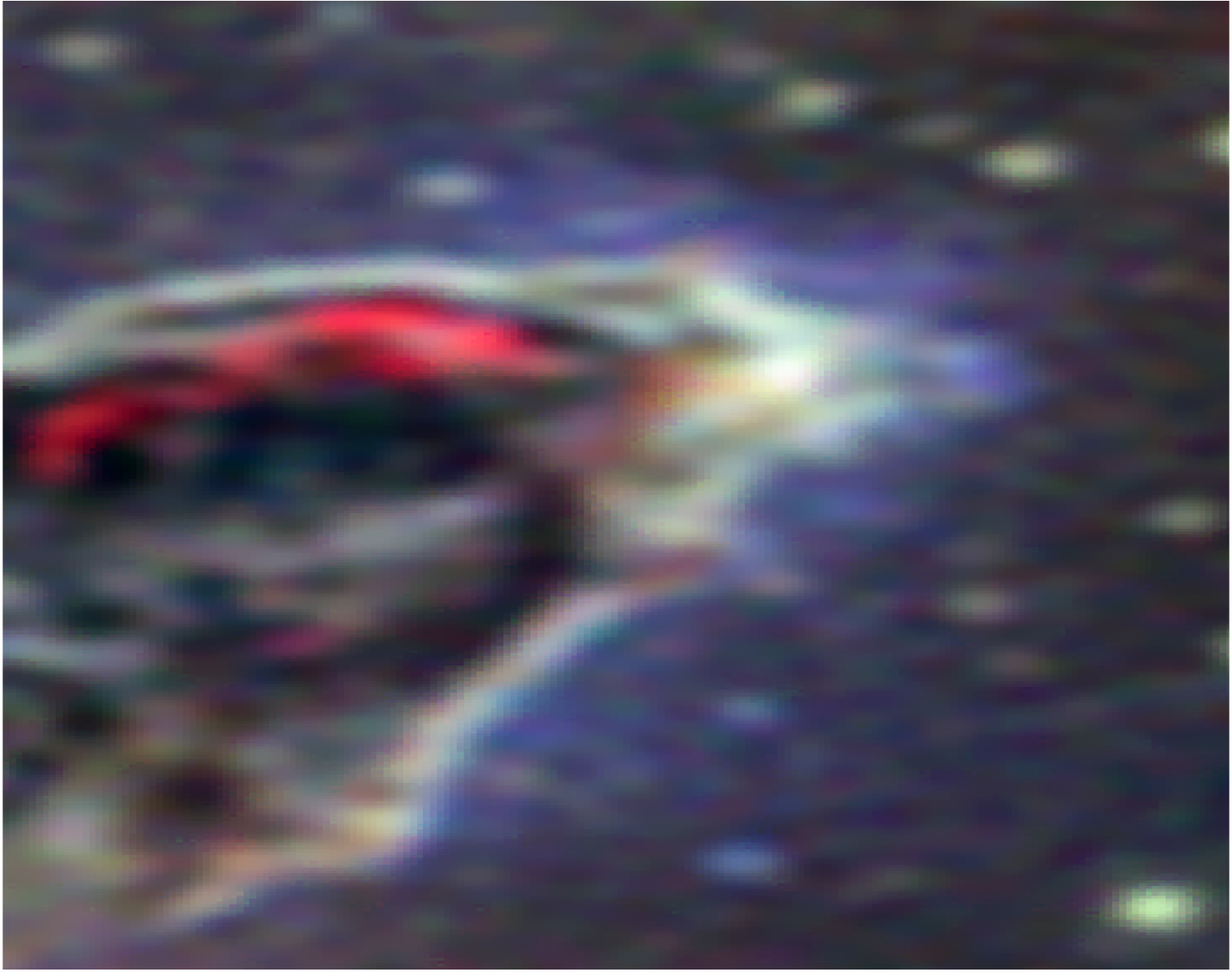} & \includegraphics[trim={0 0 0 0},clip,width=0.15\textwidth]{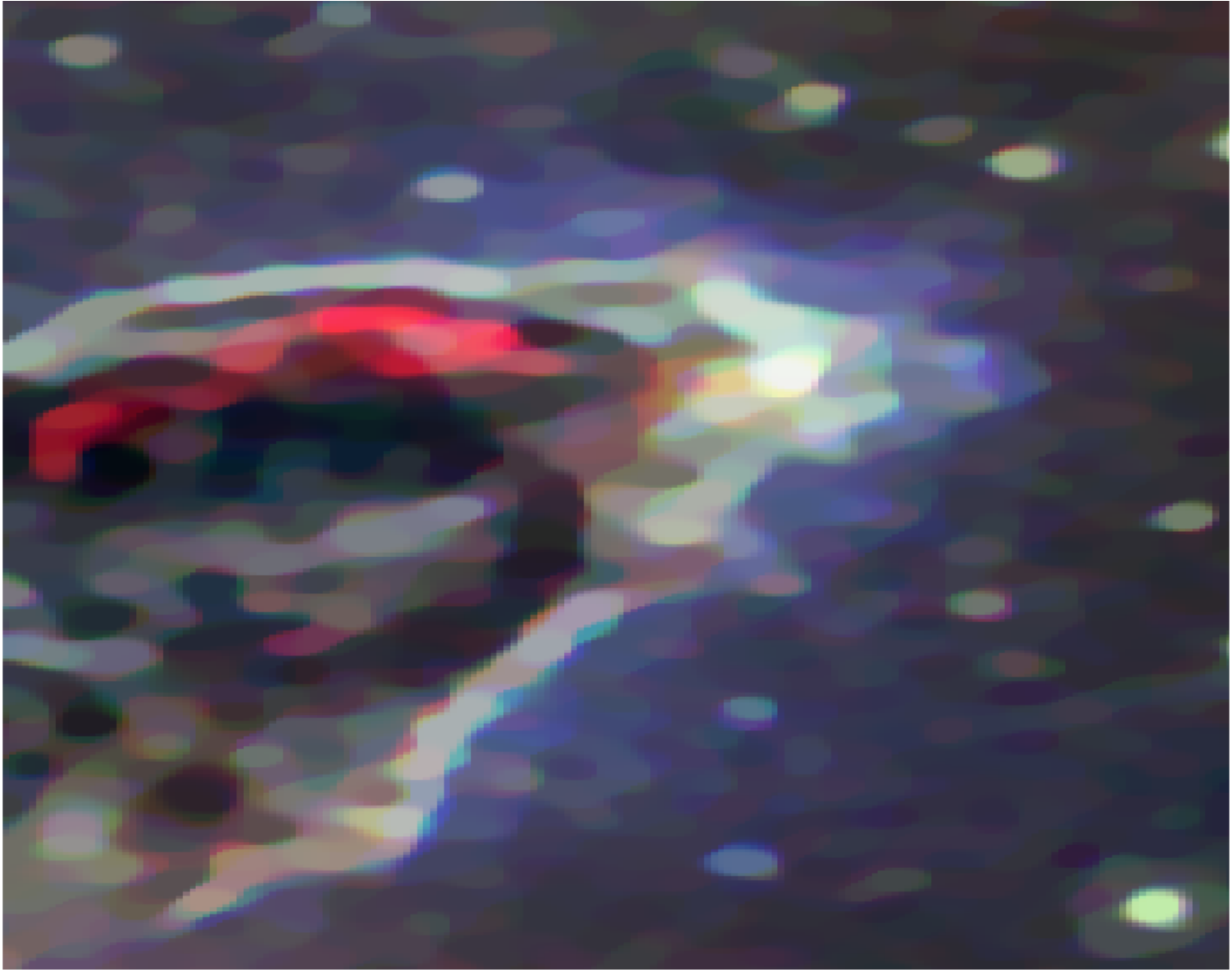}  \\
         \ftn$x$ & \ftn$x_{h,2}^{\text{FISTA}}(15.3)$ & \ftn$x_{h,50}^{\text{FISTA}}(15.7)$ & \ftn$x$ & \ftn$x_{h,2}^{\text{FISTA}}(15.2)$ & \ftn $x_{h,50}^{\text{FISTA}}(15.4)$ 
         \\
         \includegraphics[trim={0 0 0 0},clip,width=0.15\textwidth]{X-crop.pdf} & \includegraphics[trim={0 0 0 0},clip,width=0.15\textwidth]{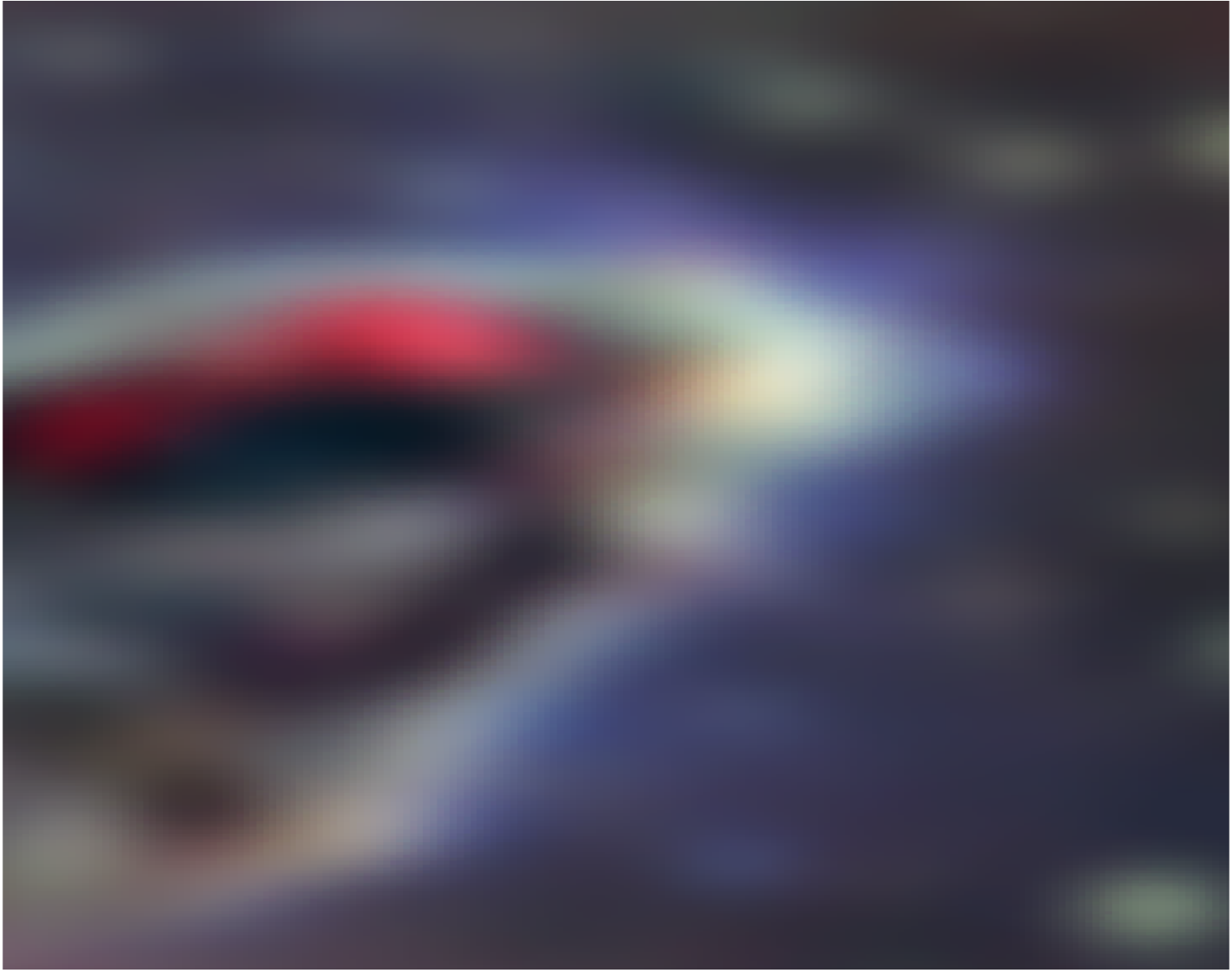} & \includegraphics[trim={0 0 0 0},clip,width=0.15\textwidth]{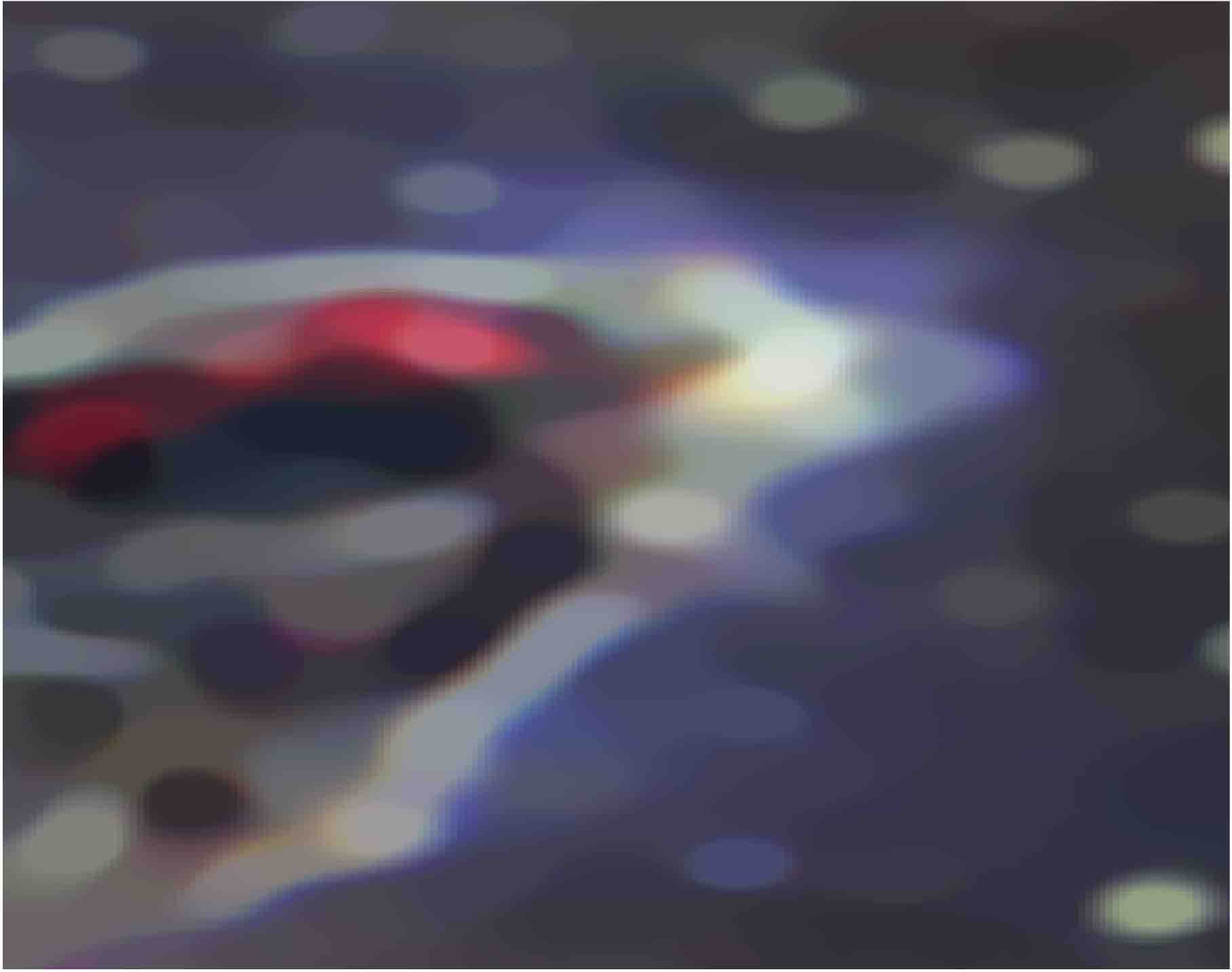} & \includegraphics[trim={0 0 0 0},clip,width=0.15\textwidth]{X-crop.pdf} 
         & \includegraphics[trim={0 0 0 0},clip,width=0.15\textwidth]{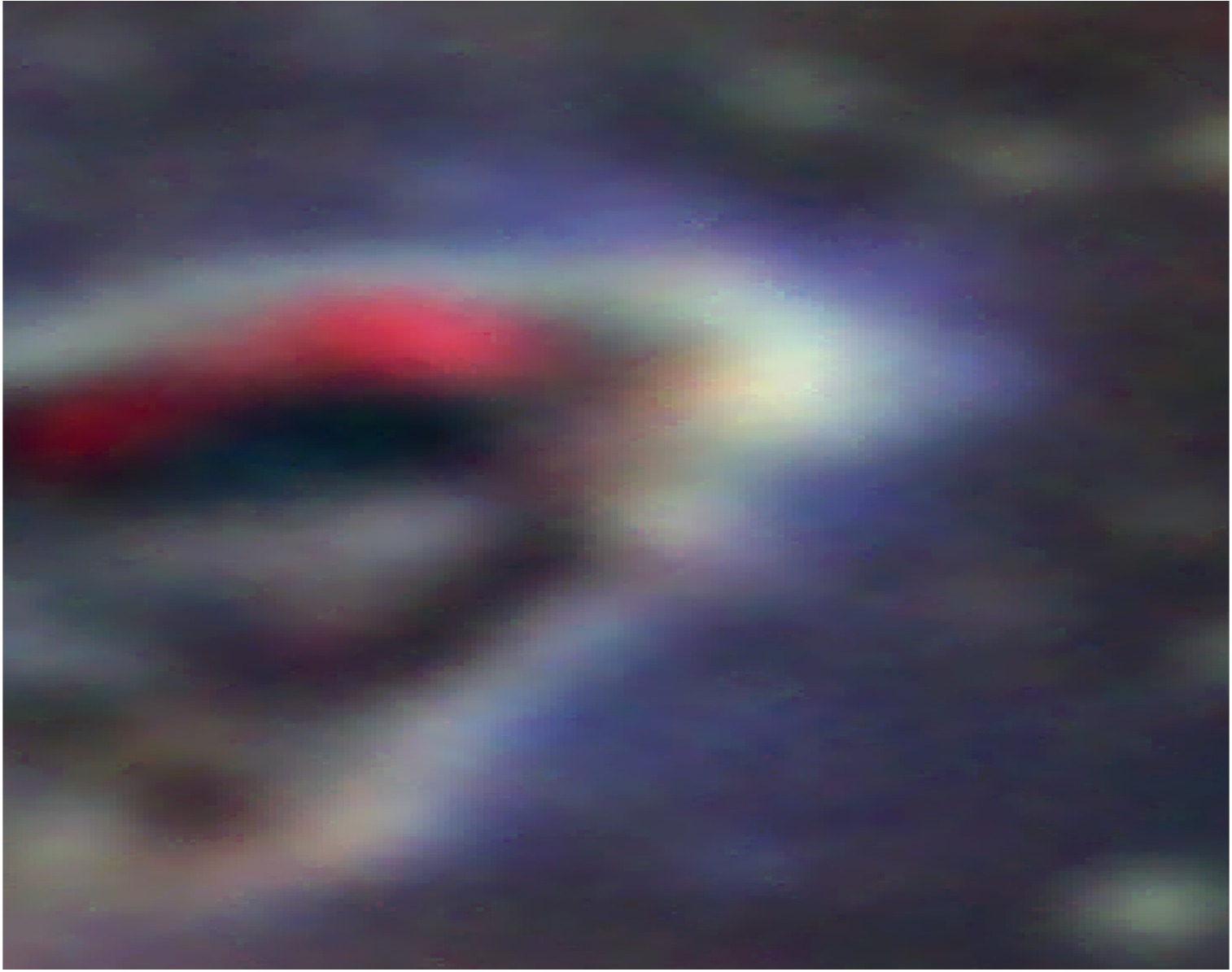} & \includegraphics[trim={0 0 0 0},clip,width=0.15\textwidth]{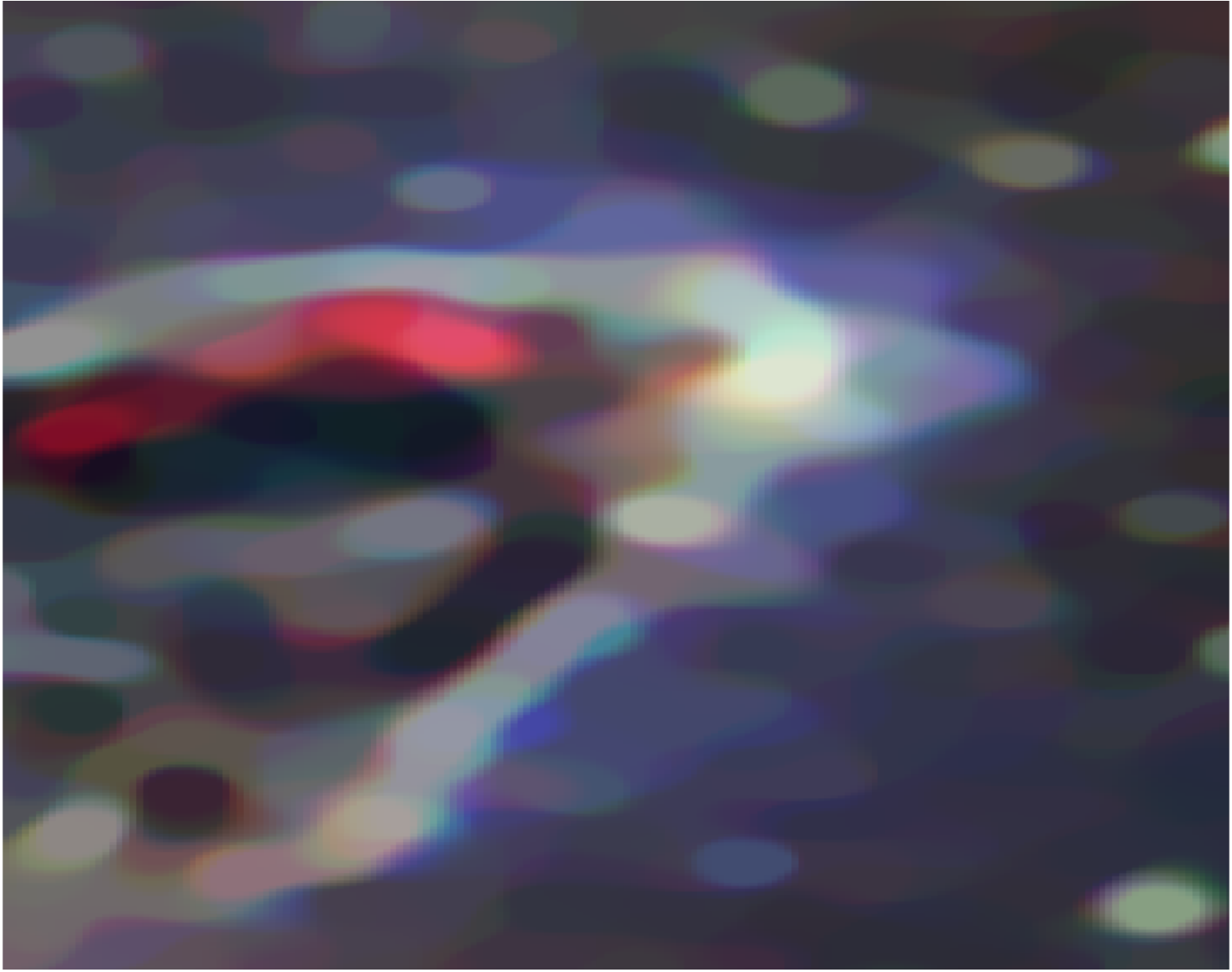}
         \\
         \ftn$z(15)$ & \ftn$x_{h,2}^{\text{IML FISTA}}(15.6)$ & \ftn$x_{h,50}^{\text{IML FISTA}}(15.7)$ & \ftn$z(14.8)$ & \ftn$x_{h,2}^{\text{IML FISTA}}(15.5)$ & \ftn$x_{h,50}^{\text{IML FISTA}}(15.4)$ 
         \\
         \includegraphics[trim={0 0 0 0},clip,width=0.15\textwidth]{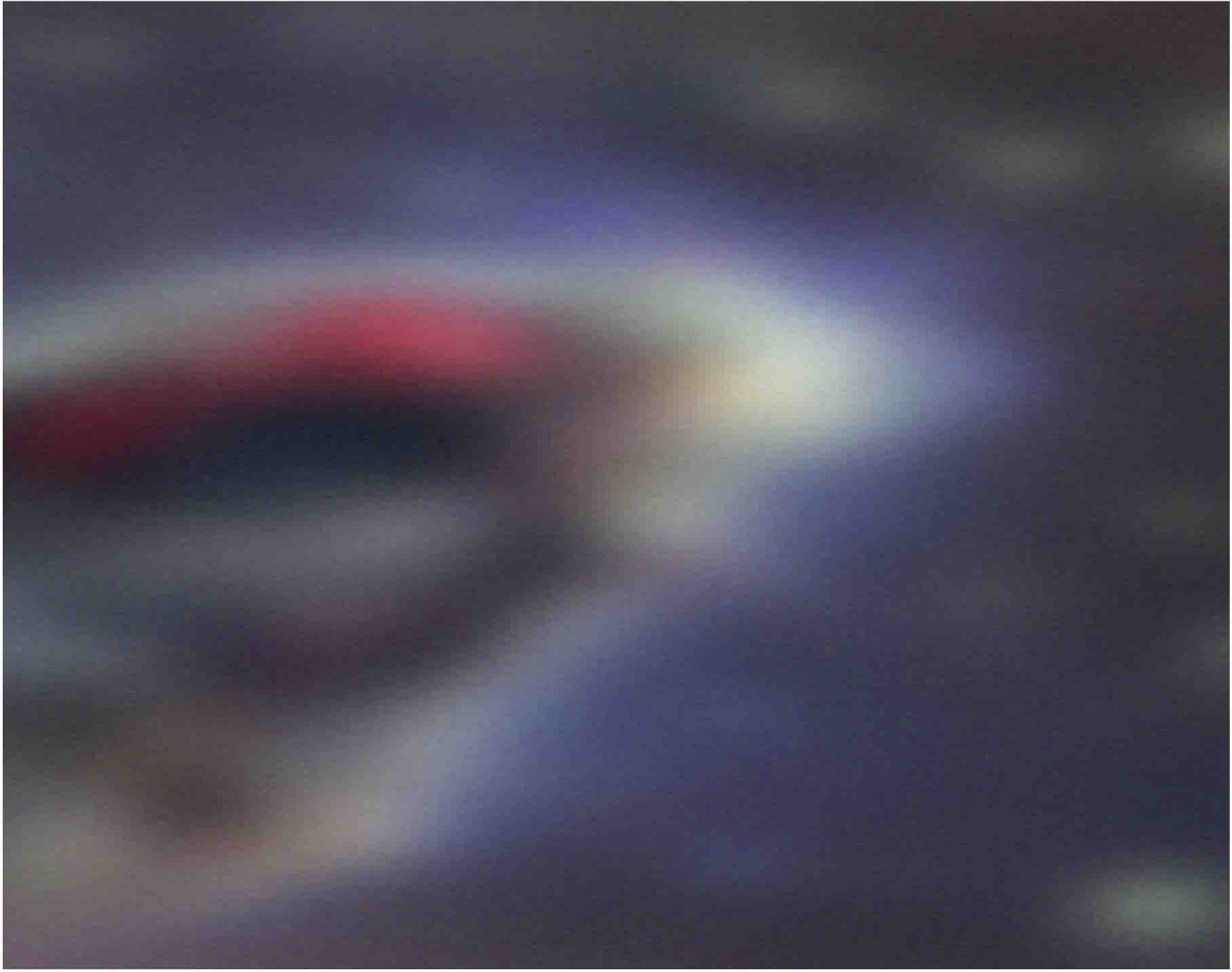} & \includegraphics[trim={0 0 0 0},clip,width=0.15\textwidth]{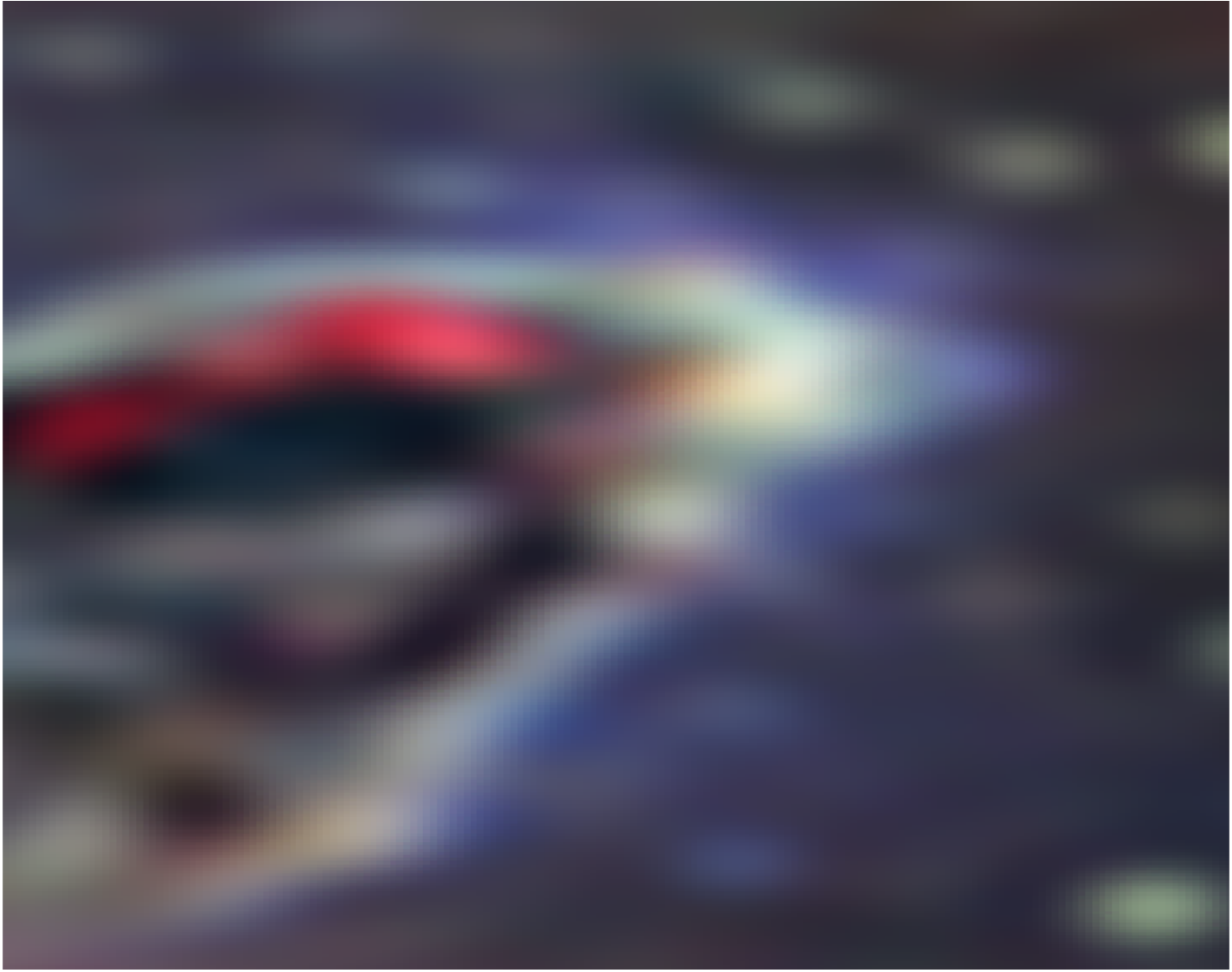} & \includegraphics[trim={0 0 0 0},clip,width=0.15\textwidth]{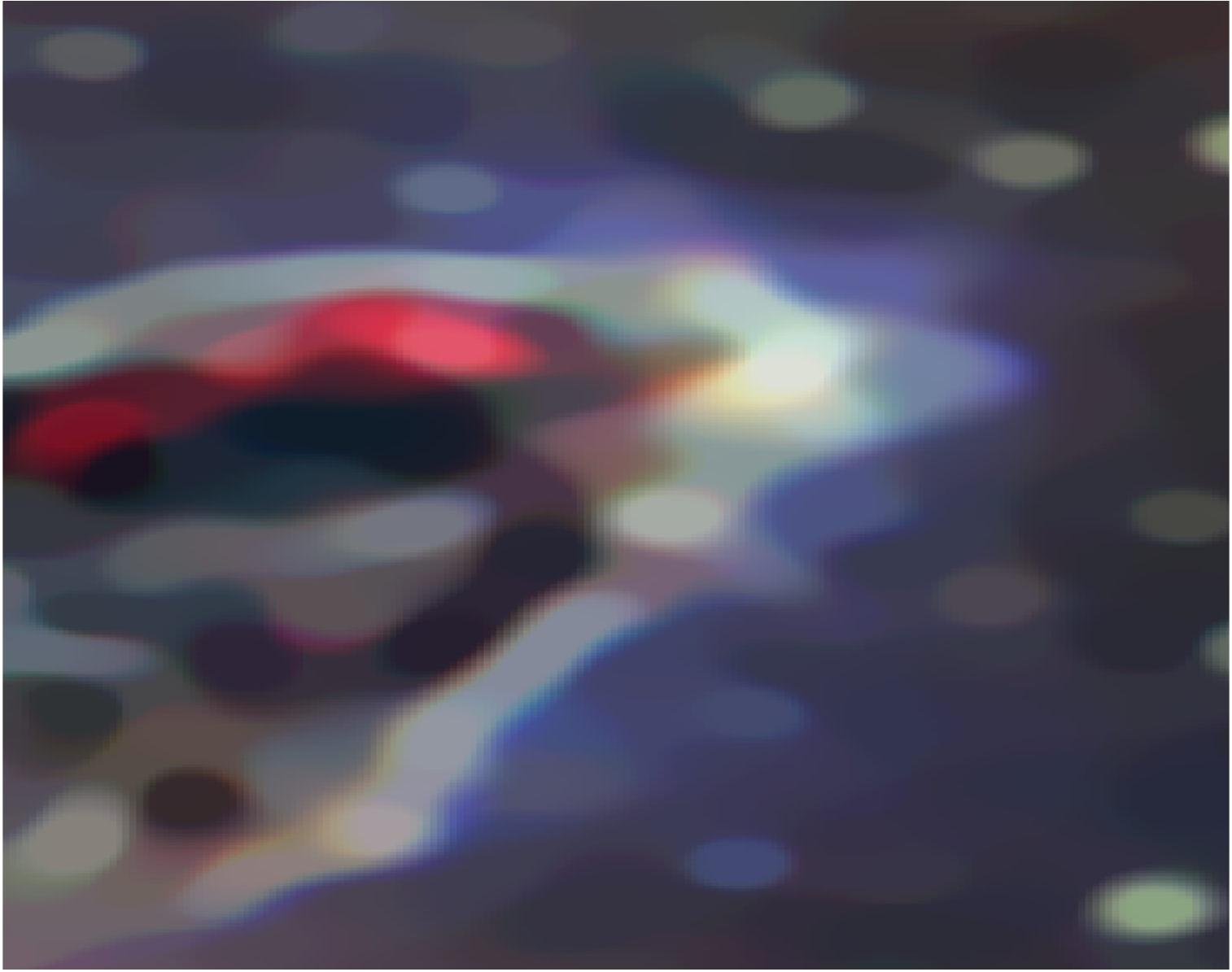} & \includegraphics[trim={0 0 0 0},clip,width=0.15\textwidth]{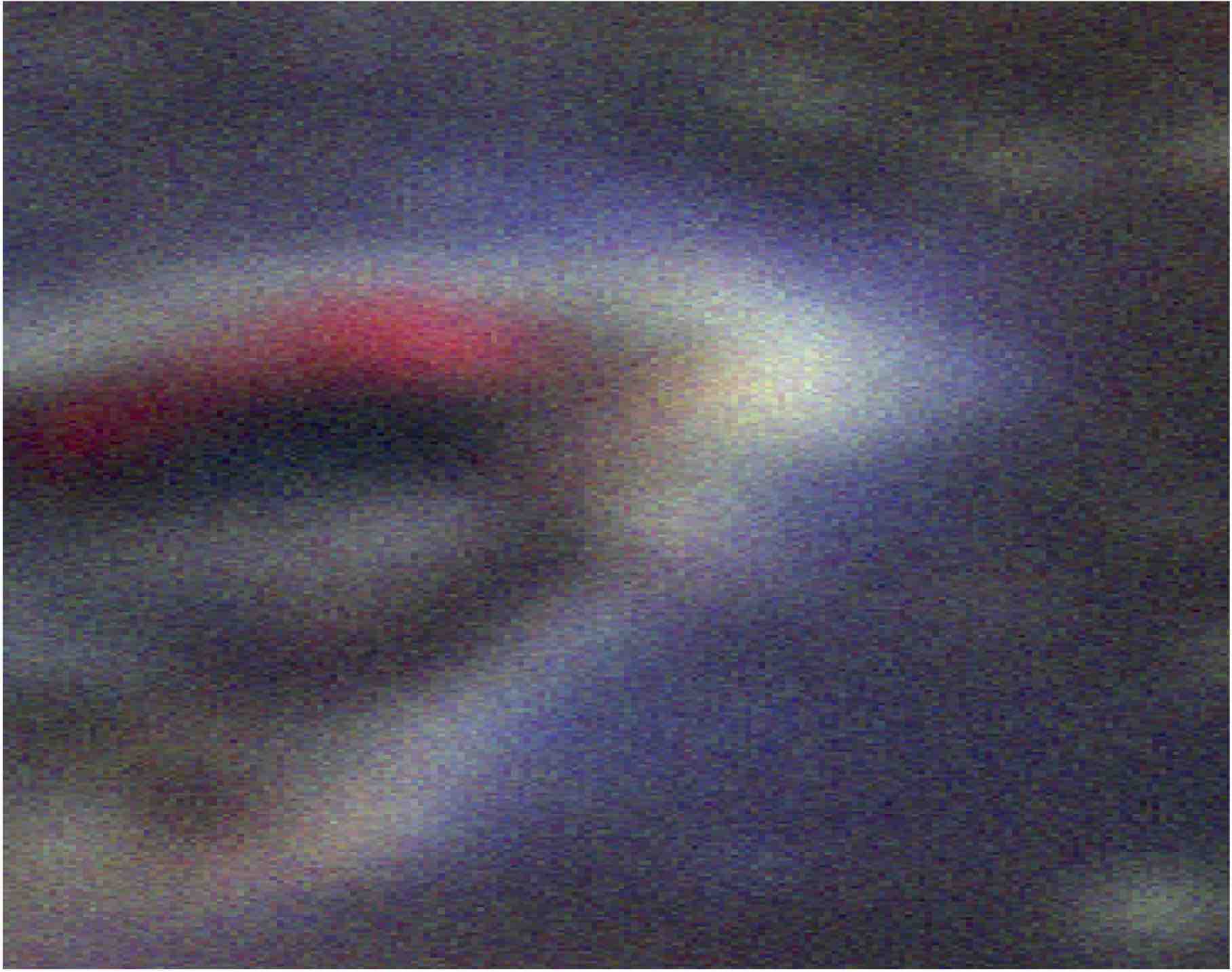} &  \includegraphics[trim={0 0 0 0},clip,width=0.15\textwidth]{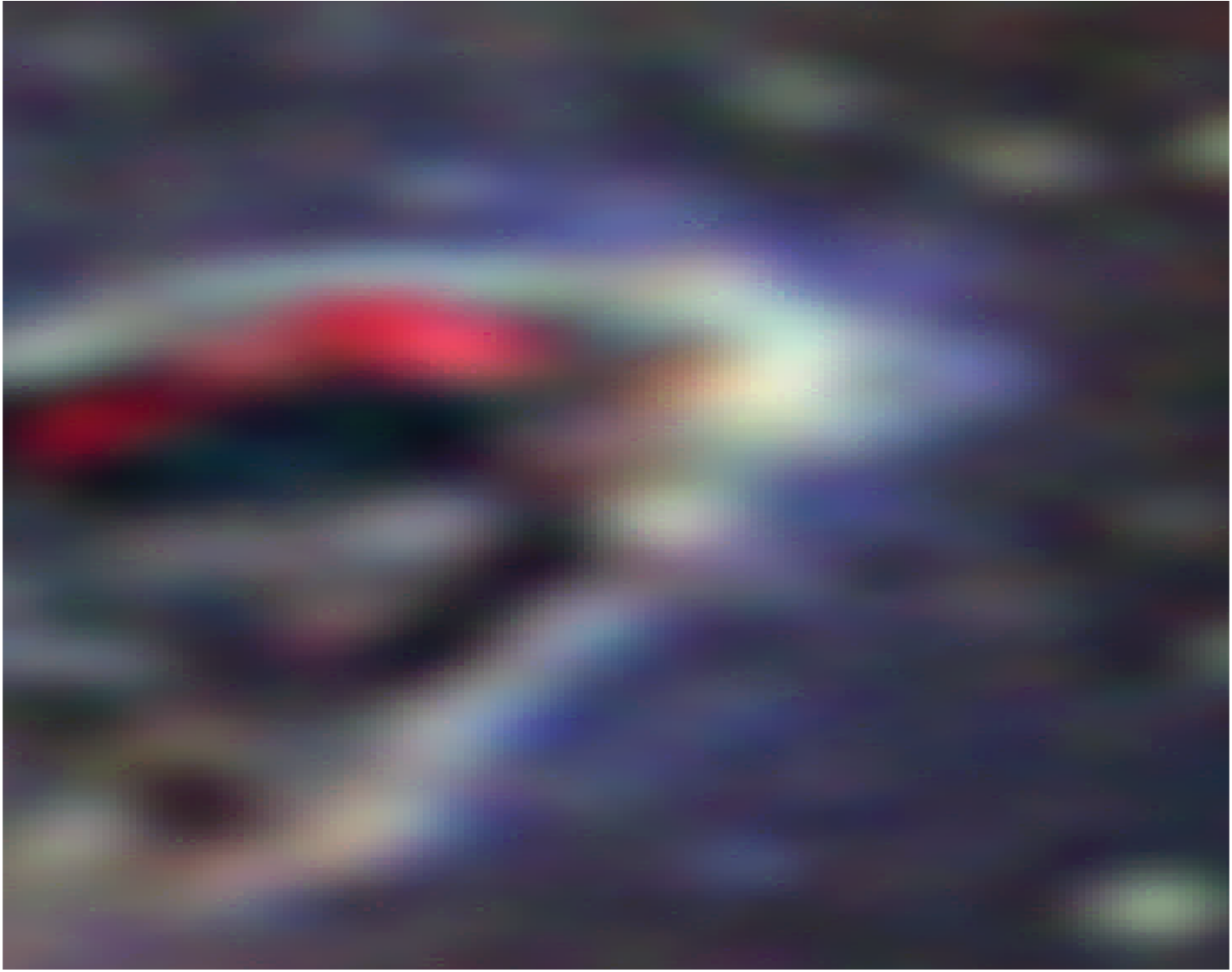} & \includegraphics[trim={0 0 0 0},clip,width=0.15\textwidth]{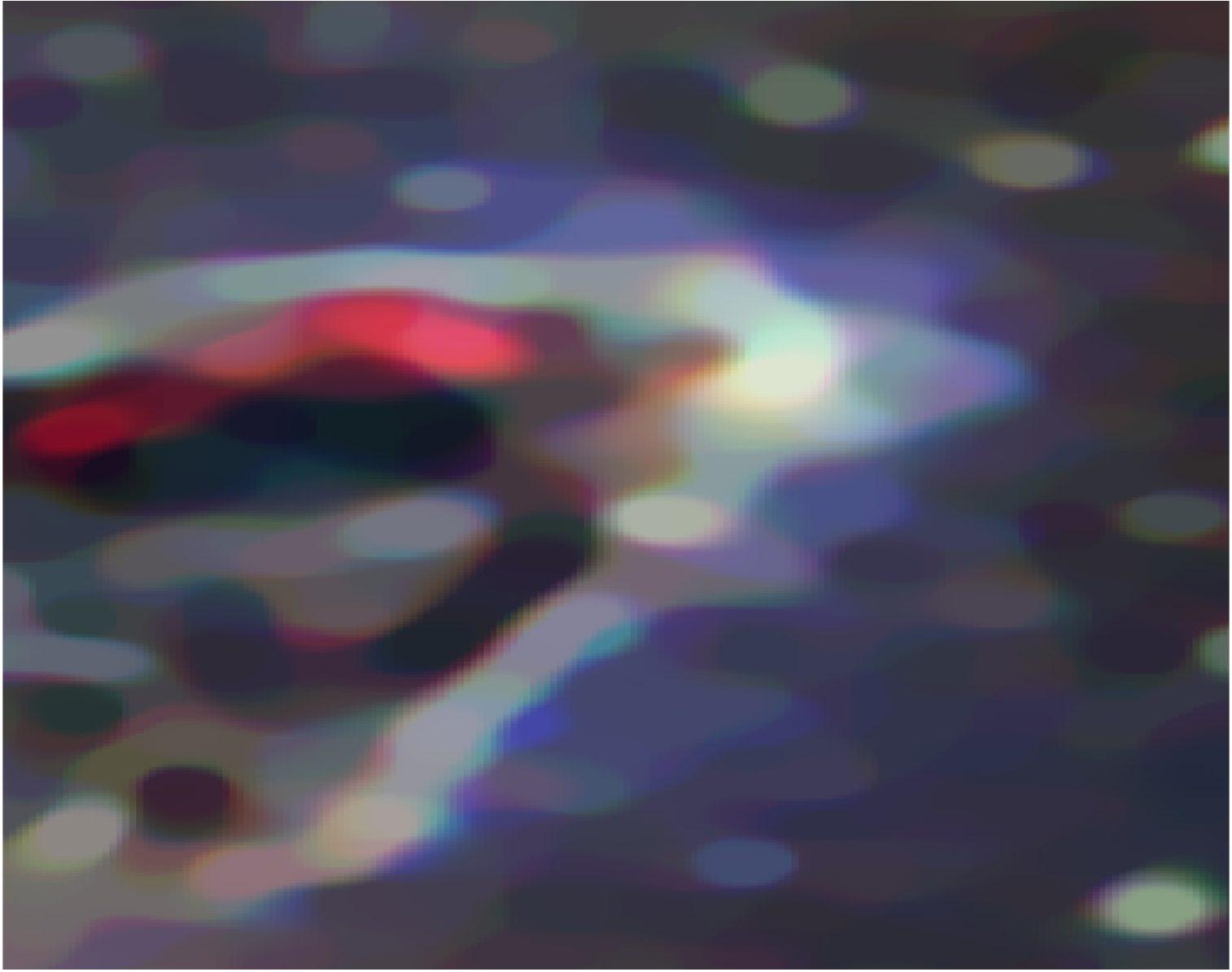}  \\
        \end{tabular}
    \end{center}
    \caption{Deblurring $\ell_{1,2}$-TV for the Pillars of Creation image. Small crop of the image after 2 iterations and after 50 iterations for FISTA (top row) and IML FISTA (bottom row) compared to the original ($x$) and degraded ($z$) images. For each image we report the SNR in dB. \\
    First column: $\sigma$(noise) $= 0.01$; second column: $\sigma$(noise) $= 0.05$. First row: dim(PSF) $=20$, $\sigma$(PSF) $= 3.6$; second row: dim(PSF) $=40$, $\sigma$(PSF) $= 7.3$.} 
    \label{fig:TV_deblurring_imagesIR}
\end{figure}

\paragraph{Experimental performance for different degradation levels} In each of the following figures, the organization of the four plots coincides with the configurations in table \ref{tab:scenarios_blur}.  
For each of them, we tested two regularizations: $\ell_{1,2}$-TV and $\ell_{1,2}$-NLTV. Because the relative behaviour of IML FISTA with respect to FISTA is similar for the two regularizations, for the sake of conciseness, we only report here the results obtained with the $\ell_{1,2}$-TV prior. 
Figure \ref{fig:TV_deblurring_FHIR} and Figure \ref{fig:TV_deblurring_imagesIR} provide a first set of results for the $2048\times 2048$ Pillars of Creation image. We focus in the following on the 25 first iterations as the main gain provided by the proposed method is obtained at the start of the optimization.
We can see that in all cases, the decreasing of the objective function of IML FISTA is faster than that of FISTA.  %

Given the cost of estimating proximity operators for TV and NLTV based regularizations, the computational overhead of a multilevel step is almost negligible, as we expected (cf. Figure \ref{fig:TV_deblurring_FHIR}). Thus, the two low cost coarse corrections are sufficient for our algorithm to gain an advantage that FISTA cannot recover without decreasing the tolerance on the approximation of the proximity operator. %
As a result, this would entail higher computation time at each iteration as the error must decrease with the number of iterations to converge.
Most interestingly, if we compare the methods at the very early stages of the optimization process, after the same number of iterations, IML FISTA reaches a much lower value for the objective function, leading to a much better reconstruction. The difference is particularly striking after 2 iterations (Figure \ref{fig:TV_deblurring_imagesIR}). 

One can also notice that increasing the noise (and thus increasing the value of regularization term $\lambda$) degrades the relative performance of our algorithm compared to FISTA. This behaviour was observed in the exact proximal case (with wavelet based regularization \cite{lauga2022_1}) albeit it is far less pronounced here. Similarly, increasing the blur size improves the relative performances of IML FISTA, just like in the case of exact expression for the proximity operator.

We stress that the potential of multilevel strategies, especially for high levels of degradations (i.e., blurring and noise), is particularly evident for large scale images: on smaller problems the overhead introduced by the method may overcome the gain obtained in passing to lower resolutions. This is evident when looking
at the results obtained in the same degradation context for the Yellow Car image of size $512 \times 512 \times 3$. %
We reproduce in Figure \ref{fig:TV_deblurring_FHIM} 
the evolution of the objective function when the regularization is the $\ell_{1,2}$-TV norm. With this problem of small dimension, the relative performances of IML FISTA compared to those of FISTA are less impressive than in the case of the Pillars of Creation image.%
\begin{figure}
    \centering
    \includegraphics[trim={2.5em 2em 0 2.7em},clip,width=0.395\textwidth]{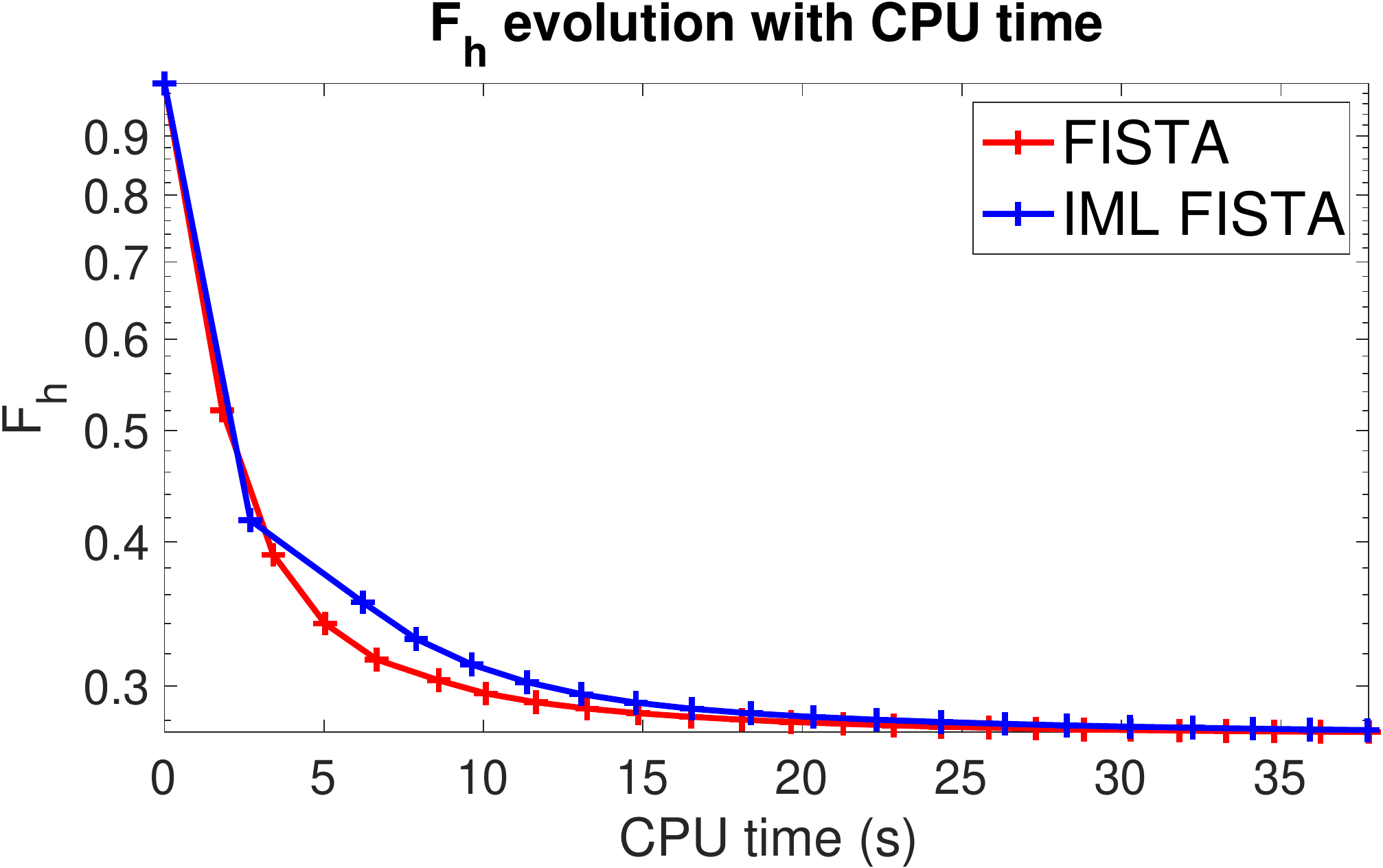}
    \hspace{2em}
    \includegraphics[trim={2.5em 2em 0 2.7em},clip,width=0.4\textwidth]{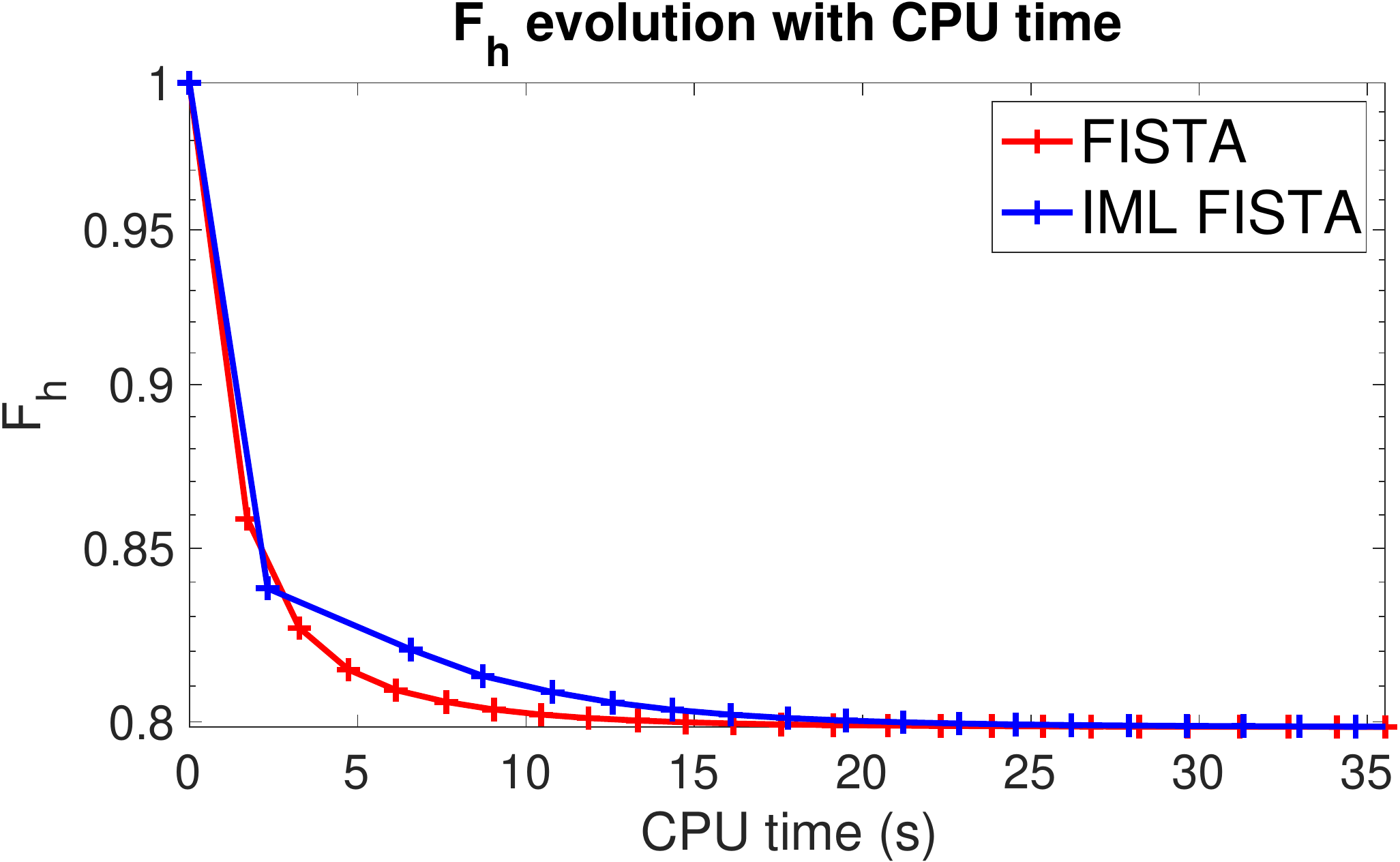} \\ \vspace{0.1em}
    \includegraphics[trim={2.5em 2em 0 2.7em},clip,width=0.395\textwidth]{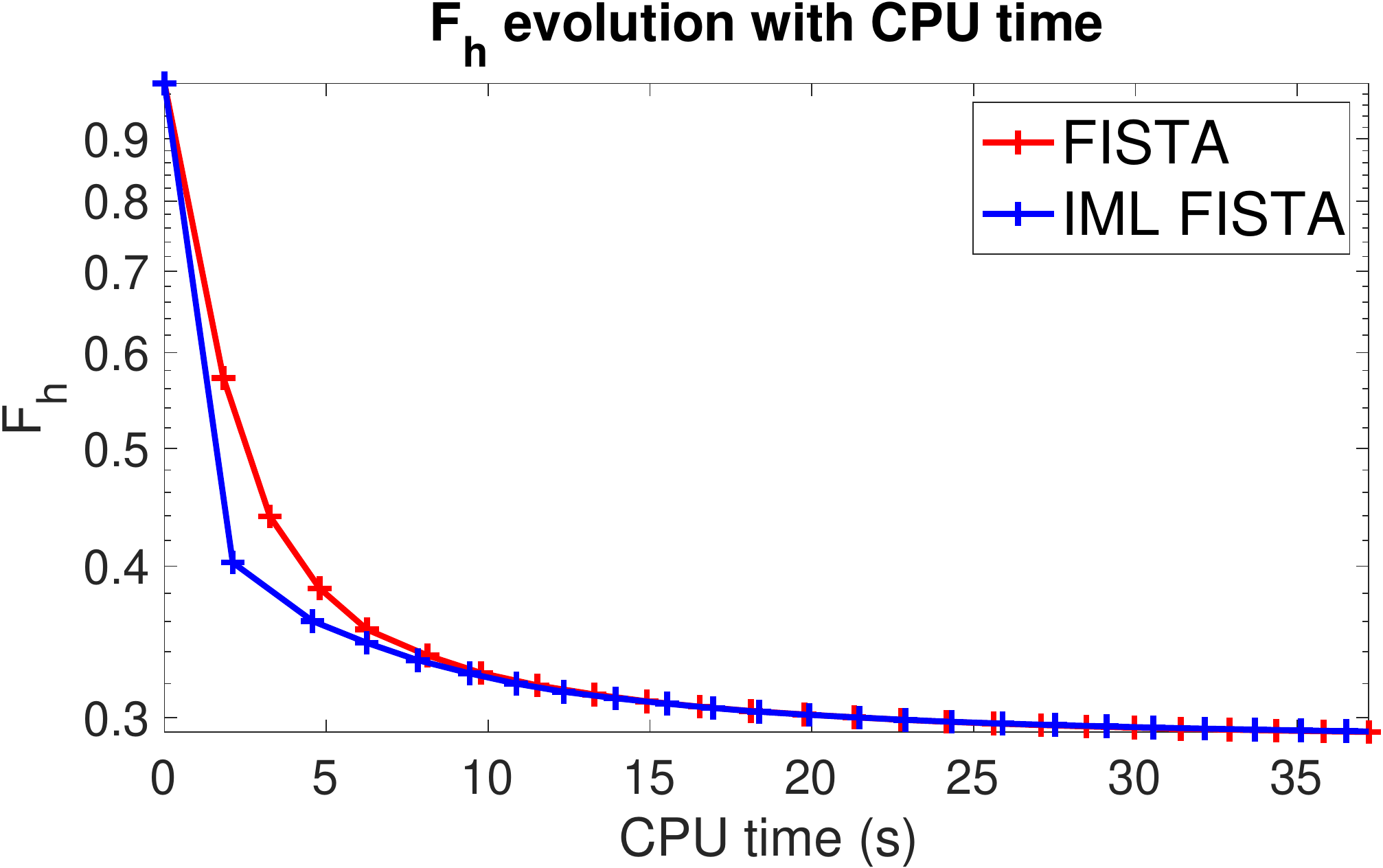}
    \hspace{1em}
    \includegraphics[trim={2.5em 2em 0 2.7em},clip,width=0.4\textwidth]{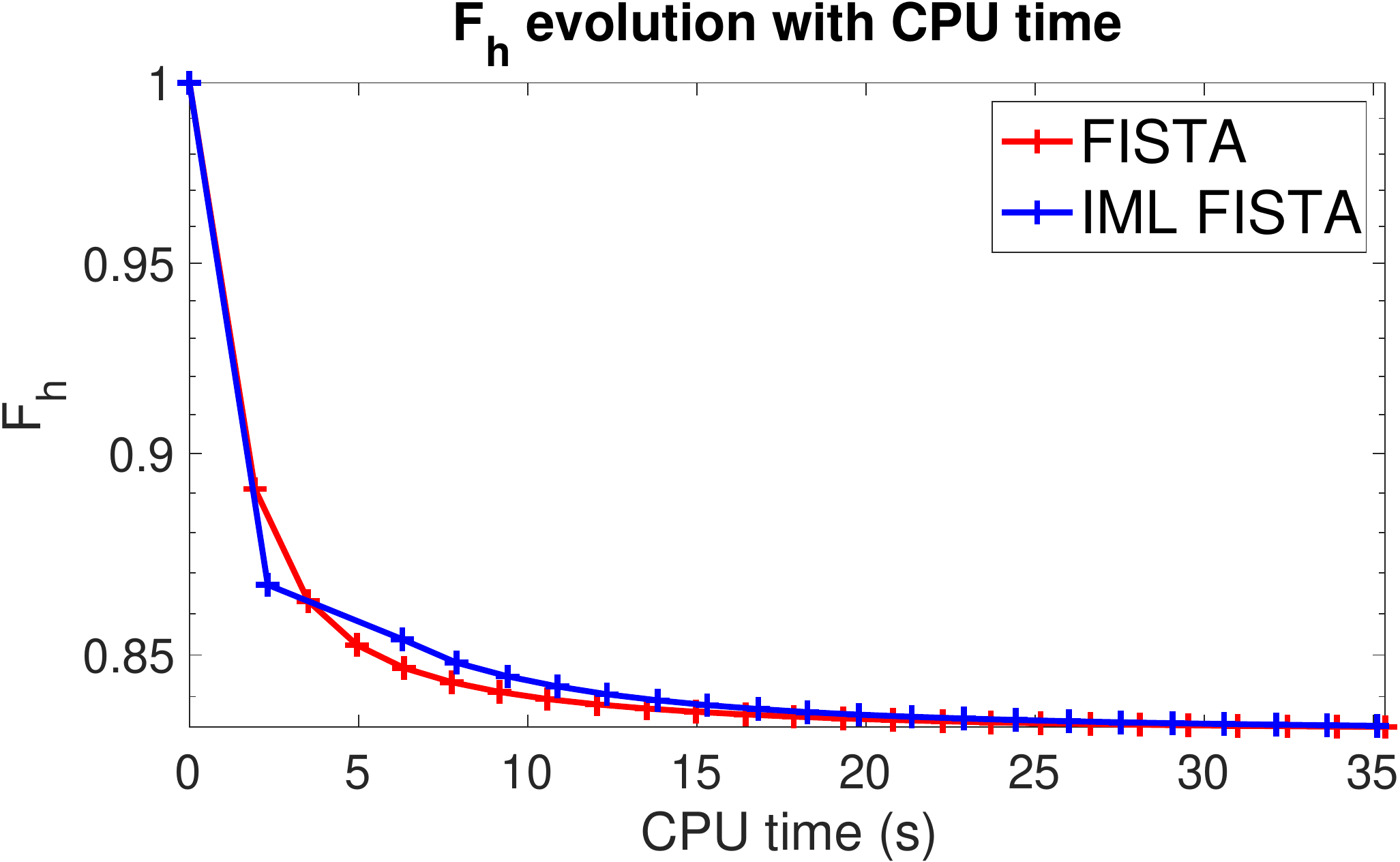}
    \caption{Deblurring $\ell_{1,2}$-TV for the Yellow Car image (small dimensional image). Objective function (normalized with initialization value) vs CPU time (sec). First column: $\sigma$(noise) $= 0.01$; second column: $\sigma$(noise) $= 0.05$. First row: dim(PSF) $=20$, $\sigma$(PSF) $= 3.6$; second row: dim(PSF) $=40$, $\sigma$(PSF) $= 7.3$. For each plot, the crosses represent iterations of the algorithm.\vspace{-1.5em}} 
    \label{fig:TV_deblurring_FHIM}
\end{figure}
\subsection{IML FISTA results on image inpainting}
\begin{figure}
    \centering
    \includegraphics[trim={2.5em 2em 0 2.7em},clip,width=0.4\textwidth]{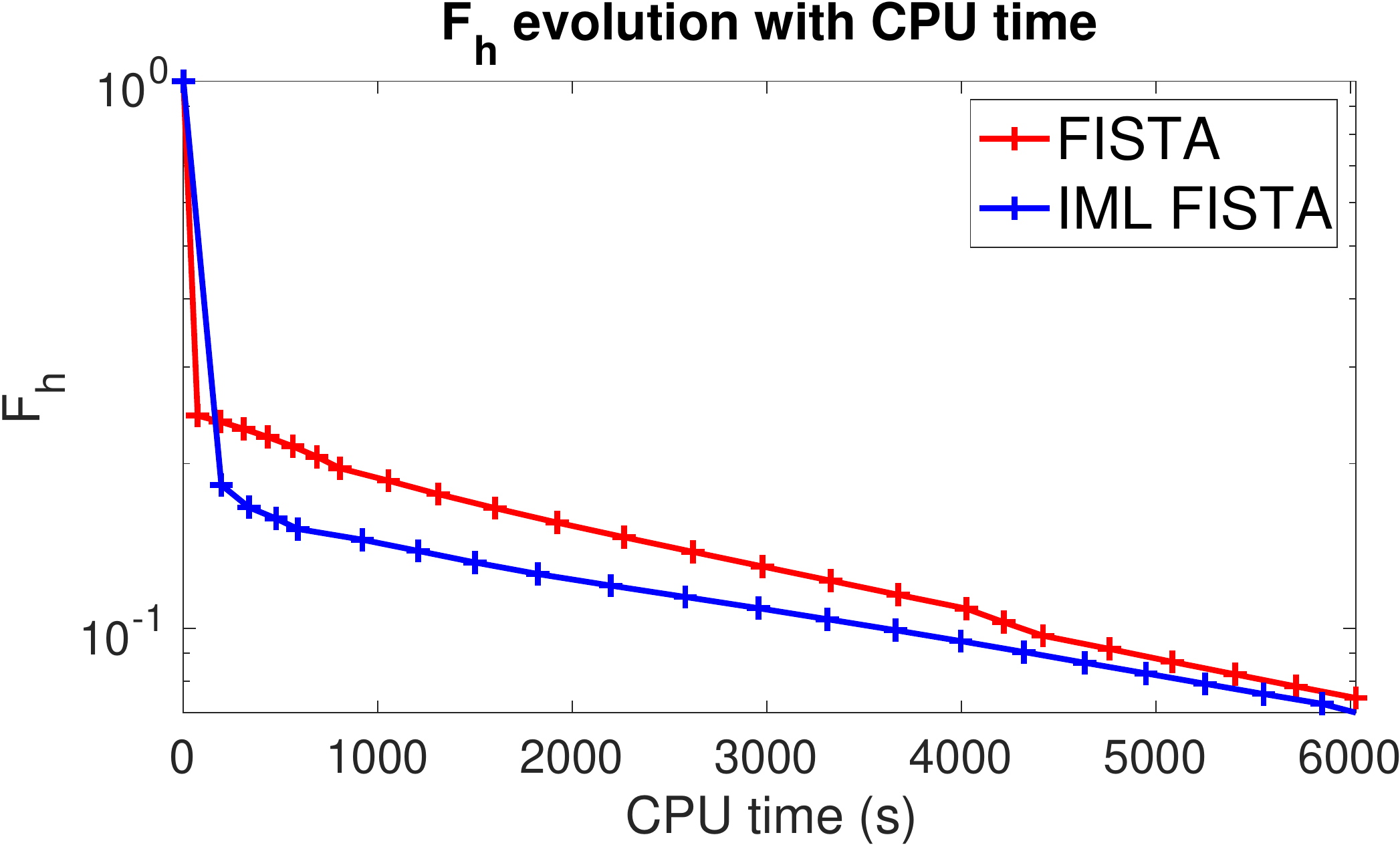}
    \hspace{2em}
    \includegraphics[trim={2.5em 2em 0 2.7em},clip,width=0.395\textwidth]{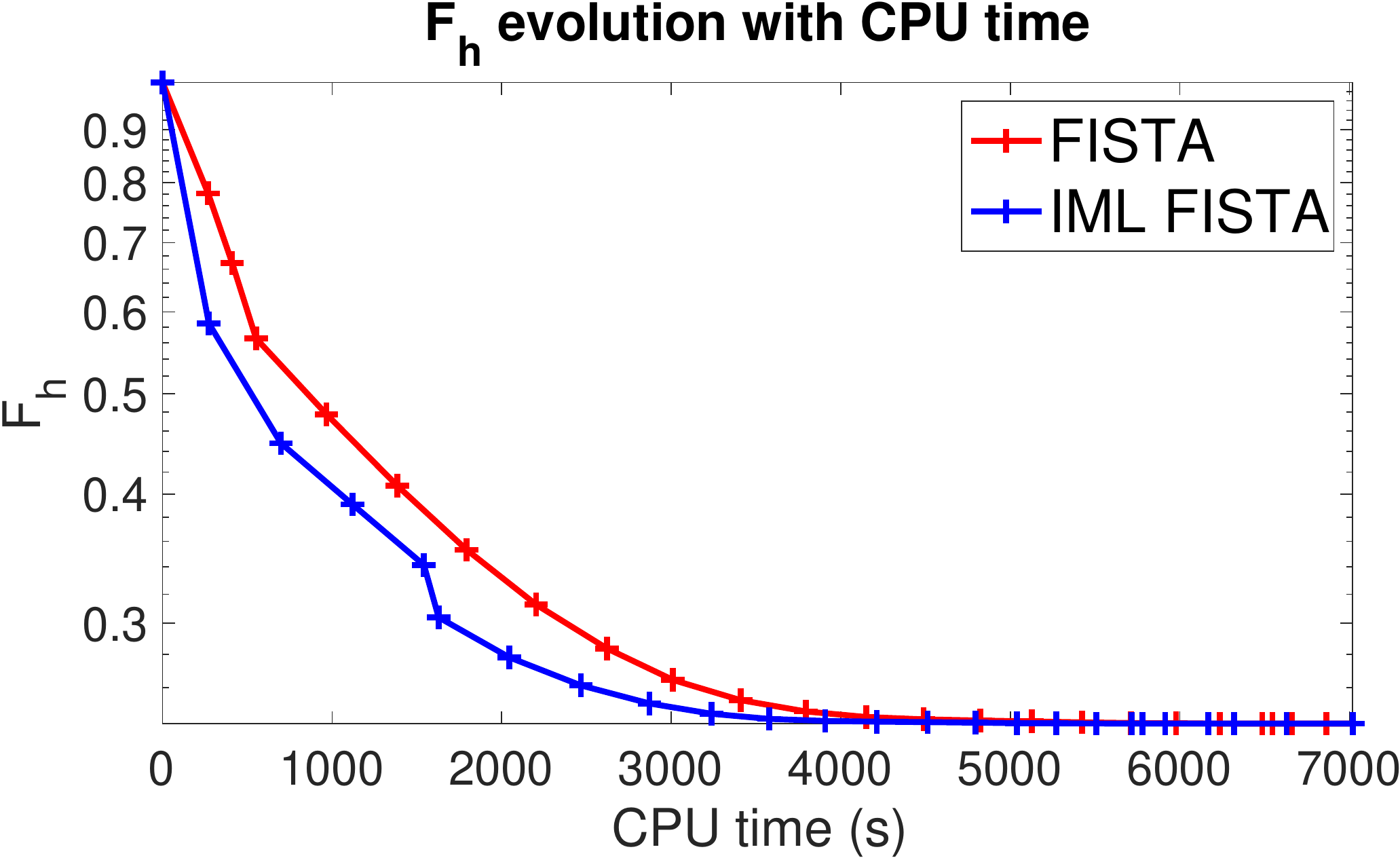} \\ \vspace{0.1em}
    \includegraphics[trim={2.5em 2em 0 2.7em},clip,width=0.4\textwidth]{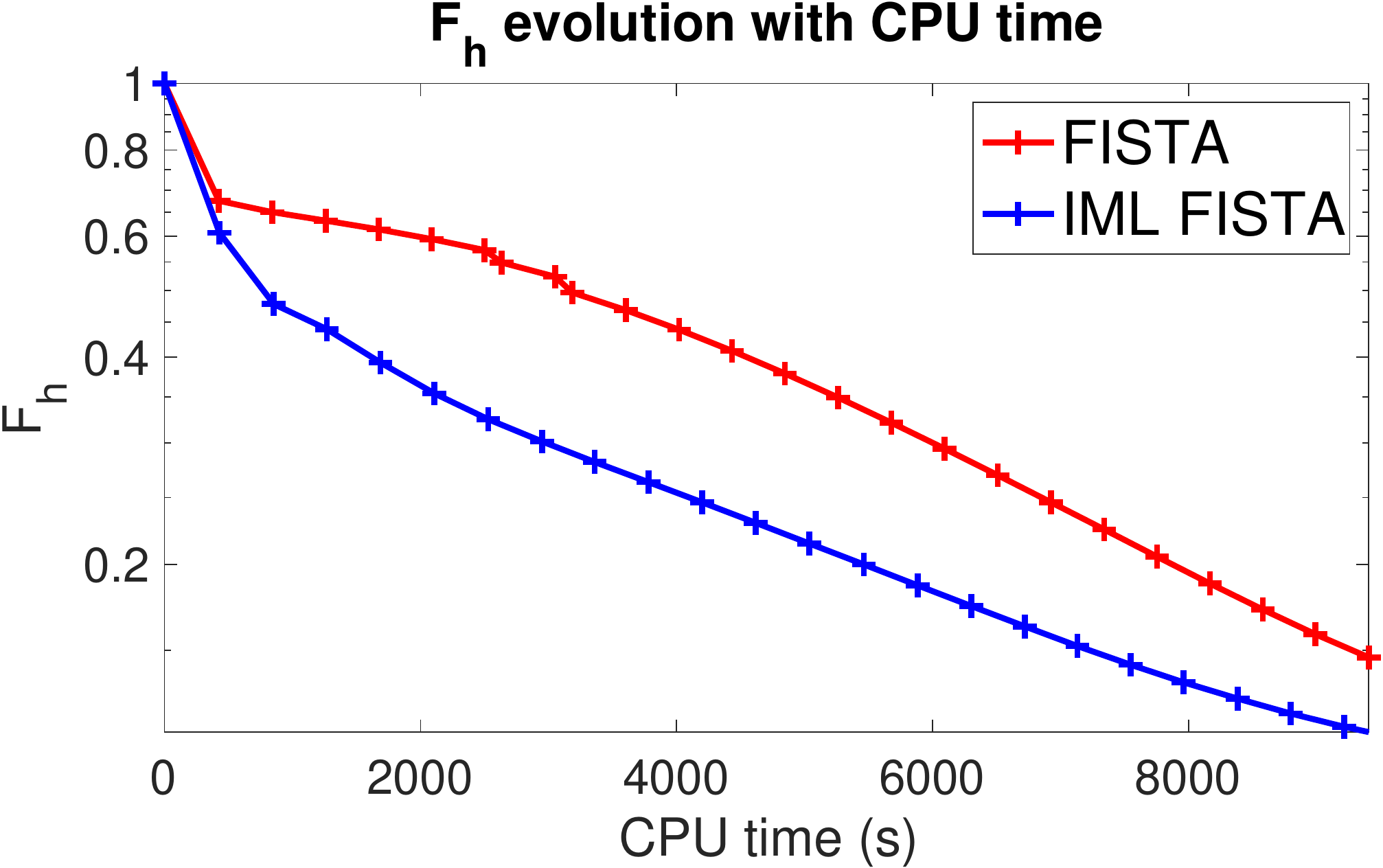}
    \hspace{2em}
    \includegraphics[trim={2.5em 2em 0 2.7em},clip,width=0.395\textwidth]{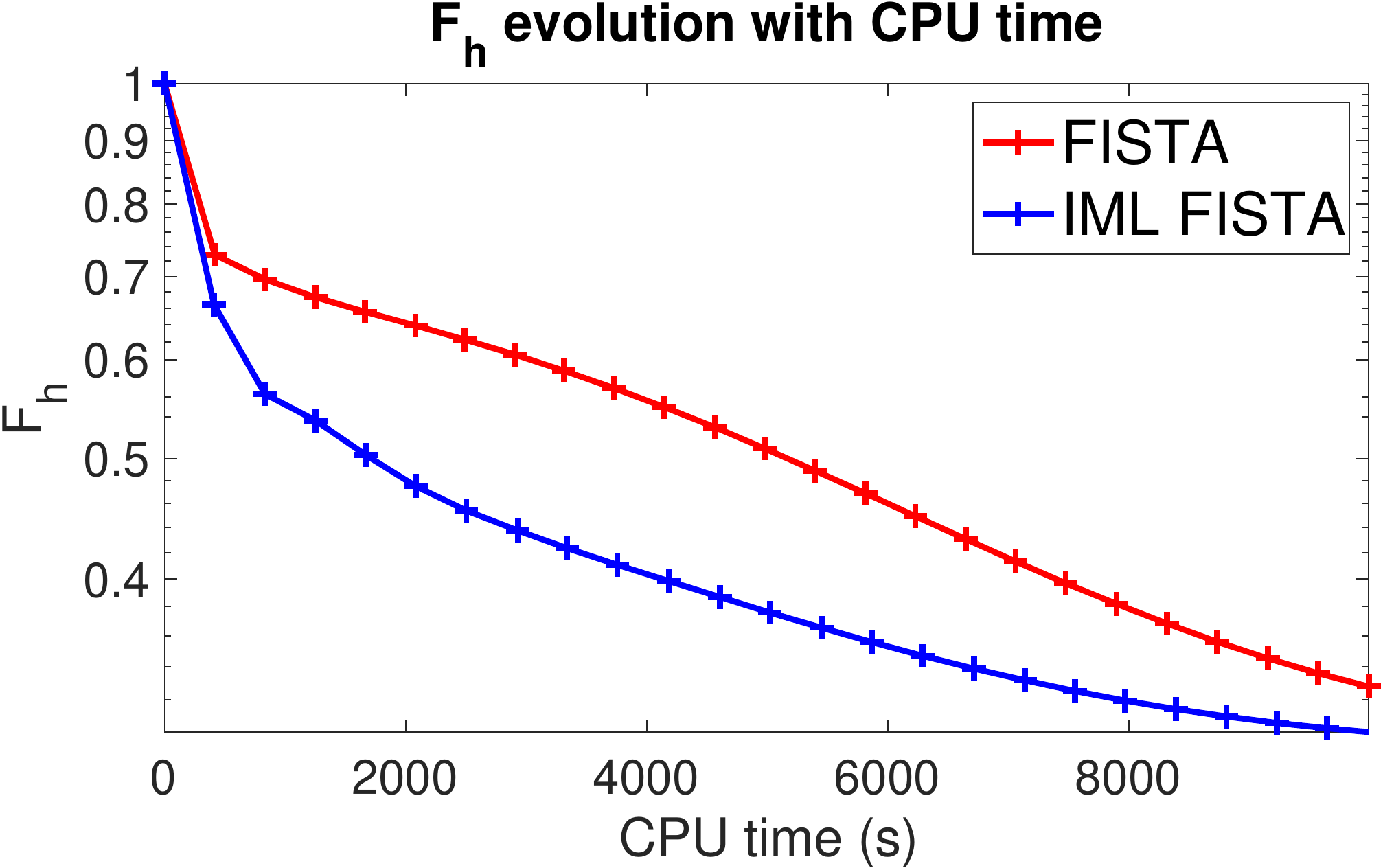}
    \caption{Inpainting $\ell_{1,2}$-NLTV for the Pillars of Creation image. Objective function (normalized with initialization value) vs CPU time (sec). First column: $\sigma$(noise) $= 0.01$; second column: $\sigma$(noise) $= 0.05$. First row: missing pixels $50\%$; second row: missing pixels $90\%$. For each plot, the crosses represent iterations of the algorithm. \vspace{-1em}} 
    \label{fig:NLTV_inpainting_FHIR}
\end{figure}
\begin{figure}
    \begin{center}
        \setlength{\tabcolsep}{3pt}
        \begin{tabular}{cccccc}
        \ftn $x$ & \ftn $x_{h,2}^{\text{FISTA}}(6.1)$ & \ftn $x_{h,50}^{\text{FISTA}}(20.5)$ & \ftn $x$ & \ftn$x_{h,2}^{\text{FISTA}}(7.4)$ & \ftn$x_{h,50}^{\text{FISTA}}(19.2)$ \\
         \includegraphics[trim={0 0 0 0},clip,width=0.15\textwidth]{X-crop.pdf} & \includegraphics[trim={0 0 0 0},clip,width=0.15\textwidth]{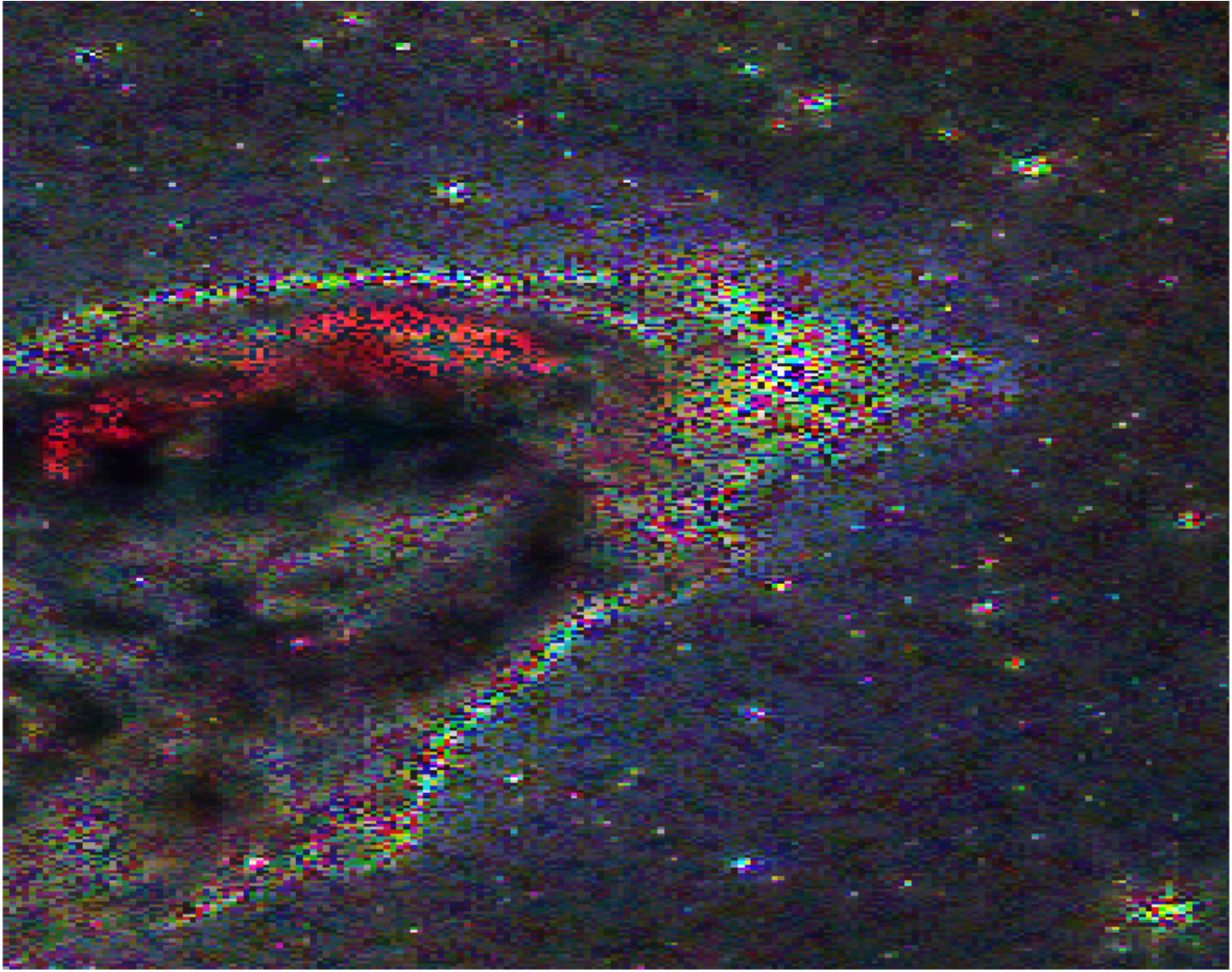} & \includegraphics[trim={0 0 0 0},clip,width=0.15\textwidth]{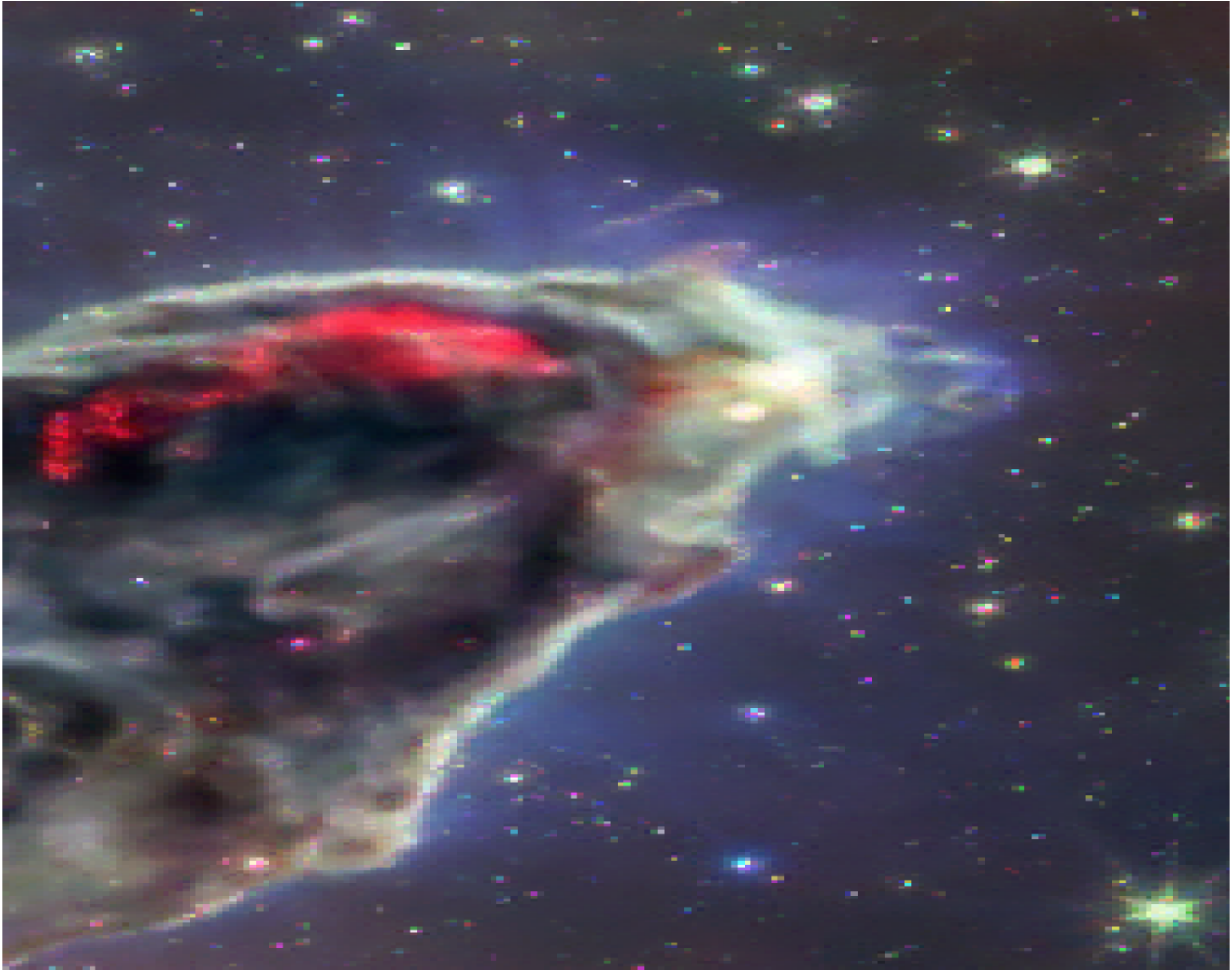} & \includegraphics[trim={0 0 0 0},clip,width=0.15\textwidth]{X-crop.pdf} 
         & \includegraphics[trim={0 0 0 0},clip,width=0.15\textwidth]{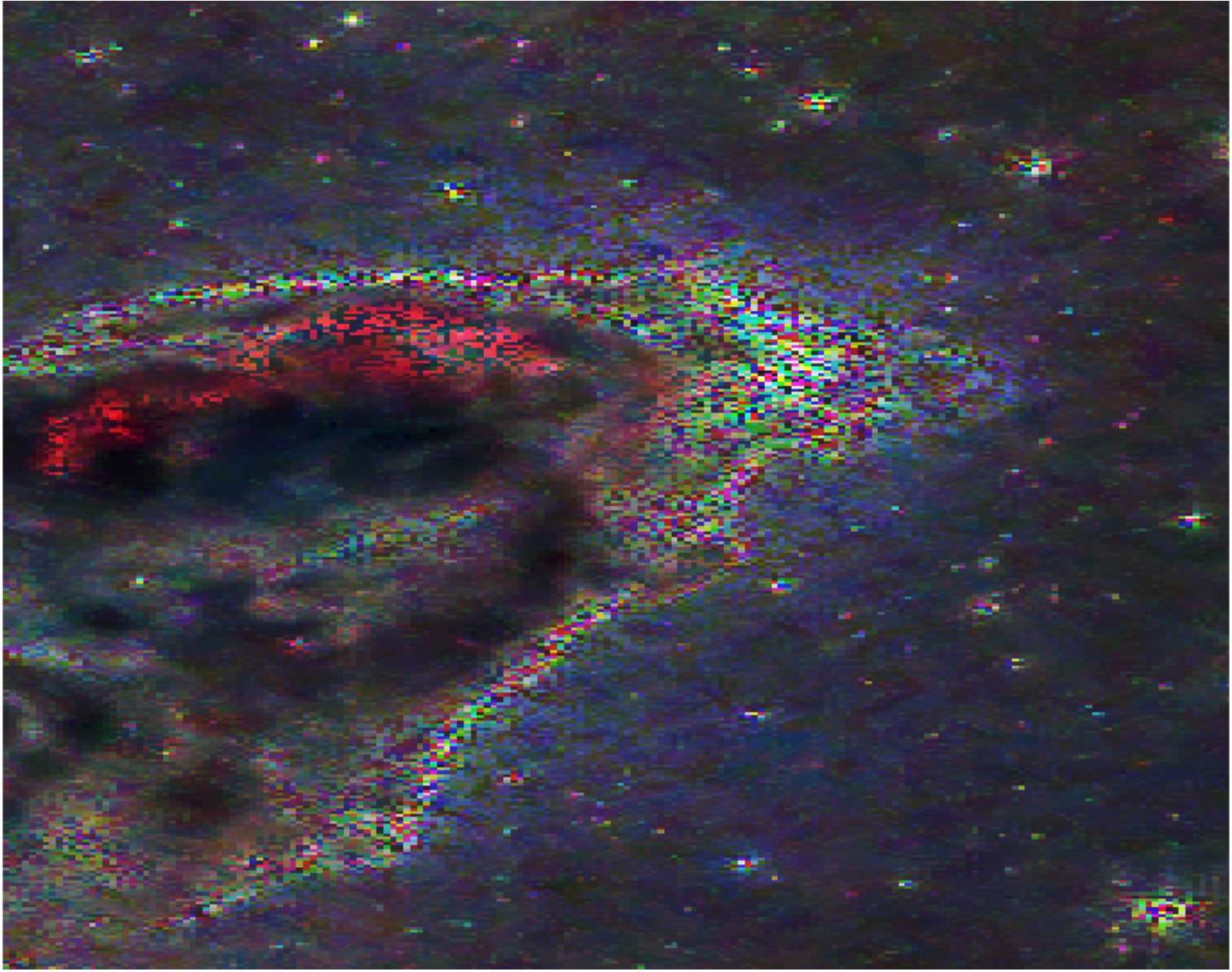} & \includegraphics[trim={0 0 0 0},clip,width=0.15\textwidth]{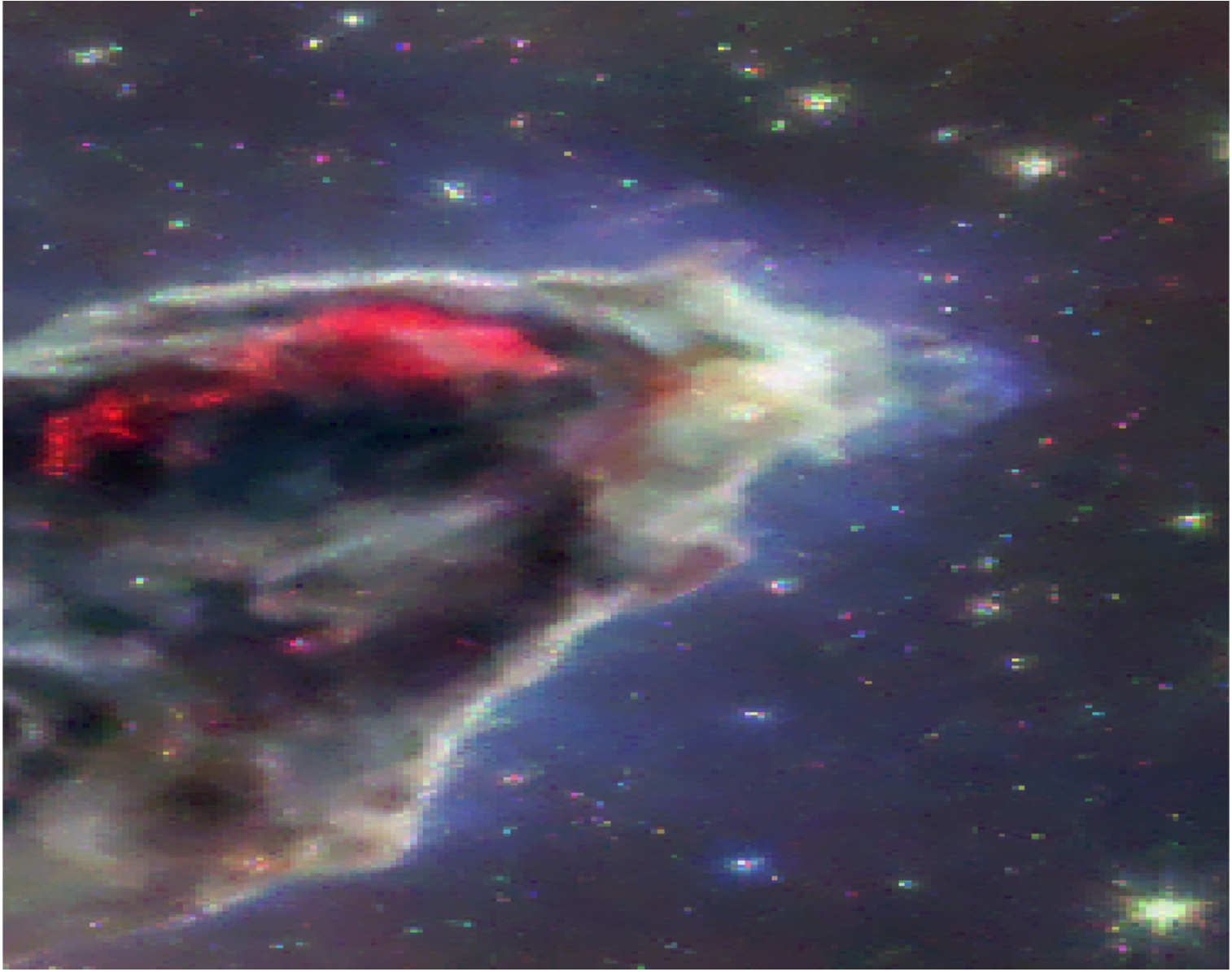}
         \\
         \ftn$z(4.9)$ & \ftn$x_{h,2}^{\text{IML FISTA}}(8.8)$ & \ftn$x_{h,50}^{\text{IML FISTA}}(20.5)$ & \ftn $z(4.9)$ & \ftn$x_{h,2}^{\text{IML FISTA}}(10.8)$ & \ftn$x_{h,50}^{\text{IML FISTA}}(19.2)$ 
         \\
         \includegraphics[trim={0 0 0 0},clip,width=0.15\textwidth]{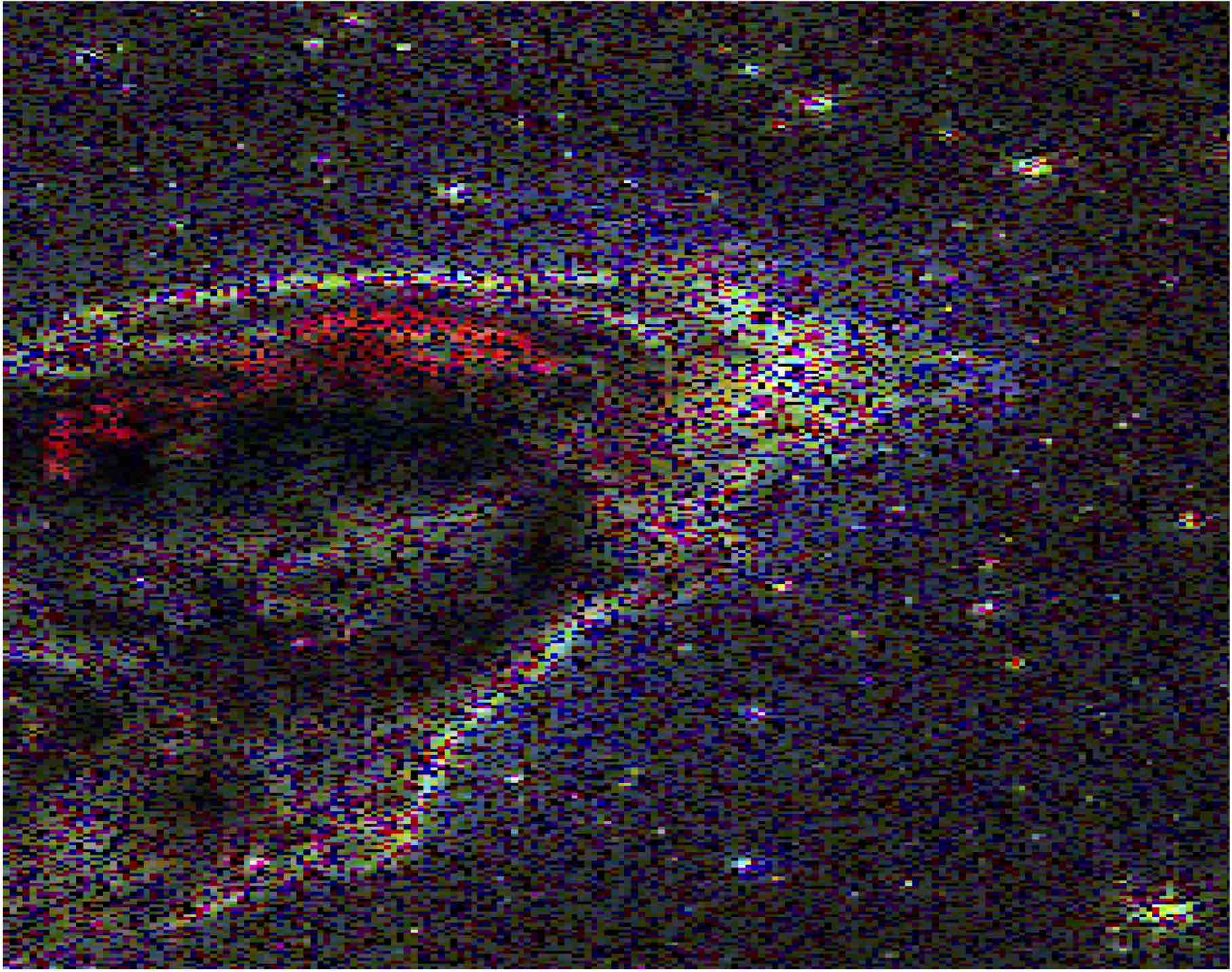} & \includegraphics[trim={0 0 0 0},clip,width=0.15\textwidth]{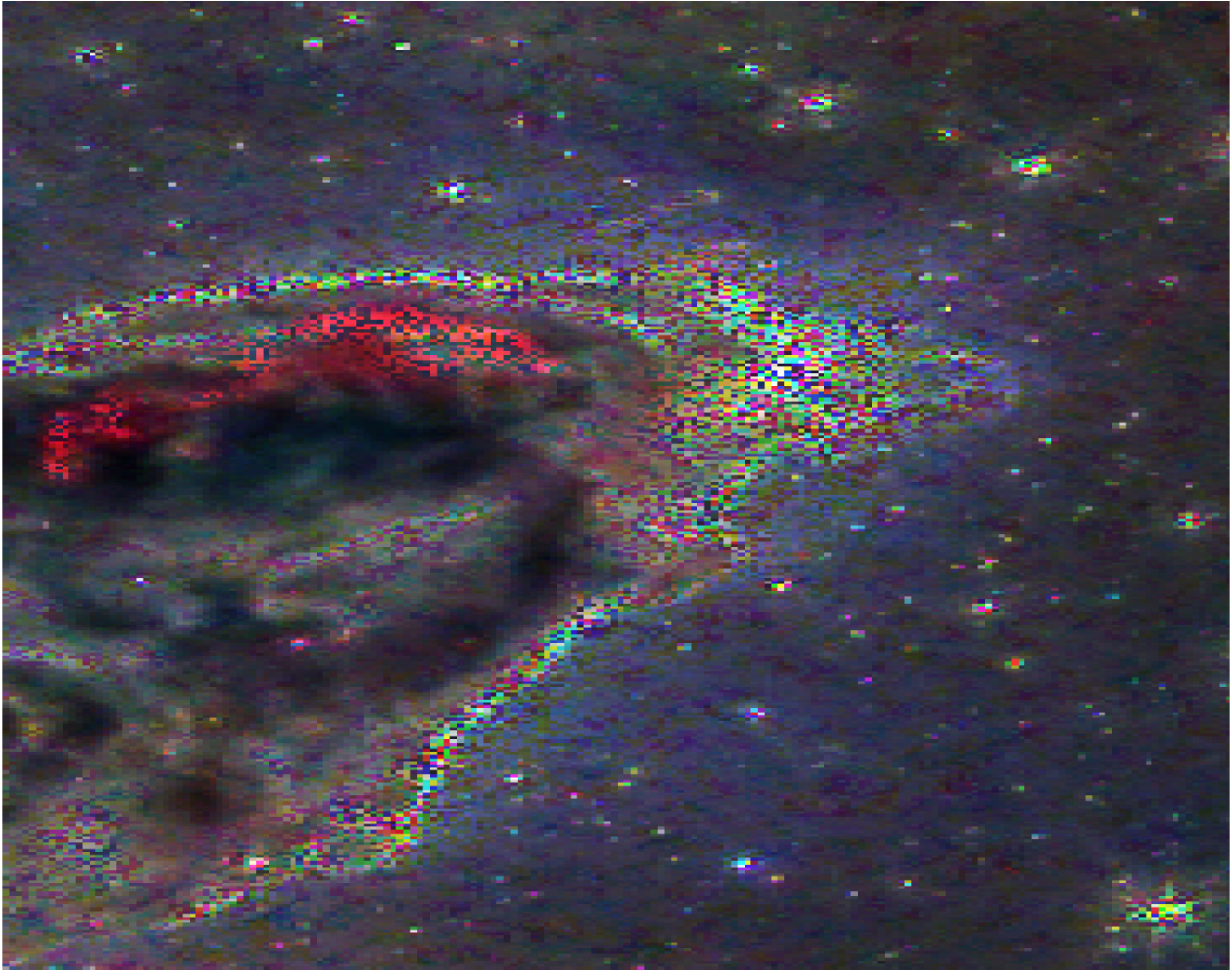} & \includegraphics[trim={0 0 0 0},clip,width=0.15\textwidth]{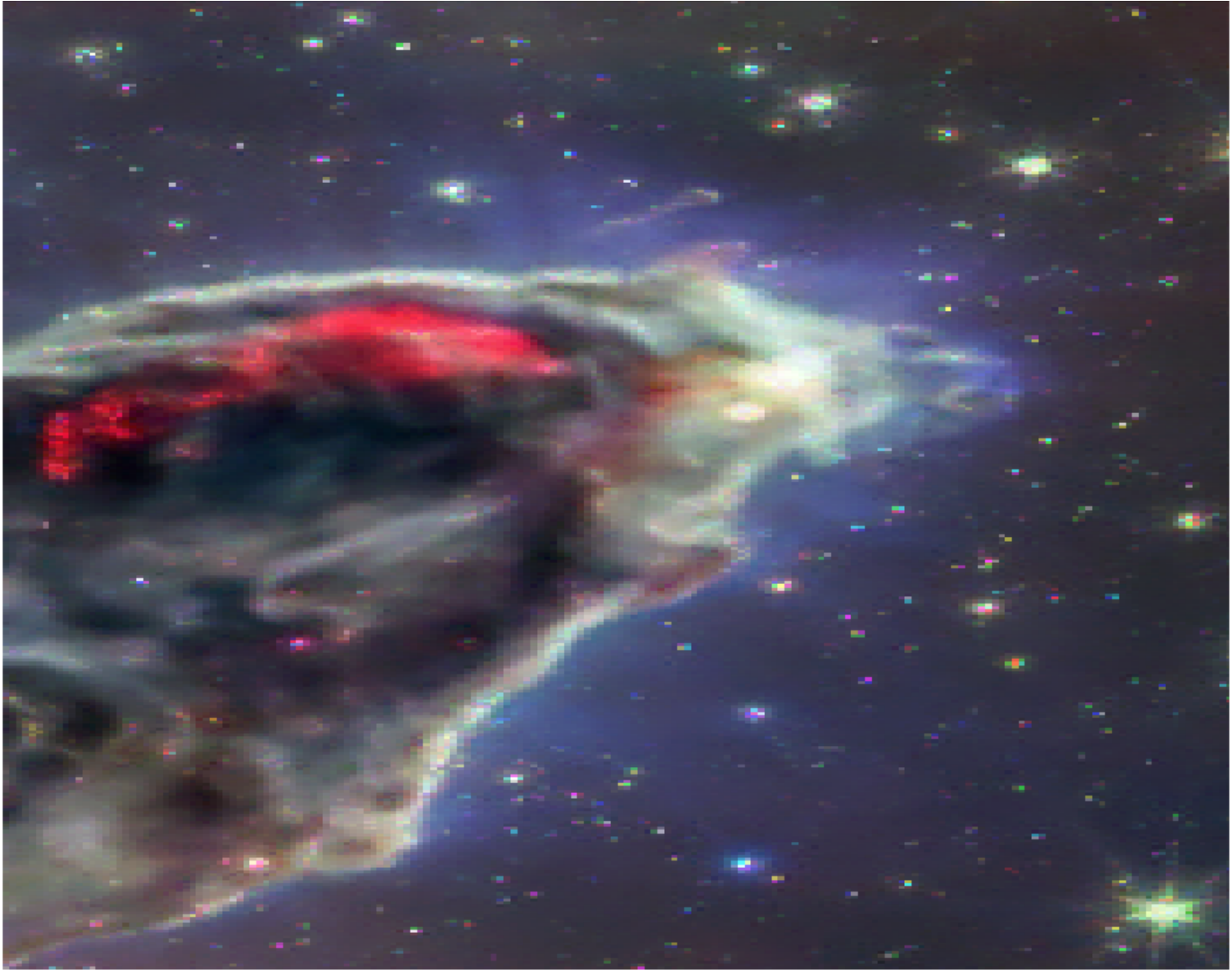} & \includegraphics[trim={0 0 0 0},clip,width=0.15\textwidth]{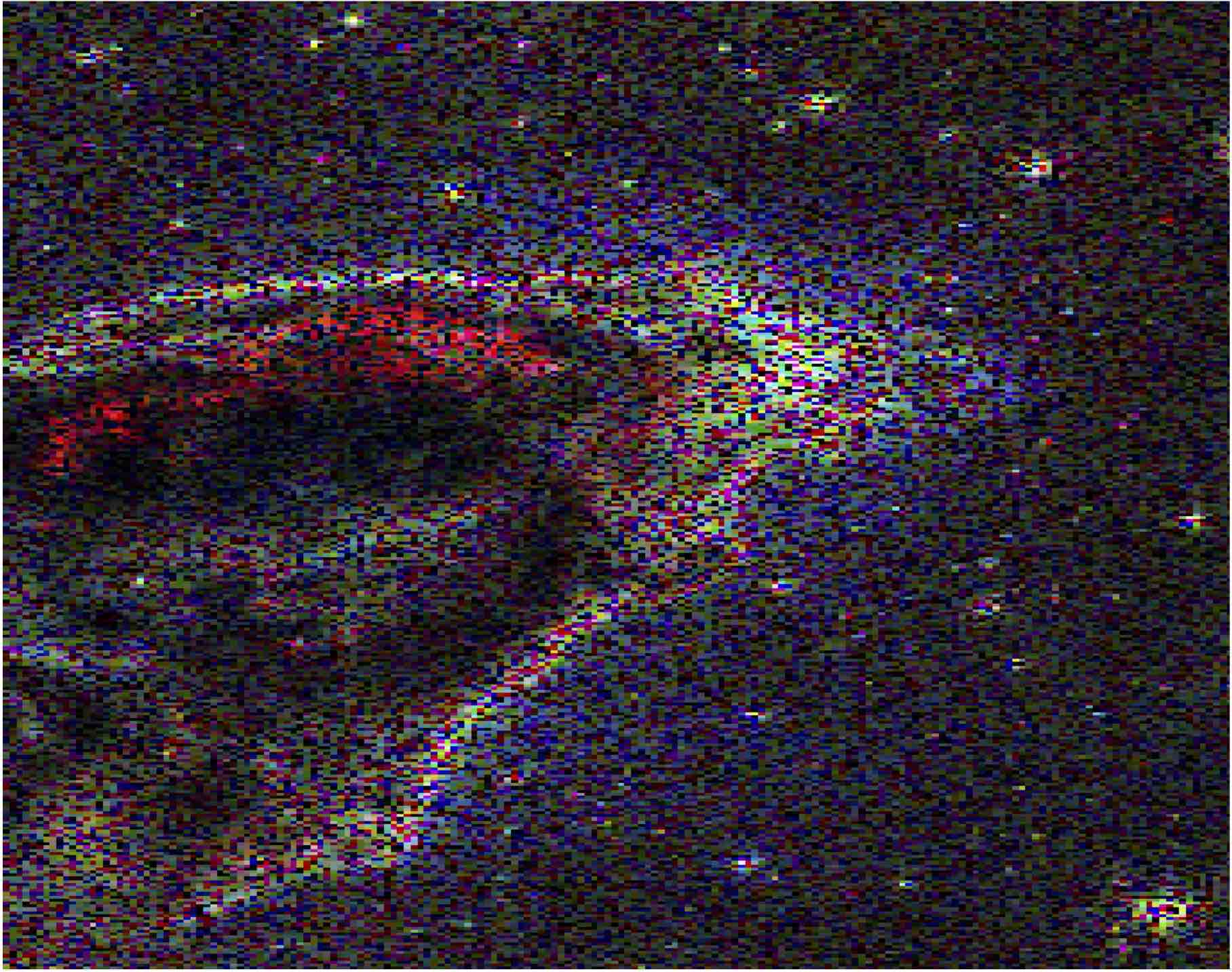} &  \includegraphics[trim={0 0 0 0},clip,width=0.15\textwidth]{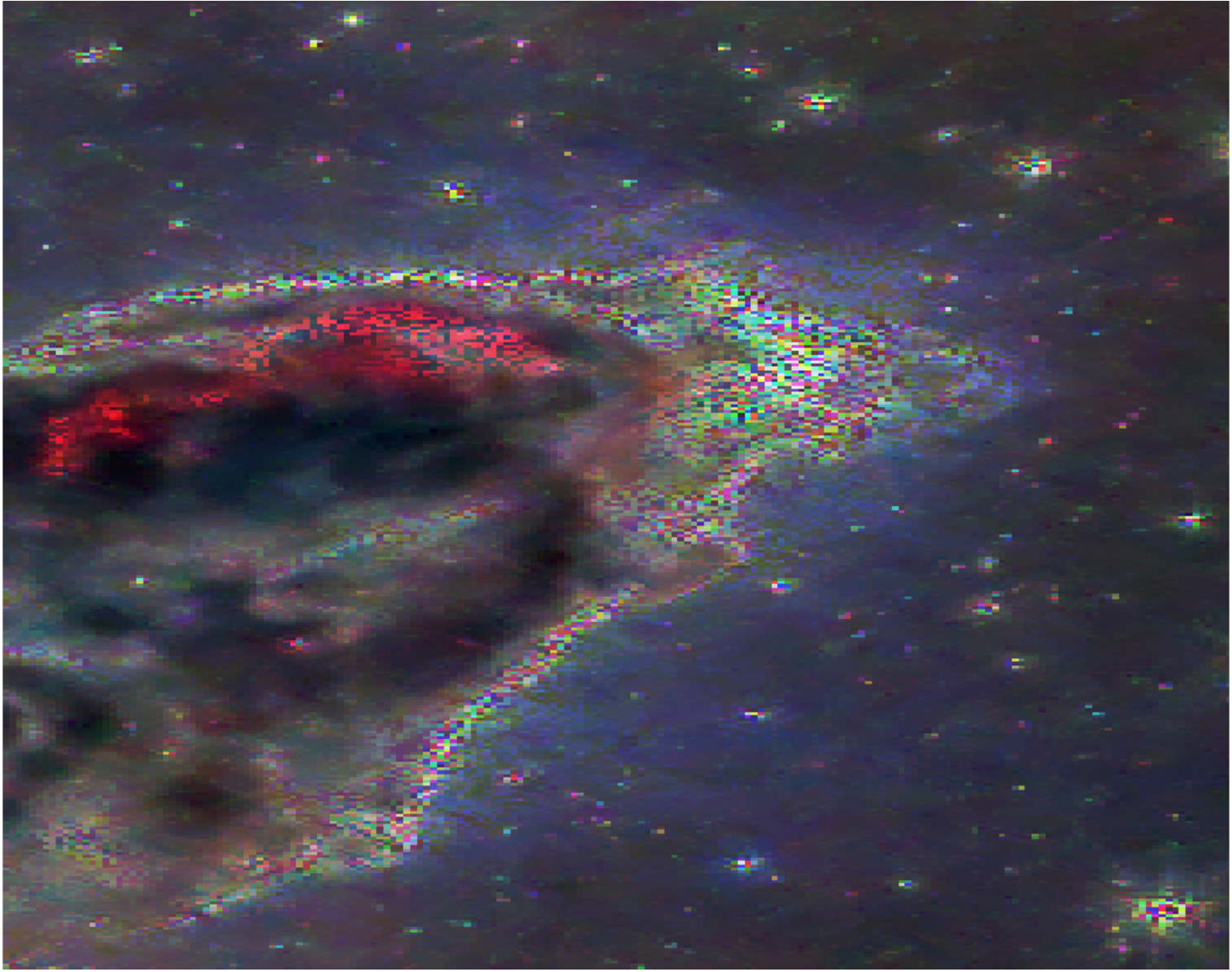} & \includegraphics[trim={0 0 0 0},clip,width=0.15\textwidth]{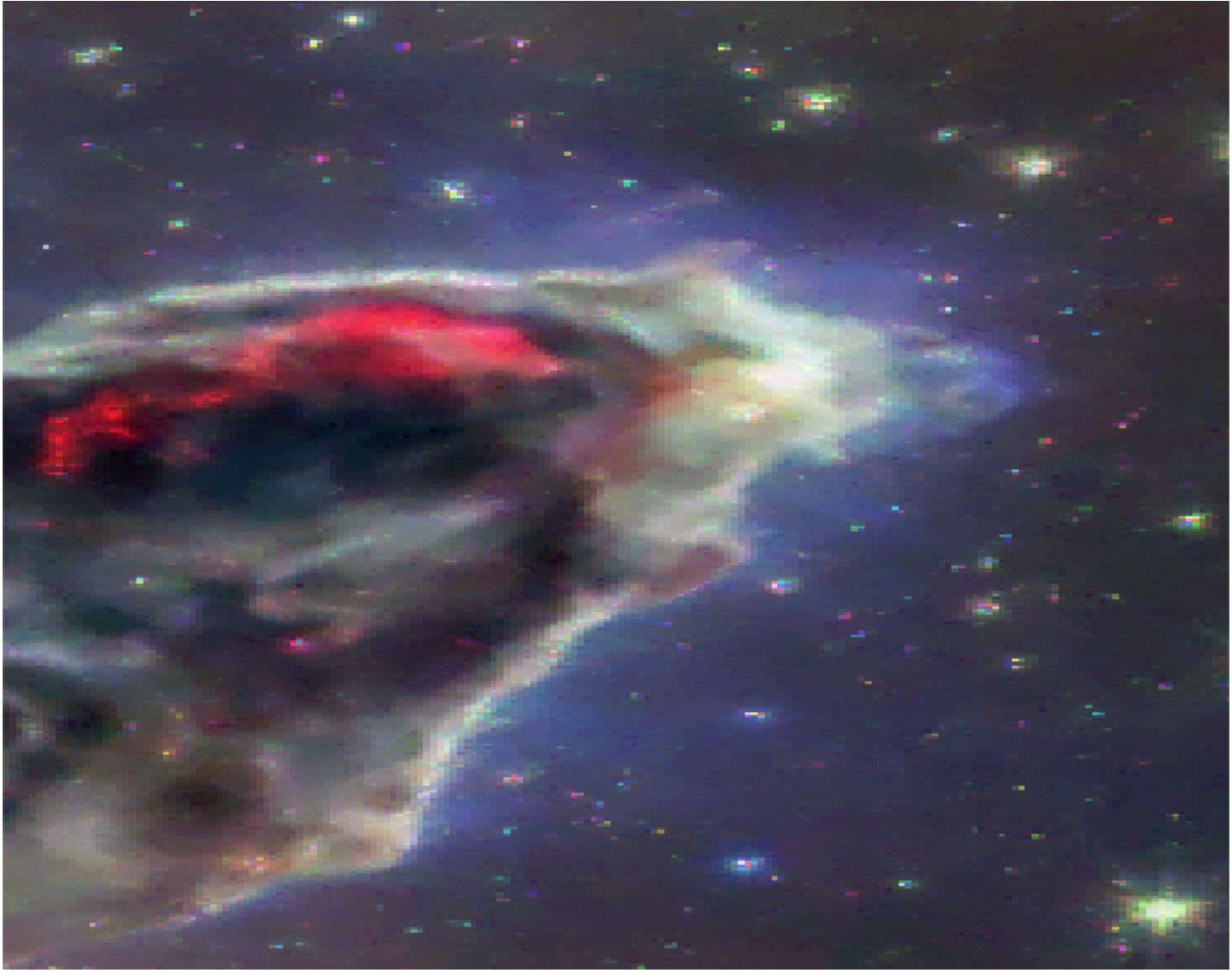}  \\
         \ftn $x$ & \ftn $x_{h,2}^{\text{FISTA}}(1.2)$ & \ftn $x_{h,50}^{\text{FISTA}}(15.3)$ & \ftn $x$ & \ftn $x_{h,2}^{\text{FISTA}}(0.8)$ & \ftn $x_{h,50}^{\text{FISTA}}(15.2)$ 
         \\
         \includegraphics[trim={0 0 0 0},clip,width=0.15\textwidth]{X-crop.pdf} & \includegraphics[trim={0 0 0 0},clip,width=0.15\textwidth]{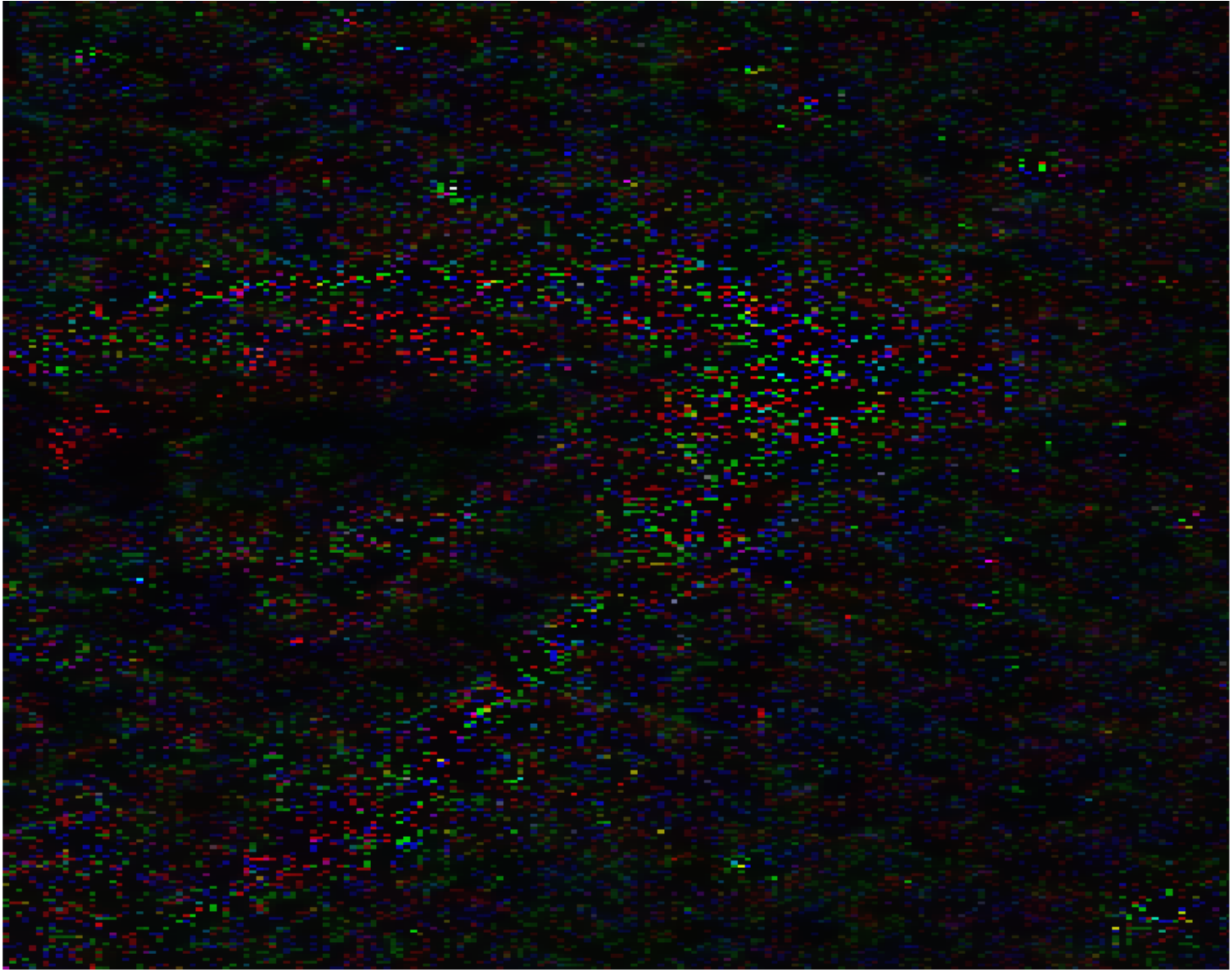} & \includegraphics[trim={0 0 0 0},clip,width=0.15\textwidth]{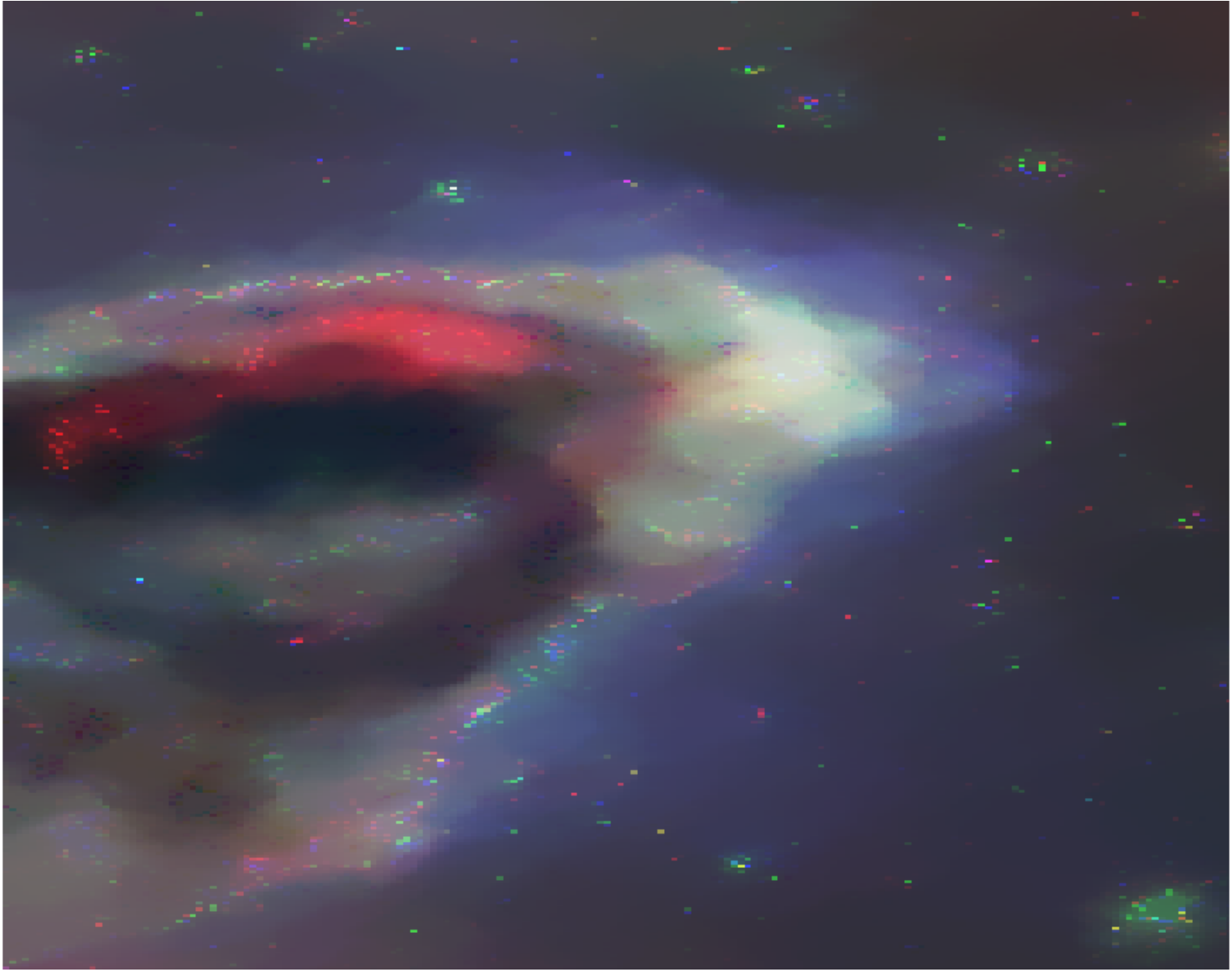} & \includegraphics[trim={0 0 0 0},clip,width=0.15\textwidth]{X-crop.pdf} 
         & \includegraphics[trim={0 0 0 0},clip,width=0.15\textwidth]{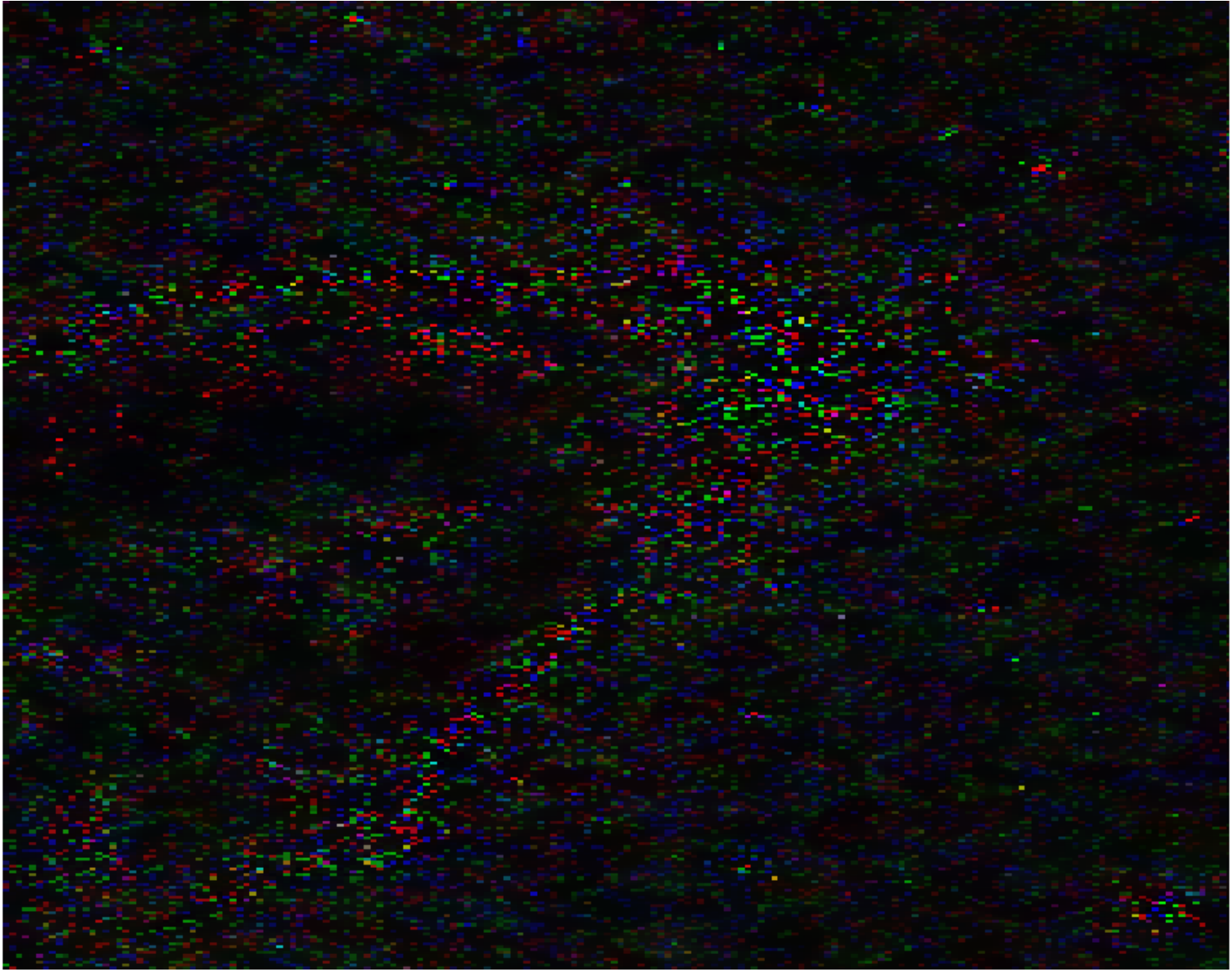} & \includegraphics[trim={0 0 0 0},clip,width=0.15\textwidth]{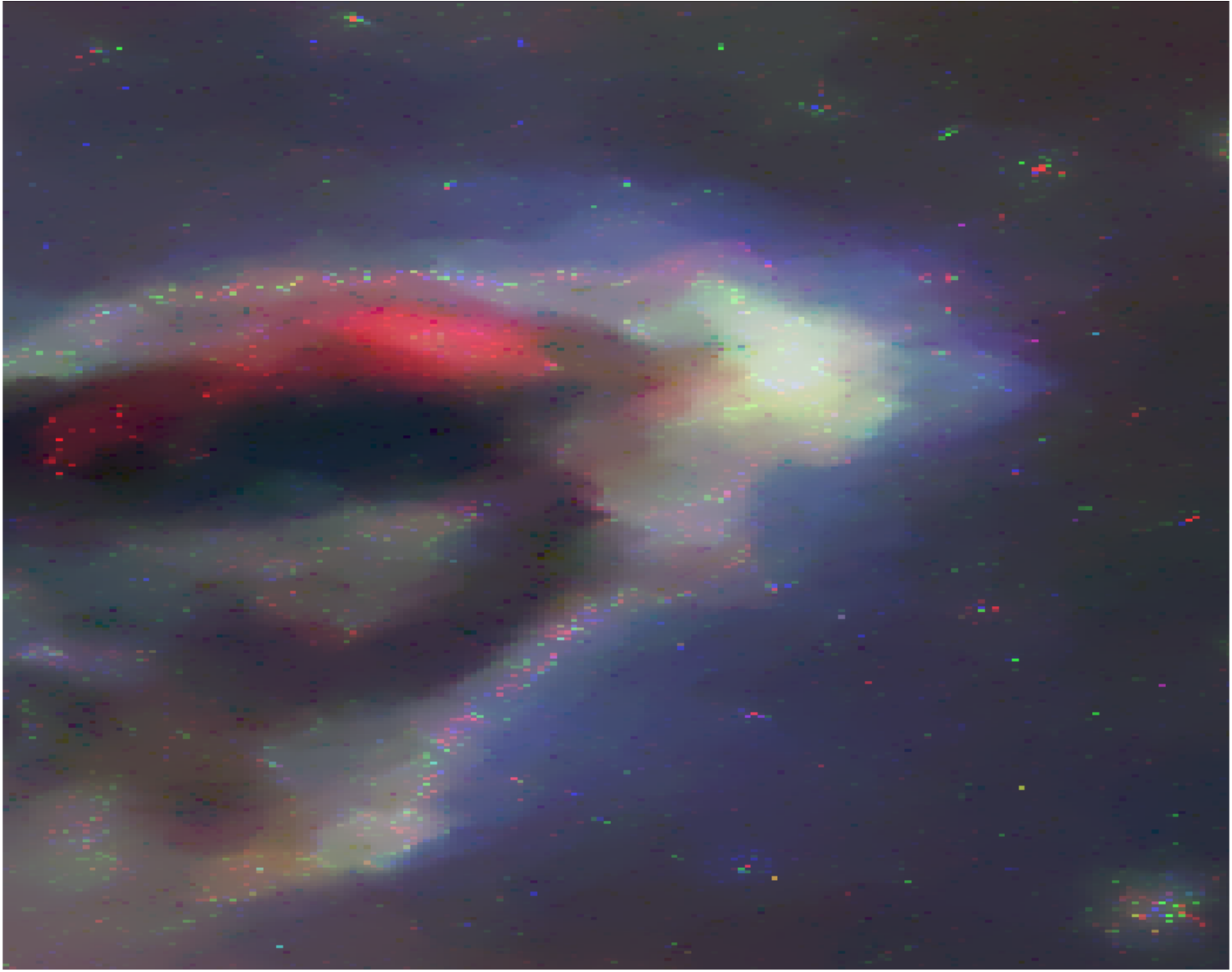}
         \\
         \ftn $z(0.8)$ & \ftn $x_{h,2}^{\text{IML FISTA}}(4.5)$ & \ftn $x_{h,50}^{\text{IML FISTA}}(15.3)$ & \ftn $z(0.8)$ & \ftn $x_{h,2}^{\text{IML FISTA}}(4.5)$ & \ftn $x_{h,50}^{\text{IML FISTA}}(15.2)$ 
         \\
         \includegraphics[trim={0 0 0 0},clip,width=0.15\textwidth]{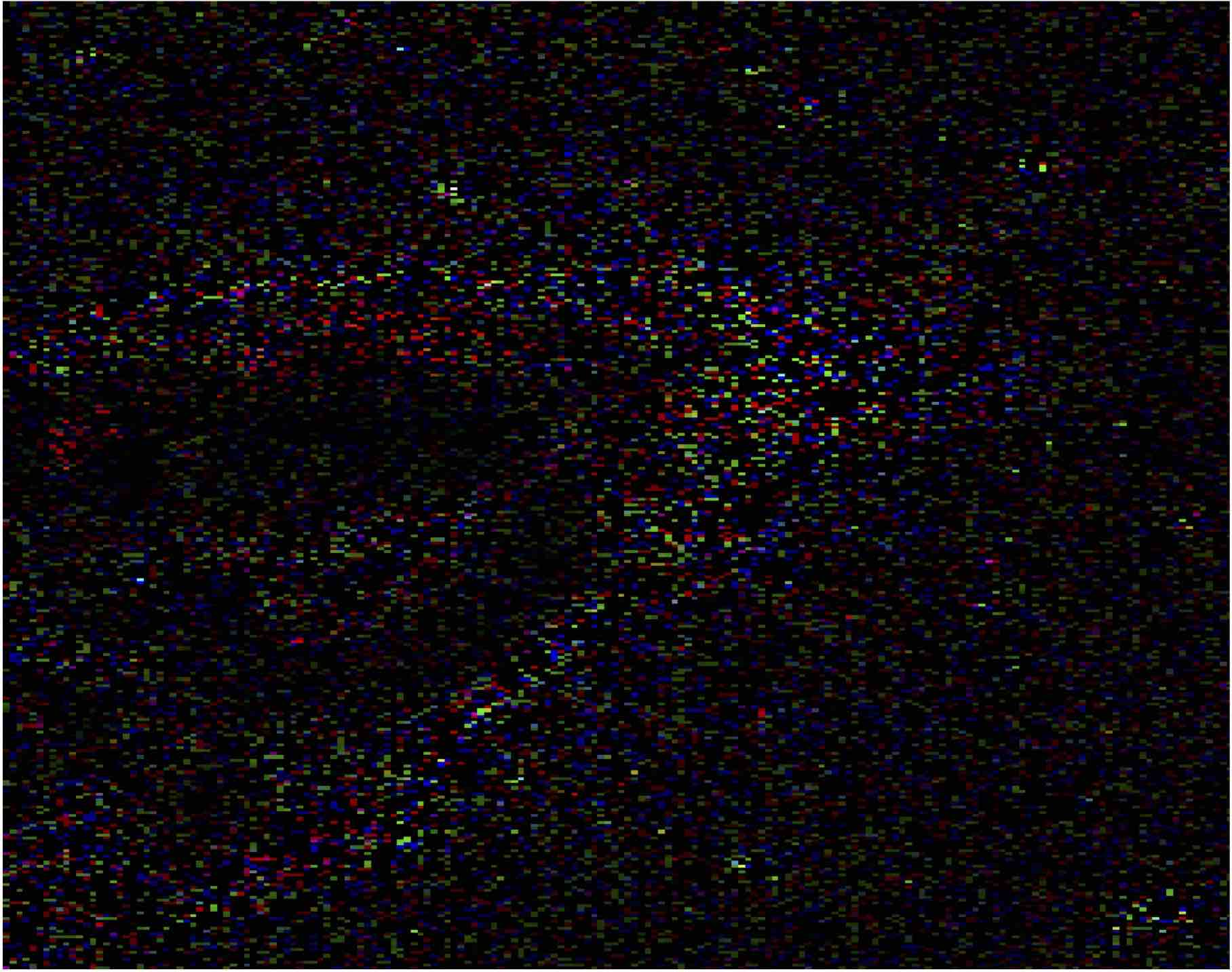} & \includegraphics[trim={0 0 0 0},clip,width=0.15\textwidth]{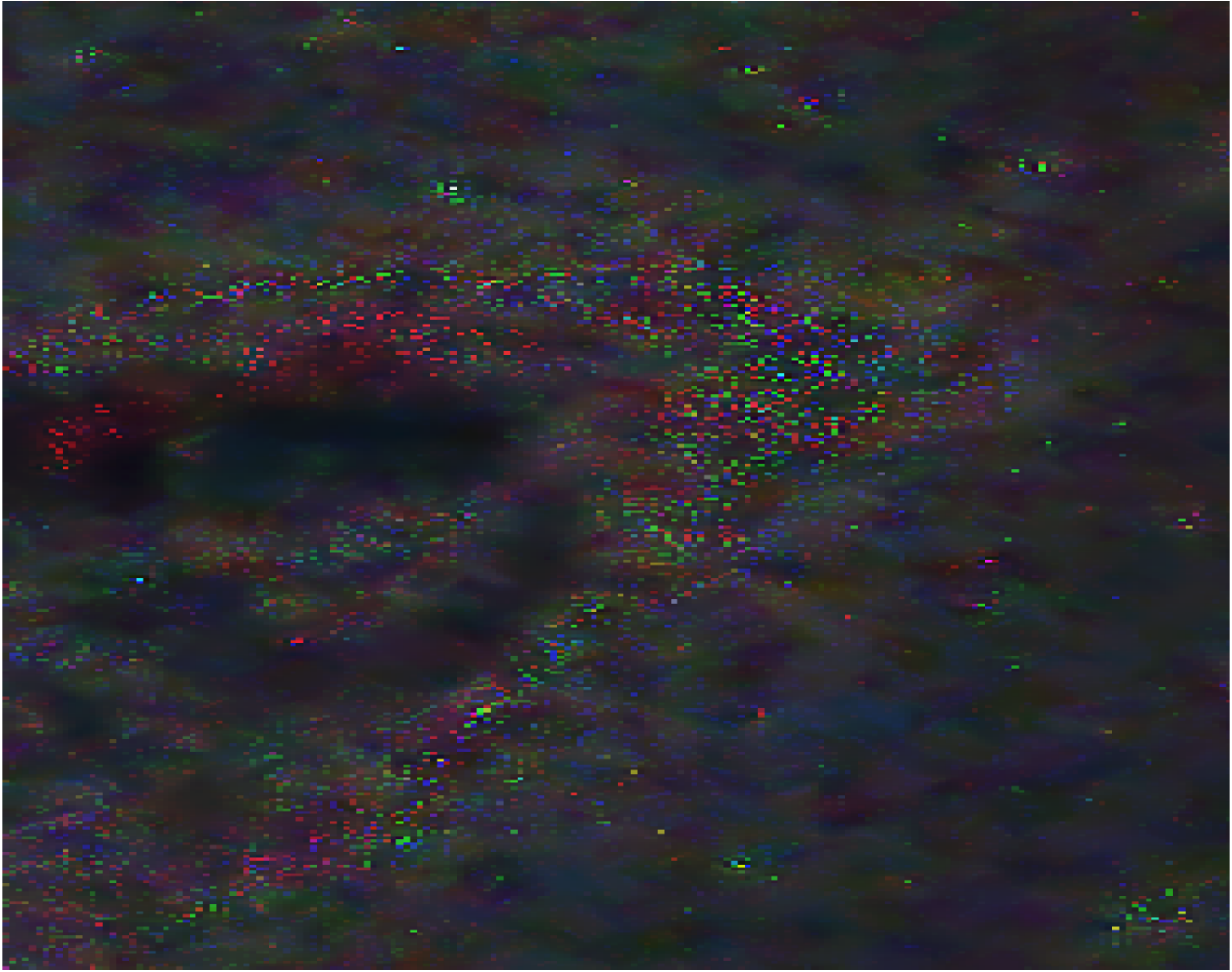} & \includegraphics[trim={0 0 0 0},clip,width=0.15\textwidth]{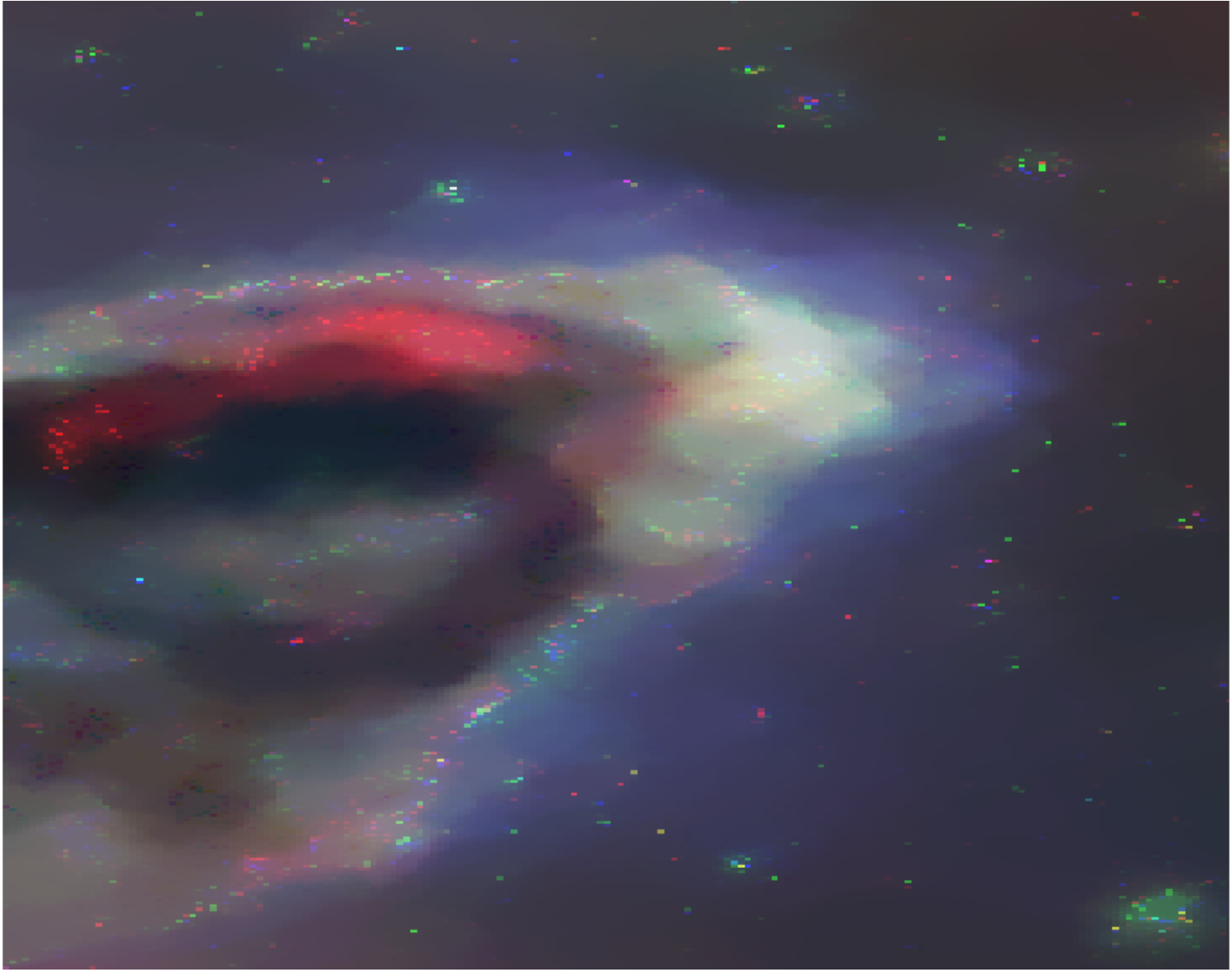} & \includegraphics[trim={0 0 0 0},clip,width=0.15\textwidth]{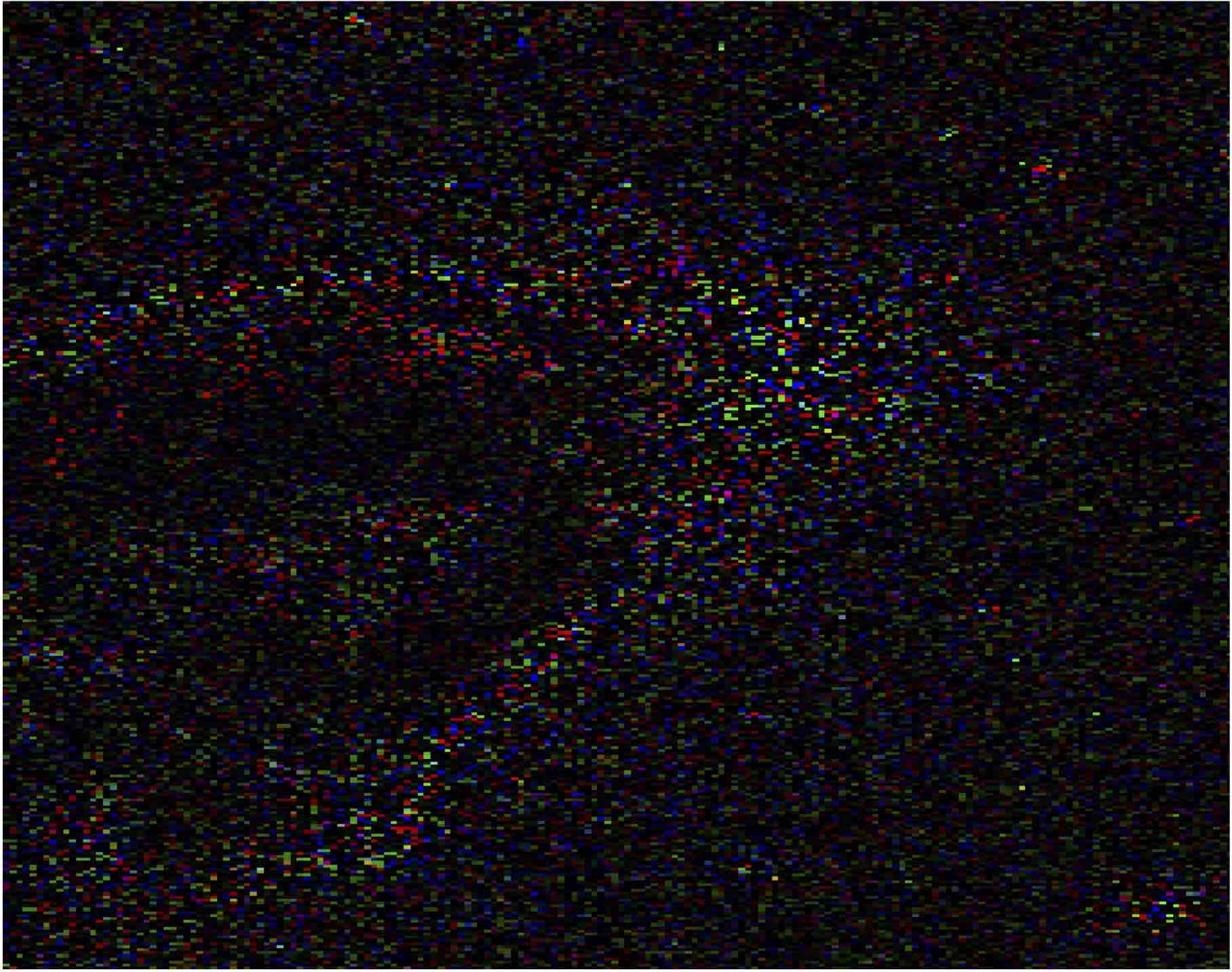} &  \includegraphics[trim={0 0 0 0},clip,width=0.15\textwidth]{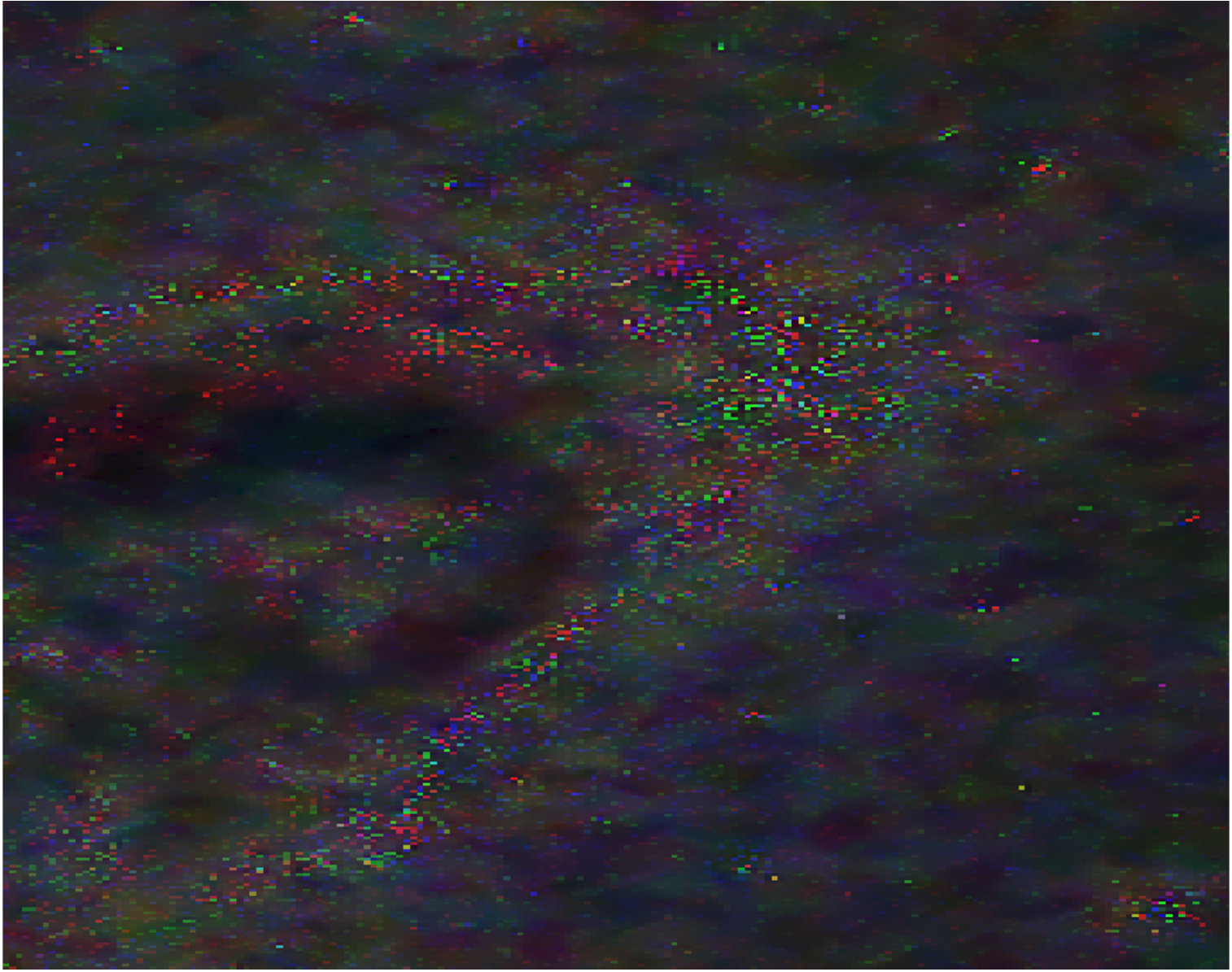} & \includegraphics[trim={0 0 0 0},clip,width=0.15\textwidth]{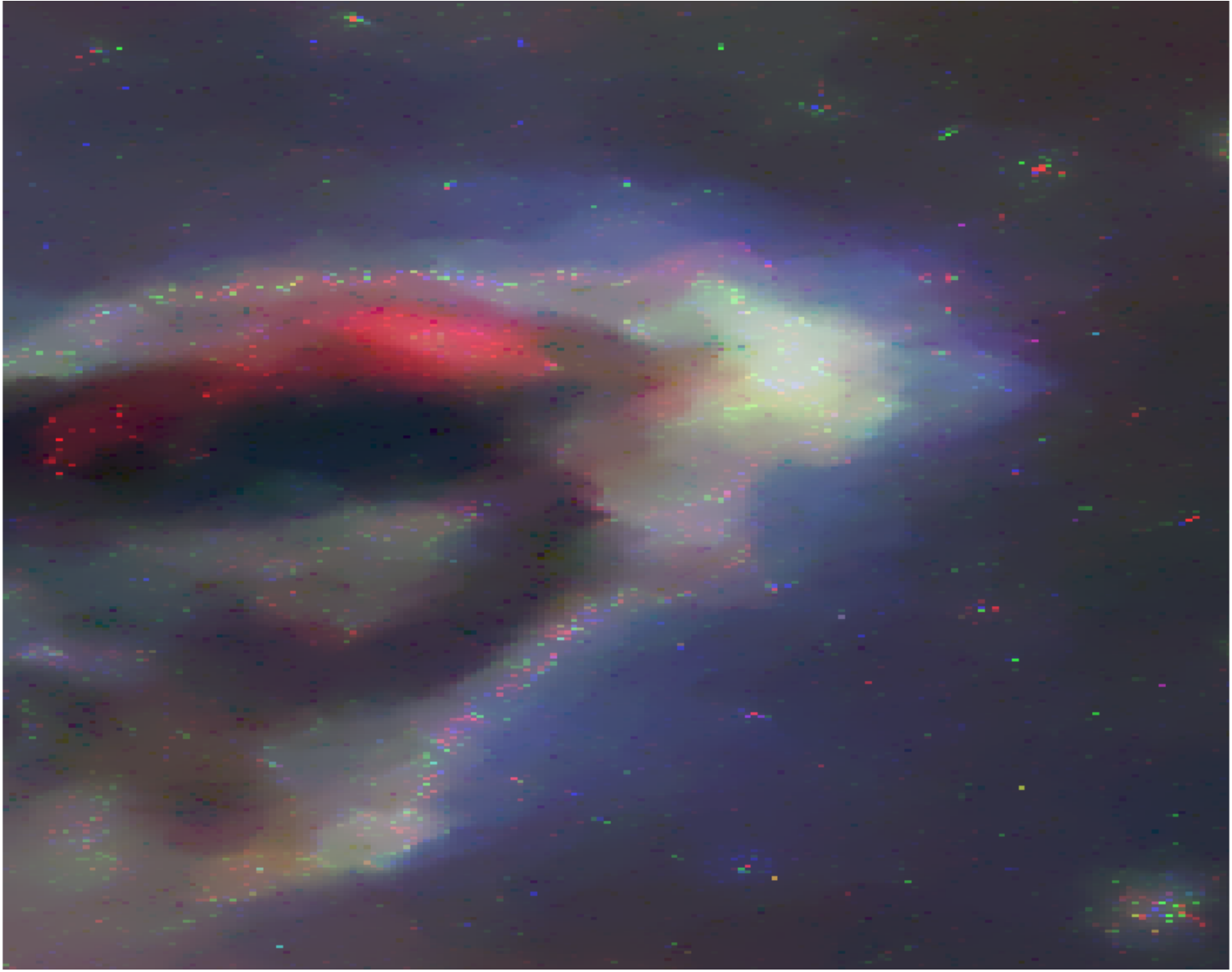}  \\
        \end{tabular}
    \end{center}
    \caption{Inpainting $\ell_{1,2}$-NLTV for the Pillars of Creation image. Small crop of the image at 2 iterations and after 50 iterations for FISTA (top row) and IML FISTA (bottom row) compared to the original ($x$) and degraded ($z$) images. For each image we report the SNR in dB. First column: $\sigma$(noise) $= 0.01$; second column: $\sigma$(noise) $= 0.05$. First row: missing pixels $50\%$; second row: missing pixels $90\%$. \vspace{-0.5em}} 
    \label{fig:NLTV_inpainting_imagesIR}
\end{figure}

Here again, we consider four scenarios based on two variables: the percentage of missing pixels and the standard deviation of the Gaussian noise $\sigma$(noise). These four scenarios are specified in Table \ref{tab:scenarios_inpainting}. 
\begin{table}[h]
\centering
\begin{tabular}{c|c|c|}
\cline{2-3}
\textbf{Inpainting $\backslash$ Noise} & $\sigma$(noise) $= 0.01$ & $\sigma$(noise) $= 0.05$ \\ \hline
\multicolumn{1}{|l|}{ missing pixels $50\%$} & low inpainting, low noise & low inpainting, high noise \\ \hline
\multicolumn{1}{|l|}{missing pixels $90\%$} &  high inpainting, low noise   & high inpainting, high noise \\ \hline 
\end{tabular} \vspace{-1em}
\caption{Four scenarios of inpainting degradation with additive Gaussian noise. \vspace{-0.5cm}}
\label{tab:scenarios_inpainting}
\end{table}
For each of these four scenarios we tested two regularizations: $\ell_{1,2}$-TV and $\ell_{1,2}$-NLTV. In this case we only report the results obtained with the $\ell_{1,2}$-NLTV prior. 

Again, in all cases, the objective function decreases faster with  IML FISTA than with FISTA, proving that the computational cost of multilevel steps is almost negligible. The two performed coarse corrections bring a considerable advantage to the minimization achieved with IML FISTA (Figure \ref{fig:NLTV_inpainting_FHIR}). Also, one can remark that given a capped sub-iterations budget, IML FISTA reaches the smallest possible value, faster than FISTA.  Comparing the two methods after only two iterations of IML FISTA and of FISTA, is particularly convincing as we can observe it in Figure \ref{fig:NLTV_inpainting_imagesIR}.
Moreover, as it was already the case for the deblurring task, IML FISTA outperforms FISTA in terms of convergence speed, specifically when the original image is heavily corrupted.

As for the deblurring task, we display the results under the same degradation contexts (i.e., inpainting and noise) for the Yellow Car image. We reproduce in Figure \ref{fig:NLTV_inpainting_FHIM} the evolution of the objective function when the regularization is the $\ell_{1,2}$-NLTV norm. In contrast to the deblurring case, IML FISTA still performs better than FISTA for an inpainting task on a relatively small image size. This suggests that the dependency of IML FISTA's performances to the problem dimension is clearly linked to the degradation context.
\begin{figure}
    \centering
    \includegraphics[trim={2.5em 2em 0 2.7em},clip,width=0.4\textwidth]{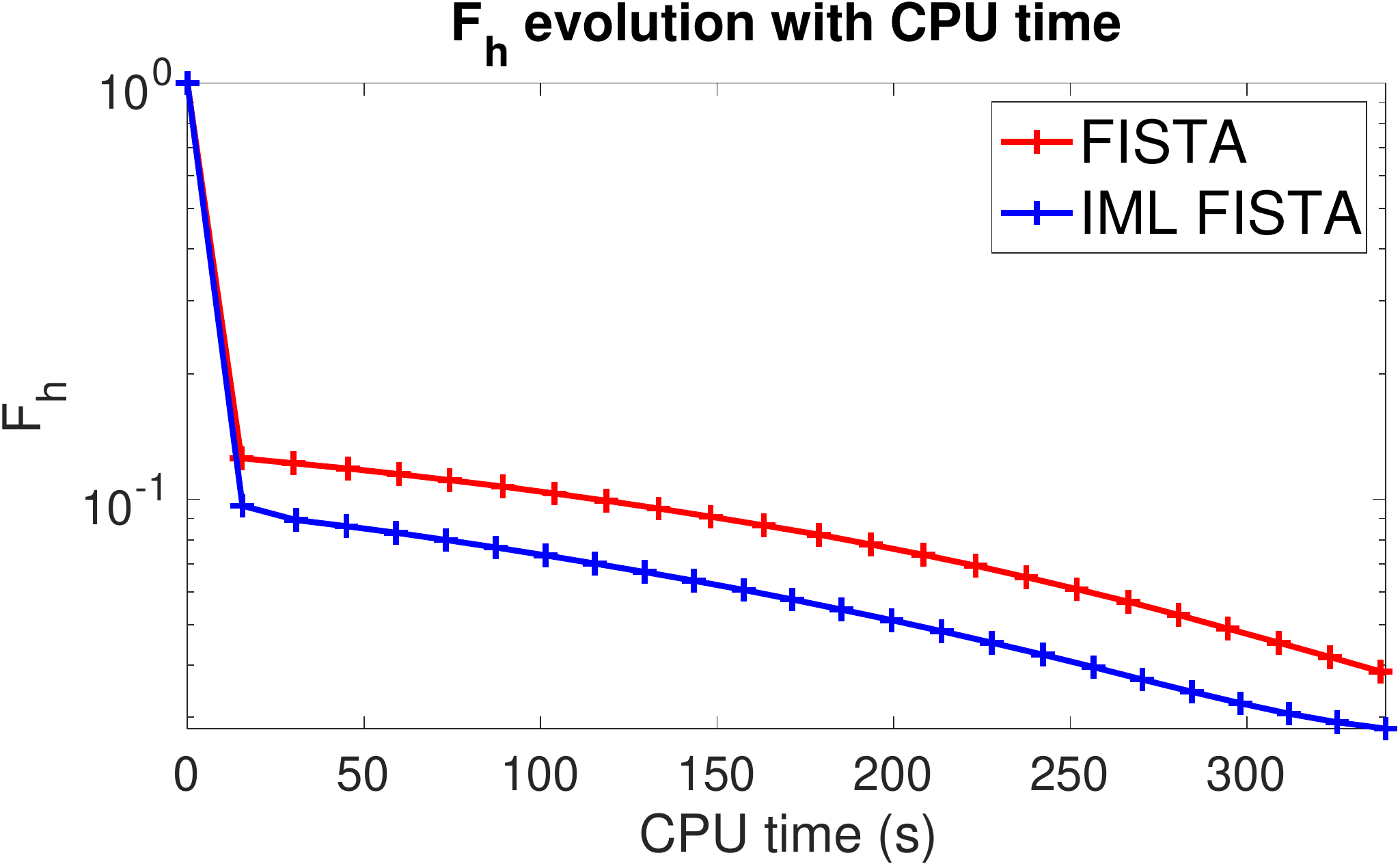}
    \hspace{2em}
    \includegraphics[trim={2.5em 2em 0 2.7em},clip,width=0.395\textwidth]{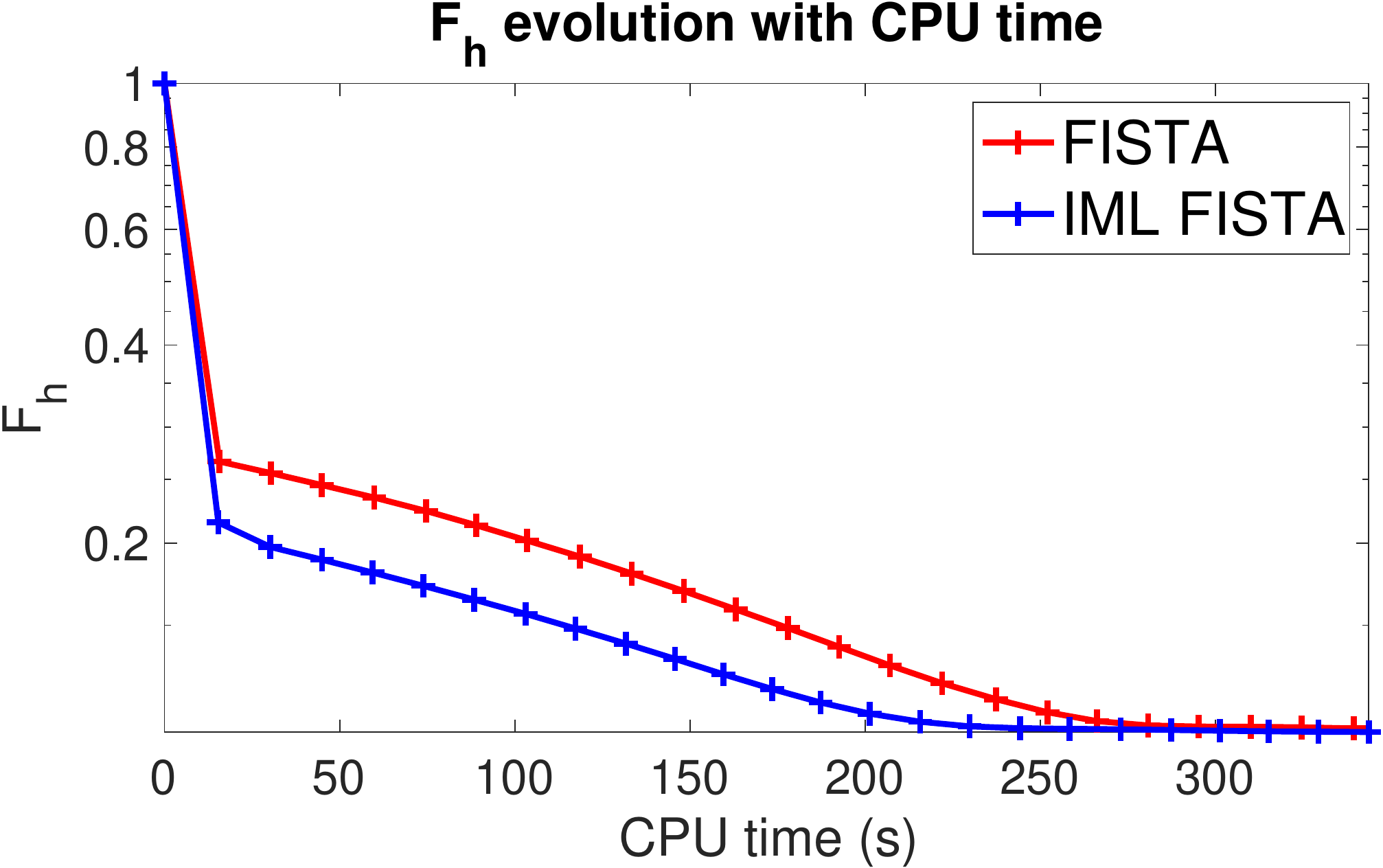} \\ \vspace{0.1em}
    \includegraphics[trim={2.5em 2em 0 2.7em},clip,width=0.4\textwidth]{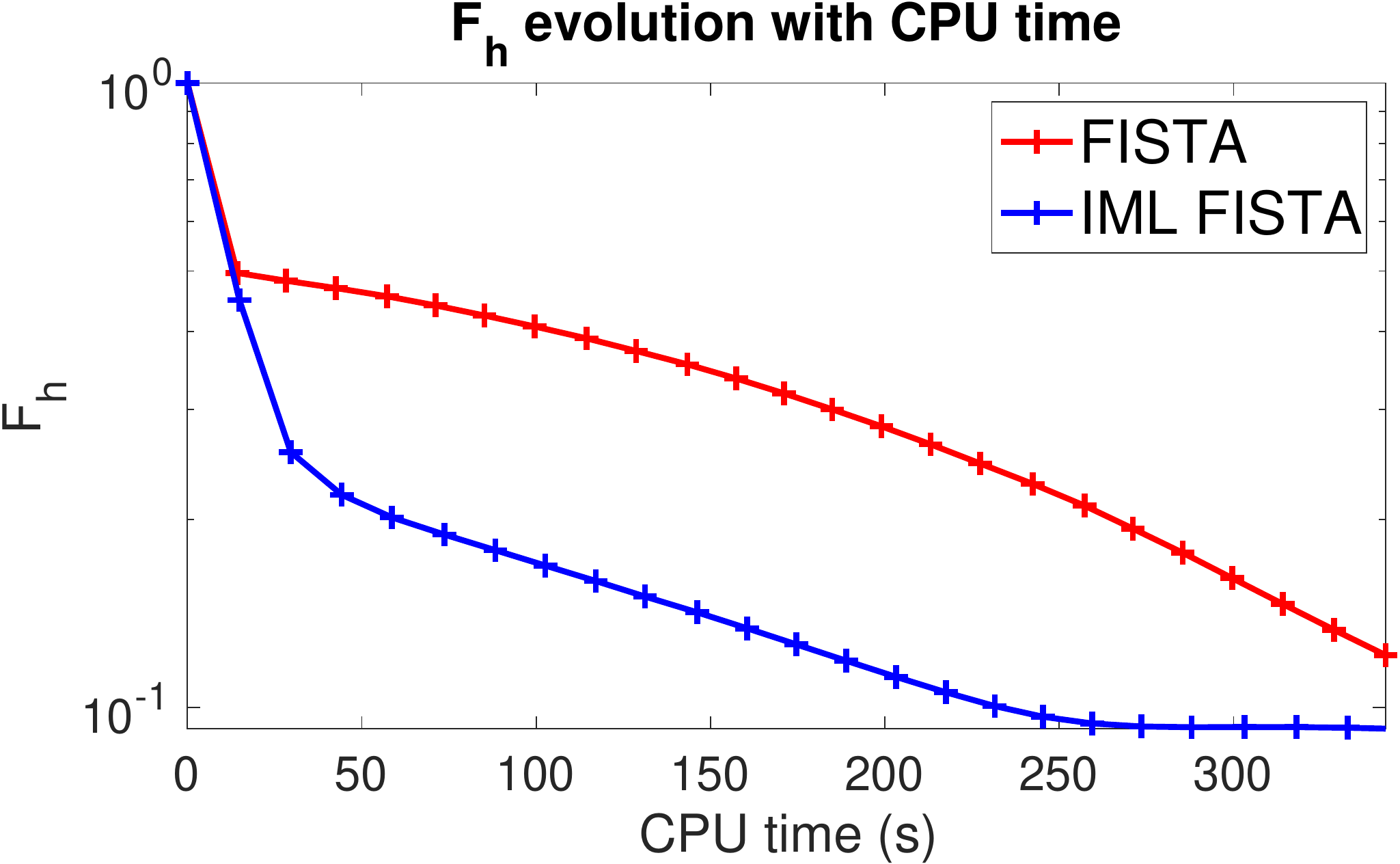}
    \hspace{2em}
    \includegraphics[trim={2.5em 2em 0 2.7em},clip,width=0.395\textwidth]{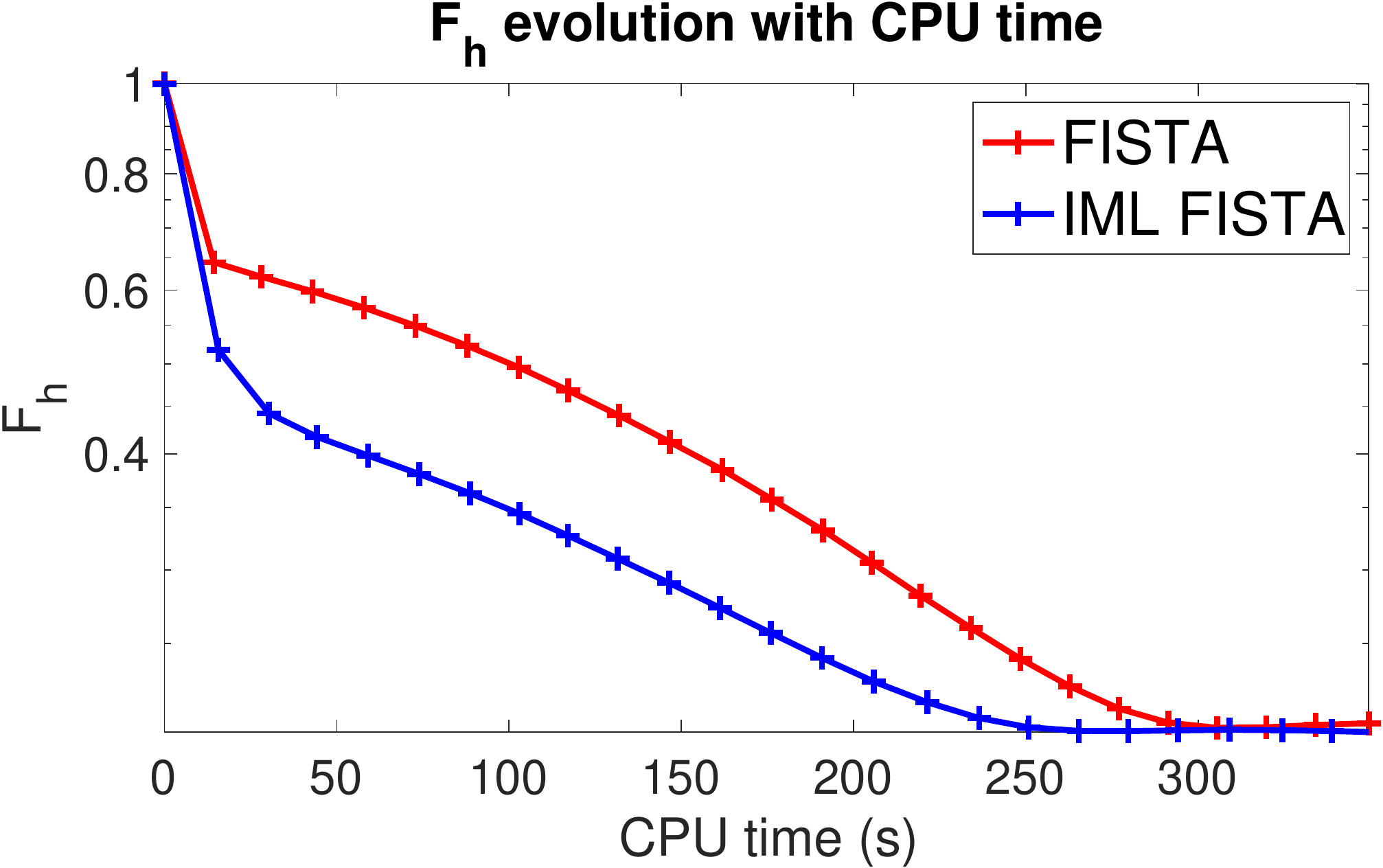}
    \caption{Inpainting $\ell_{1,2}$-NLTV for the Yellow Car image. Objective function (normalized with initialization value) vs CPU time (sec). First column: $\sigma$(noise) $= 0.01$; second column: $\sigma$(noise) $= 0.05$. First row: missing pixels $50\%$; second row: missing pixels $90\%$. For each plot, the crosses represent iterations of the algorithm. \vspace{-1.5em}} 
    \label{fig:NLTV_inpainting_FHIM}
\end{figure}

\subsection{Impact of the information transfer operators}
In this section we investigate the influence of the choice of the CIT on the performance of our multilevel algorithm. 

From the experiments of the previous section, we claim that the performances result from the combination of a faithful coarse minimization and a good design of the information transfer operators. The latter should allow to capture information over wider regions than what is typically done in multilevel optimization \cite{nash2000,parpas2016,parpas2017} where the filter used to generate the information transfer operators is of a rather small size.

To test this hypothesis we compare the CIT built using the ``Haar" wavelet (filter size equal to $2$), the ``Daubechies 20" wavelet (filter size equal to $20$), and the ``Symlets 10" wavelet. The last two have the same quadrature mirror filter length. For deblurring problems, no significant difference was observed between these three CITs. The influence was more noticeable for inpainting problems, and the results are displayed in Figure \ref{fig:NLTV_inpainting_IT}. The principal factor seems to be the length of the filter, this is not surprising since it determines the domain extension over which pixels are aggregated. Nonetheless, even with the Haar wavelet, IML FISTA reaches lower objective function values faster than FISTA, meaning that the use of coarse models is beneficial to the optimization regardless of the chosen wavelet-based CIT. 
\begin{figure}
    \centering
    \includegraphics[trim={2.5em 2em 0 2.7em},clip,width=0.4\textwidth]{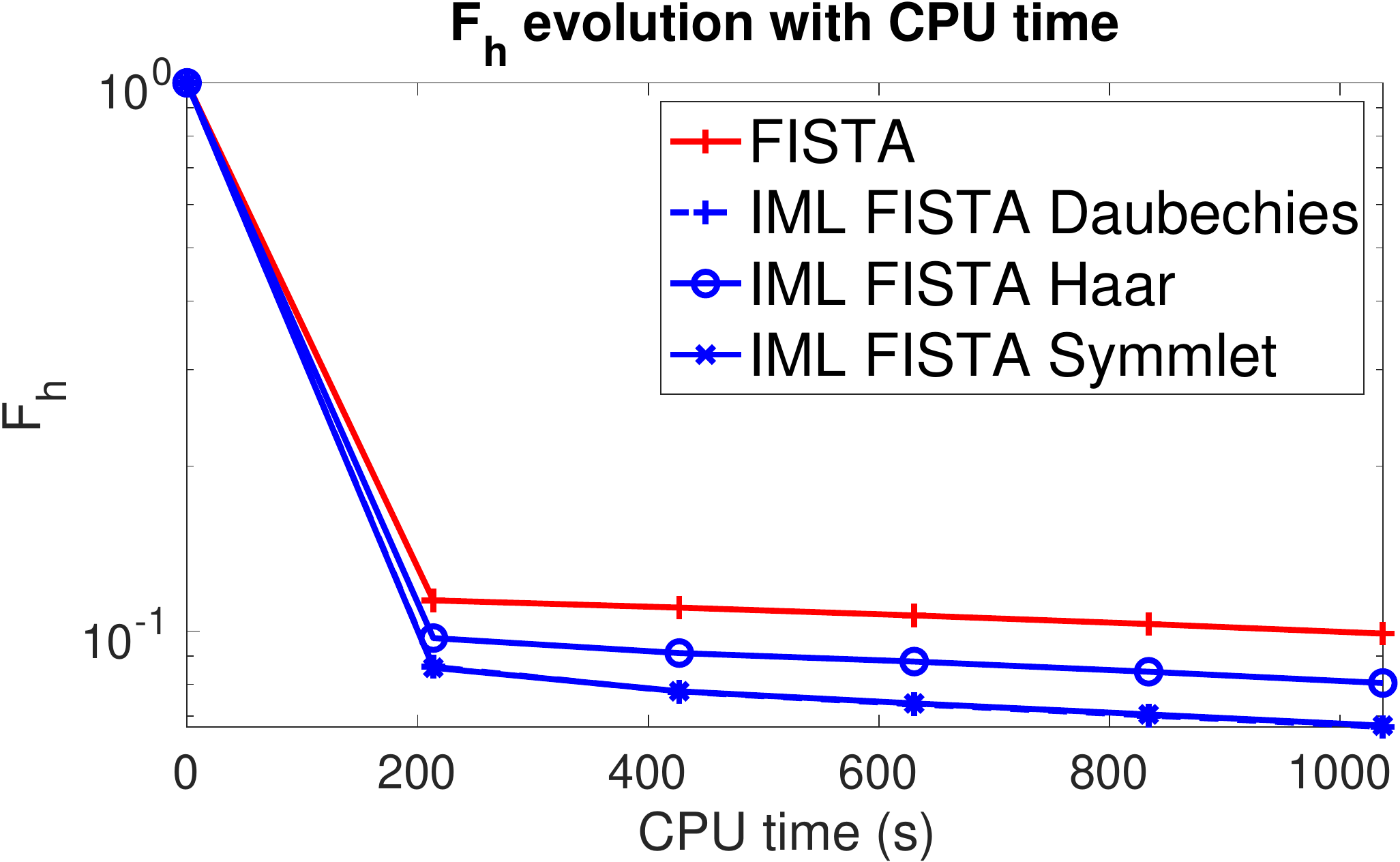}
    \caption{Inpainting $\ell_{1,2}$-NLTV for the Pillars of Creation image. Objective function (normalized with initialization value) vs CPU time (sec). $\sigma$(noise) $= 0.01$, missing pixels $50\%$. Comparison of information transfer operators: ``Haar", ``Daubechies 20" and ``Symlets 10". \vspace{-1.5em}} 
    \label{fig:NLTV_inpainting_IT}
\end{figure}
\subsection{IML FISTA: application to hyperspectral images restoration} 
We conclude this experimental section by applying IML FISTA to an hyperspectral image (HSI) restoration problem. The acquisition of hyperspectral images is of tremendous importance in many fields such as remote sensing \cite{lu2020} or art analysis \cite{khan2018,pillay2019}.  It is often impaired by missing data  and noise due to cameras defect and blurring effects. Several methods have been designed to handle them \cite{rasti2018}. Among them, the variational approach is of great interest but suffers from a high computational cost \cite{rasti2018}. This approach is a particular case of Problem \eqref{eq:ml_pb} where
\begin{equation}
\vspace{-0.5em}
    \label{eq:pb_hyp}
    F_h(x) = \frac{1}{2} \Vert \A_h x - z \Vert_2^2 + \lambda \sum_{i=1}^{N_h} \Vert (\D_h x)^i \Vert_{*},
\end{equation}
Here a coarse level can be derived naively from the nature of those images: high correlation between bands is observed on hyperspectral images and thus it seems natural to exploit this redundancy to reduce the dimension and restore the image. 

\paragraph{Notations} Formally, we denote $x^{(i,b)} = x^{(i_1,i_2,b)}$  the pixel located at the spatial index $i=(i_1,i_2) \in \{1,\ldots,N_r\} \times \{1,\ldots,N_c\}$ and band $b\in \{1 ,\ldots,L\}$ of HSI $x$. $x$ can be represented as a hypercube of size $N_h = L \times N_r \times N_c$. We denote $w^{(b)}$ the wavelength associated with the $b$ band. We also note $\mu(\Delta)$ the mean of the wavelength differences $\Delta_b = w^{(b+1)}-w^{(b)}$ for all bands $b$ and $\sigma(\Delta)$ the associated standard deviation. Here we are interested in restoring a $33 \times 512 \times 512$ hyperspectral image of an engraving\footnote{\href{https://personalpages.manchester.ac.uk/staff/d.h.foster/Hyperspectral_Images_of_Illustrated_Manuscripts.html}{St Christopher} : acquired by the authors of \cite{foster2006}.} which can be seen in Figure \ref{fig:NLTV_inpainting_blur_HSI}.

\paragraph{Data fidelity term} To perform the restoration of such images, we model the degradation as the combination of a gaussian blur and a mask on the pixels (in this order). 


\paragraph{Regularization term} We consider the structure tensor non-local TV  penalization proposed in \cite{chierchia2014}, that consists in applying the nuclear norm $\Vert \cdot \Vert_{*}$ on a tensor concatenating the non-local finite difference for every band. 
The nuclear norm 
allows us to take into account the strong correlation between the bands to improve the reconstruction.

\paragraph{Information transfer operator} We aim to reduce the size of an HSI by reducing the number of bands. 
A small wavelength difference between two successive bands suggests a strong correlation between them. This similarity can be difficult to measure in our case (for a review of methods see \cite{jia2016}) because the observed HSI is very degraded. We have therefore chosen a simple heuristic to infer this correlation, independent of the content of the band. For all $b\in \{1,\ldots,L\}$, 
%
{every two consecutive bands whose wavelengths difference is smaller than $\mu(\Delta)+\sigma(\Delta)$ 
are averaged and merged in a single one. Other bands are kept.}{}
We apply the same operation on $\A_h$ by averaging the blocks that represent the bands. 
\paragraph{Multilevel parameters} The proposed multilevel algorithm has then $5$ levels, and at the last level the HSI is of size $3 \times 512 \times 512$. The configuration remains the same as presented in previous experiments. In a previous work \cite{lauga2023methodes} we have seen that $d=0.5$ was a good trade-off between relaxing the necessary decrease of the proximity operator estimation's error and having a sufficient decrease of the objective function at each iteration with the inertia. 
The two algorithm were stopped after a given computation time accounting for $50$ iterations of FISTA, and $41$ of IML FISTA.

\paragraph{Results} The evolution of the objective function and the reconstructed hyperspectral image  of this experiment are displayed in Figure \ref{fig:NLTV_inpainting_blur_HSI}. Essentially, the decrease of the objective function obtained by IML FISTA is faster than what it is obtained by FISTA on about 50 iterations while only calling \textbf{ML} twice.
\begin{figure}
    \centering
    \begin{center}
        \setlength{\tabcolsep}{3pt}
        \begin{tabular}{cccc}
        \multirow{4}{*}[-0.4in]{\includegraphics[trim={2.5em 2em 0 2.7em},clip,width=0.4\textwidth]{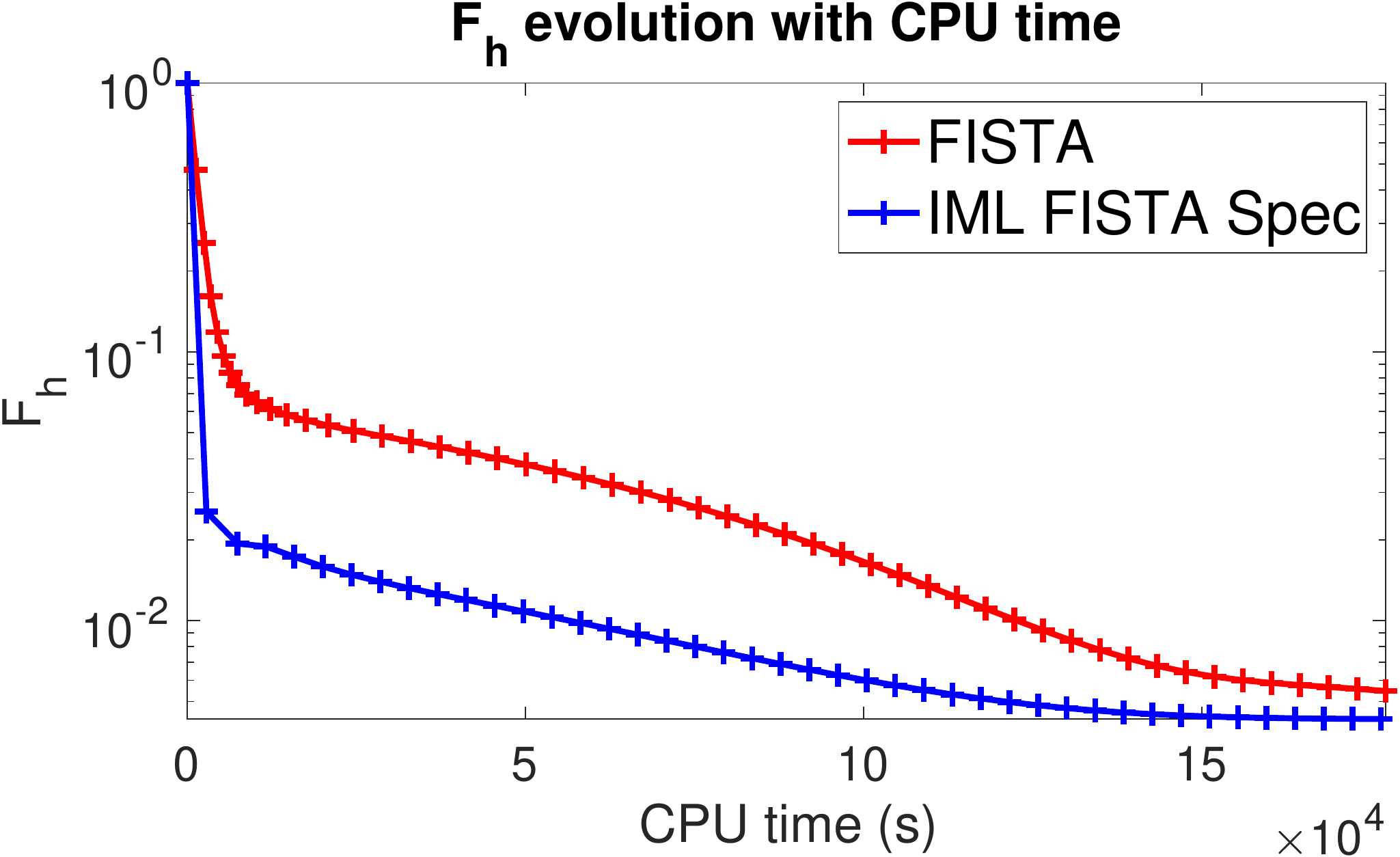}}
        & \ftn $x$ & \ftn $x_{h,2}^{\text{FISTA}}(6.5)$ & \ftn $x_{h,\text{end}}^{\text{FISTA}}(28.3)$ \\
         & \includegraphics[trim={0 0 0 0},clip,width=0.15\textwidth]{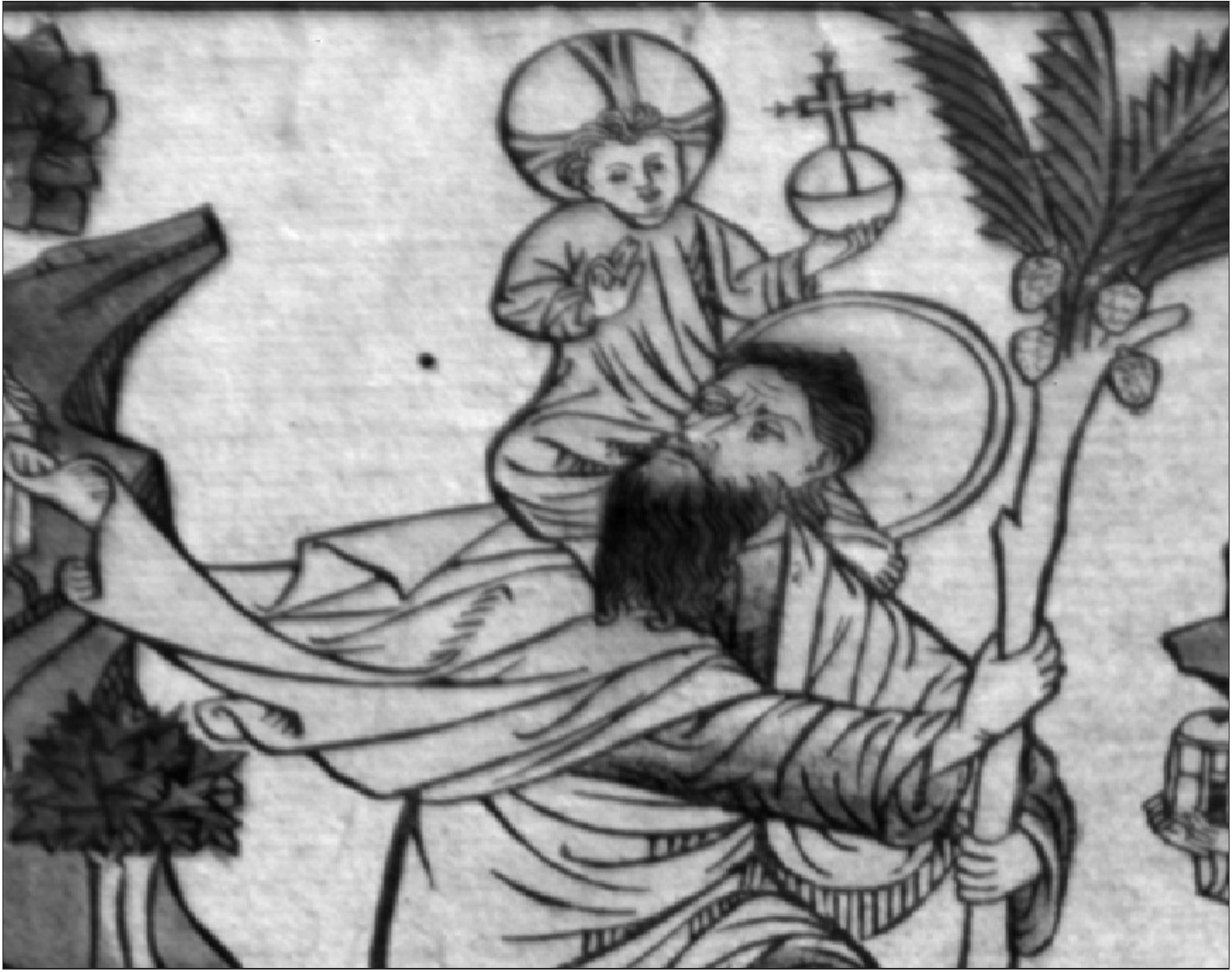} & \includegraphics[trim={0 0 0 0},clip,width=0.15\textwidth]{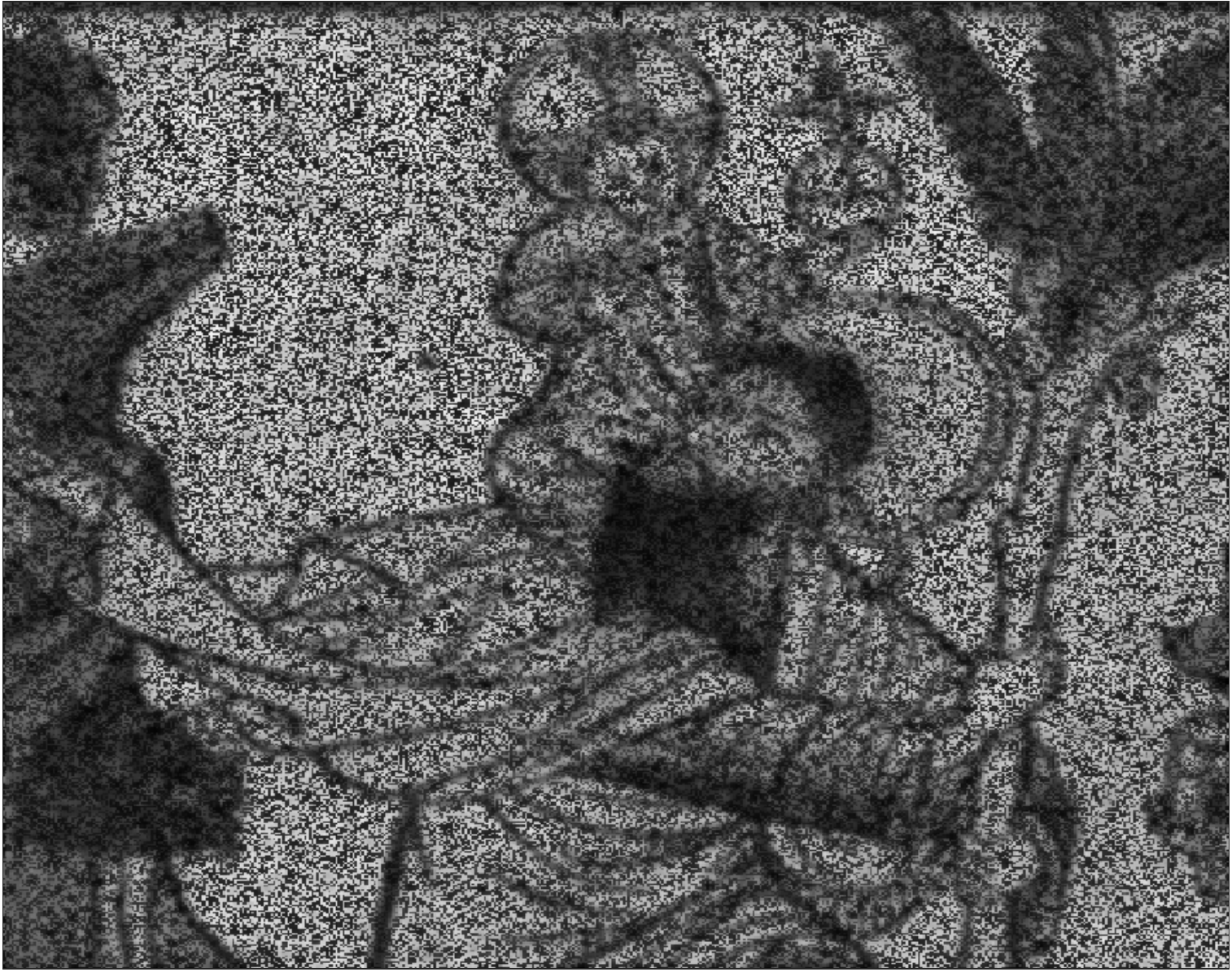} & \includegraphics[trim={0 0 0 0},clip,width=0.15\textwidth]{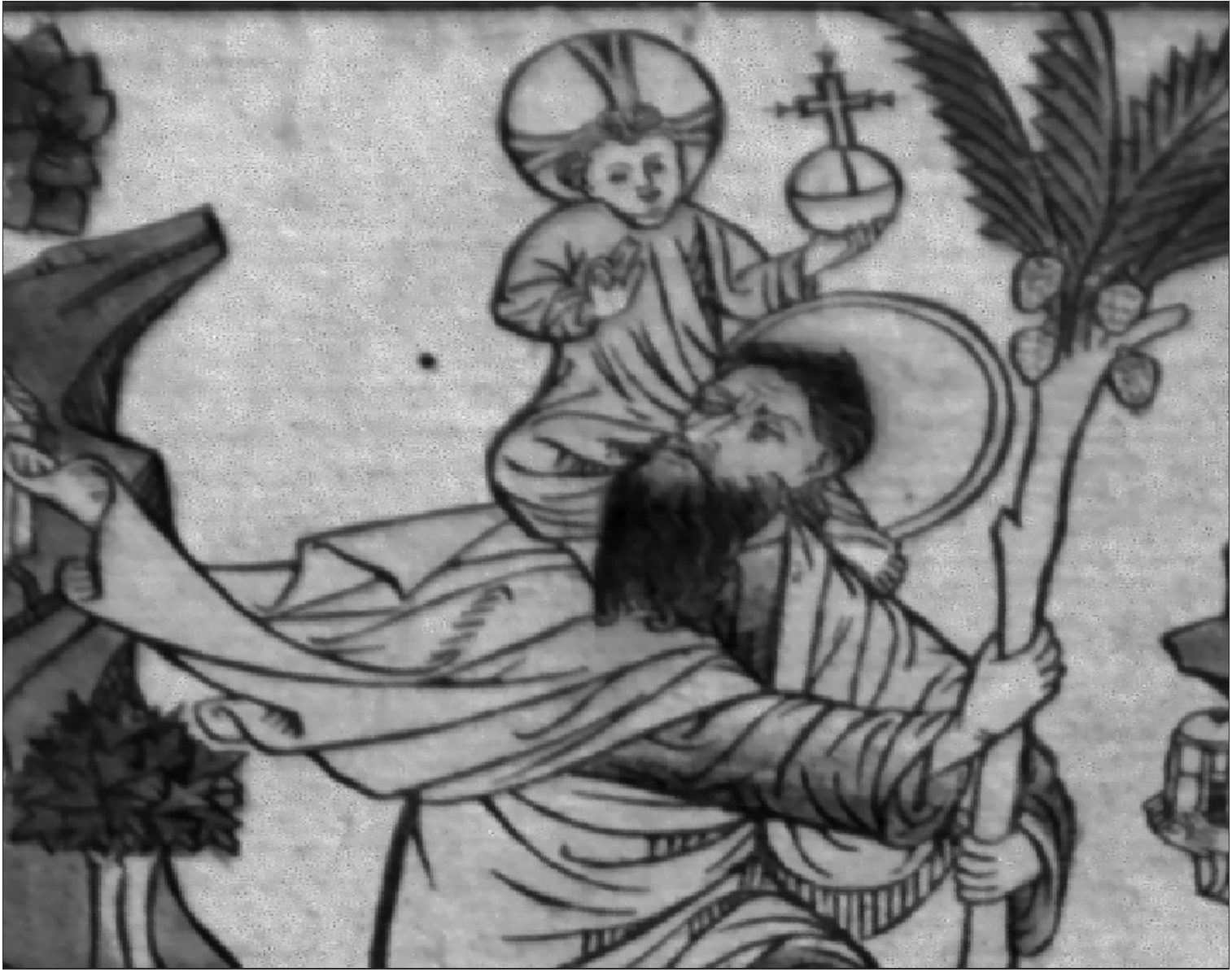} 
         \\
         & \ftn$z(3)$ & \ftn$x_{h,2}^{\text{IML FISTA}}(18.0)$ & \ftn$x_{h,\text{end}}^{\text{IML FISTA}}(33.7)$
         \\
         & \includegraphics[trim={0 0 0 0},clip,width=0.15\textwidth]{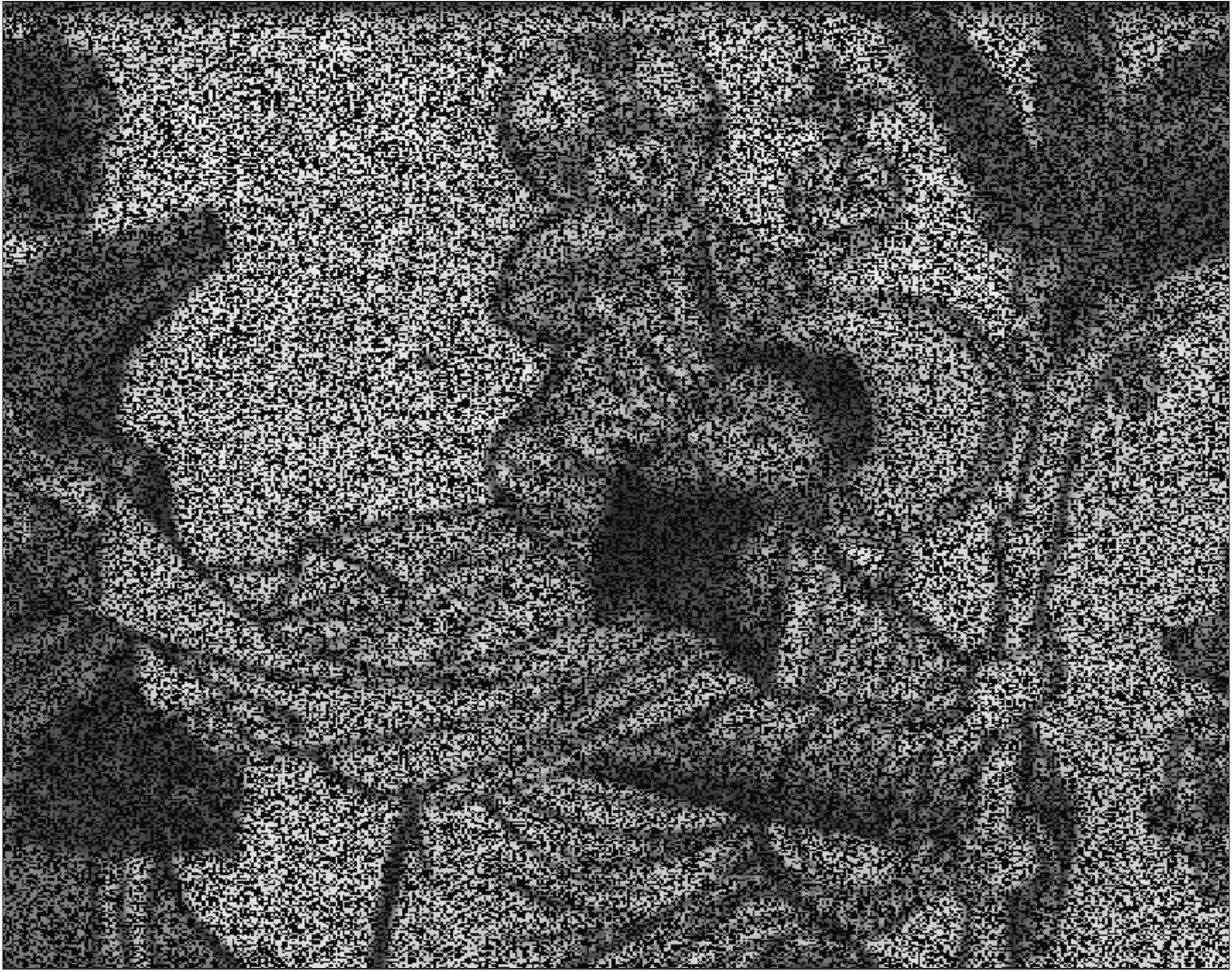} & \includegraphics[trim={0 0 0 0},clip,width=0.15\textwidth]{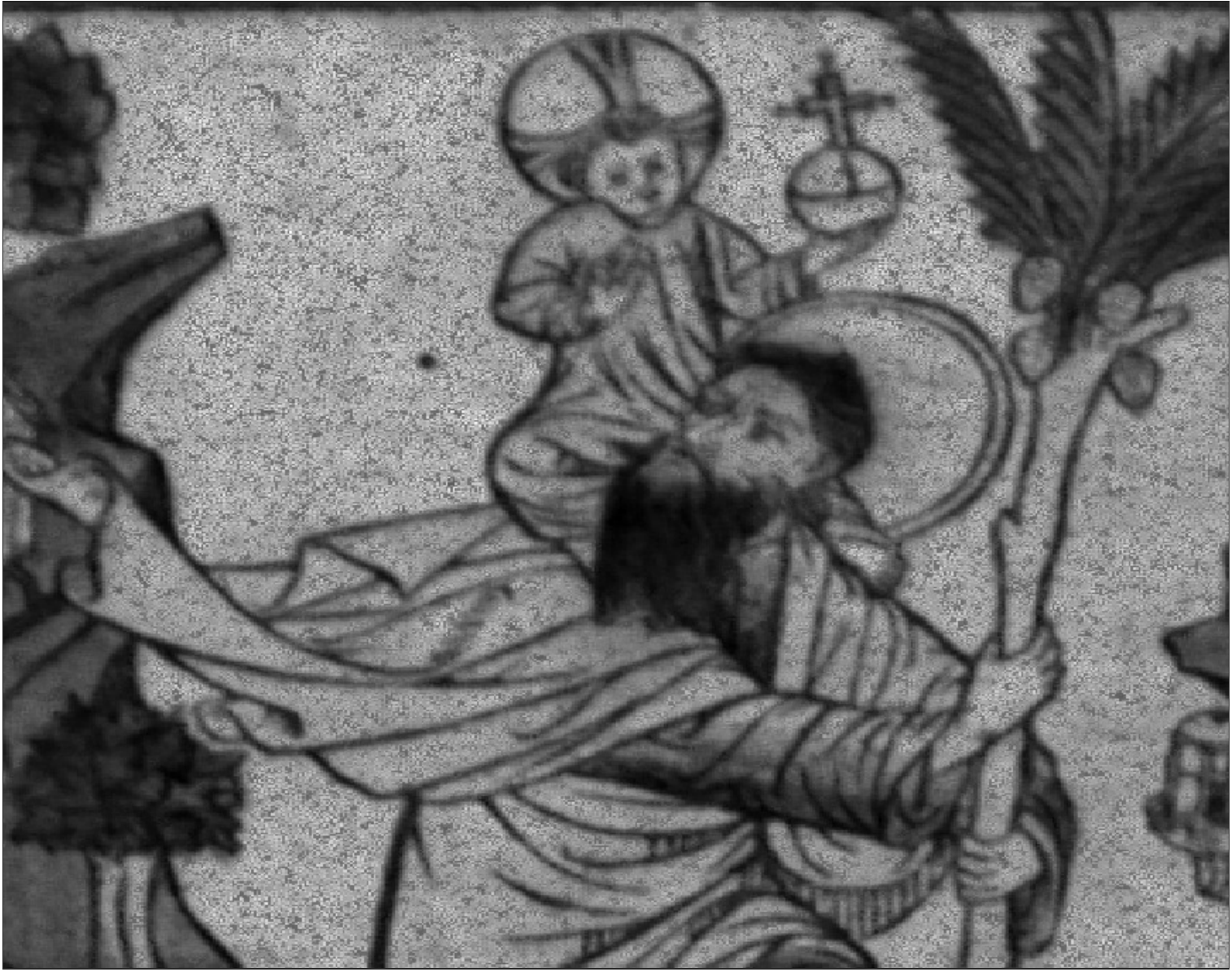} & \includegraphics[trim={0 0 0 0},clip,width=0.15\textwidth]{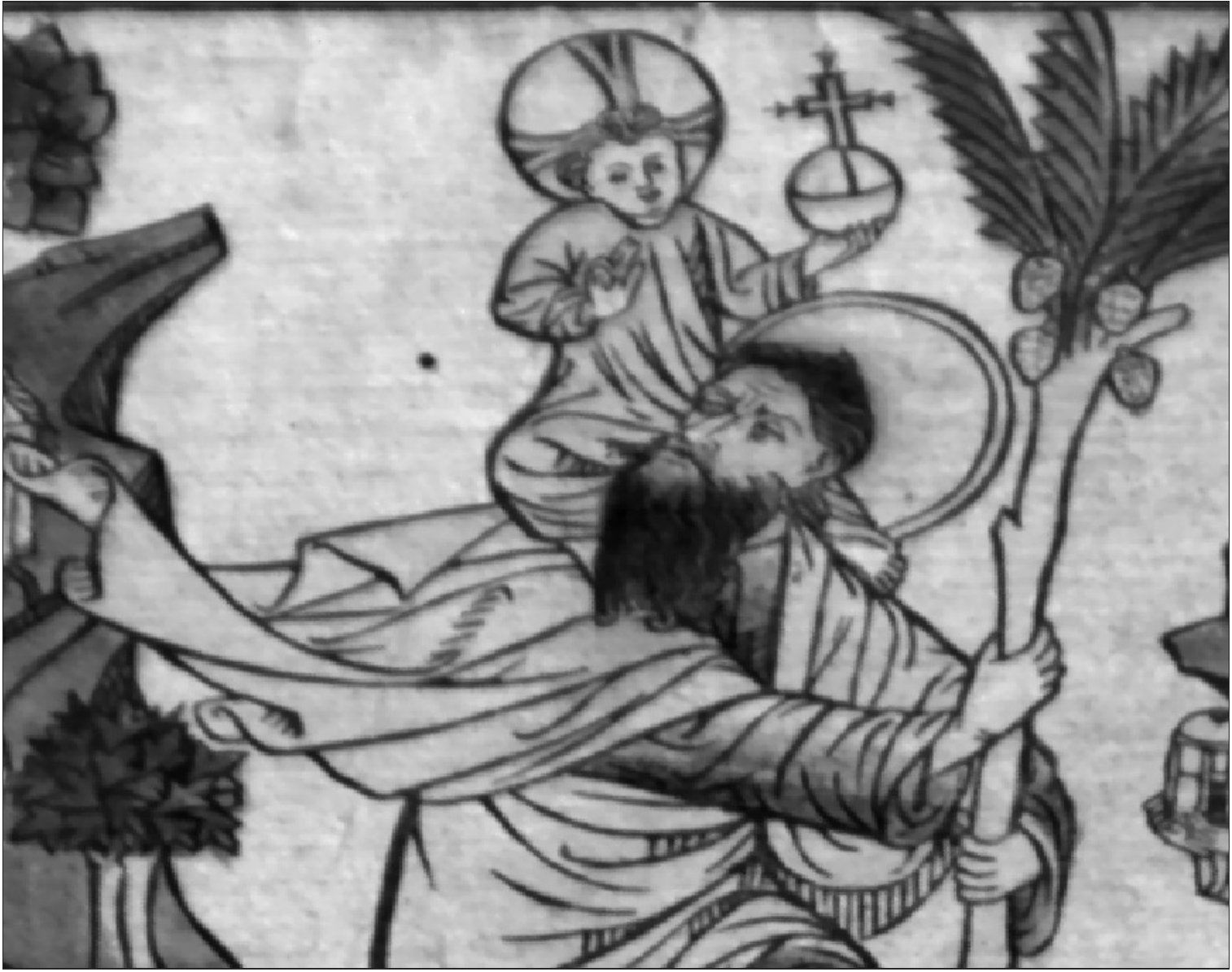} 
         
        \end{tabular}
    \end{center}
    \caption{Blurring and inpainting $\ell_{*}$-NLTV for the St-Christopher engraving hyperspectral image. Missing pixels $50\%$, dim(PSF) $=5$, $\sigma$(PSF) $= 0.9$, $\sigma$(noise) $= 0.01$. On the left, objective function (normalized with initialization value) vs CPU time ($\times 10^4$ sec). On the right, band $15$ of the HSI for FISTA and IML FISTA after $2$ iterations and at the end of the computation time budget ($50$ iterations of FISTA). \vspace{-1.5em}} 
    \label{fig:NLTV_inpainting_blur_HSI}
\end{figure}

\color{black}
\section{Conclusion and perspectives}

We have proposed a convergent multilevel framework for inertial and inexact proximal algorithms. In particular, this framework encompasses an inexact multilevel FISTA designed for image restoration and it is able to handle state-of-the-art regularizations in this context.  The proposed algorithm has theoretical convergence guarantees that are comparable to those obtained with leading algorithms in the field. From a practical point of view, our method shows very good performance on a wide range of degradation configurations and reaches good approximations of the optimal solution in a much smaller CPU time than FISTA, on large size problems. It also allows to deal with non differentiable functions whose proximity operator is not available under closed form, making the procedure applicable to a broad set of problems.

Among its many advantages, IML FISTA provides good quality reconstructions faster than FISTA. This opens up a great opportunity to deal with problems of large dimension, especially when limited computational resources prevent convergence from being achieved in a systematic way. In addition, this accelerated coarse approximation could play an important role in applications where image reconstruction is only a pre-processing task.
    
The main challenge for future work is to provide a better understanding of the link between the coarse and the fine level optimization. If the effect of coarse iterations and their impact on the different frequency components of the error is well studied for partial differential equations, much remains to be understood for proximal multilevel methods and specifically in the context of image restoration. For instance, it is yet not clear which is the best way of constructing lower level models, which deeply influences the performance of the method but depends on the problem at hand \cite{lauga2022_1}, or what are the conditions that make the coarse optimization useful for the general problem. One of the factors identified in this article is the nature and the intensity of the degradation : more degradation means better performance for IML FISTA compared to FISTA, while lower signal-to-noise ratio may worsen the results.  

\section*{Acknowledgments}
The authors would like to thank the GdR ISIS for the funding of the MOMIGS project and the ANR-19-CE48-0009 Multisc’In project. We also gratefully acknowledge the support of the Centre Blaise Pascal’s IT test platform at ENS de Lyon (Lyon, France) for the computing facilities.  The platform operates the SIDUS [1] solution developed by Emmanuel Quemener. 

\vspace{-1.5em}

\bibliographystyle{siamplain}
\bibliography{strings}

\end{document}